\documentclass[11pt]{article}
\usepackage[letterpaper]{geometry}
\usepackage{amsthm,amsmath,amsfonts}
\usepackage[normalem]{ulem}
\usepackage{color,graphicx}
\usepackage{float}
\usepackage[colorlinks]{hyperref}

\DeclareMathOperator{\argmin}{argmin}
\usepackage[numbers]{natbib}
\usepackage{centernot}
\usepackage{enumerate}
\newcommand{\R}{\mathbb{R}}

\newcommand{\Z}{\mathbb{Z}}
\newcommand{\I}{\mathbb{I}}
\newcommand{\E}{\mathbb{E}}

\newcommand{\X}{\mathcal{X}}
\newcommand{\M}{\mathcal{M}}
\newcommand{\A}{\mathcal{A}}
\newcommand{\F}{\mathcal{F}}
\newcommand{\V}{\mathcal{V}}
\newcommand{\B}{\mathcal{B}}

\newcommand{\LL}{\mathcal{L}}
\newcommand{\J}{\mathcal{J}}

\usepackage{graphicx}
\graphicspath{{images/}{../images/}}
\providecommand{\abs}[1]{\lvert#1\rvert}
\providecommand{\argmin}{\text{argmin}}

\newtheorem{lemma}{Lemma}
\newtheorem{theorem}{Theorem}

\newtheorem{remark}{Remark}
\theoremstyle{definition}
\numberwithin{equation}{section}

 \usepackage{thmtools, thm-restate}

\usepackage{tikz}
 \usetikzlibrary{arrows}

\usepackage[font=small, labelfont=bf]{caption}
\usepackage{subfig}
\usepackage{float}
\usepackage[ruled,vlined,linesnumbered]{algorithm2e}
\usepackage{xcolor}
\usepackage{authblk}
\providecommand{\keywords}[1]
{
  \small	
  \textbf{\textit{Keywords:}} #1
}

\usepackage[utf8]{inputenc}
\usepackage{quoting}
\usepackage{graphicx}
\usepackage{tikz}
\usetikzlibrary{positioning}
\usetikzlibrary{calc,fit}
\usepgflibrary{shapes.geometric}
\definecolor{processblue}{cmyk}{0.96,0,0,0}
\tikzset{
  server/.style={circle, inner sep=0.0mm, minimum width=.8cm,
    draw=green,fill=green!10,thick},
  ellipseserver/.style={ellipse, inner sep=0.0mm, minimum width=.8cm,
    draw=green,fill=green!10,thick},
  buffer/.style={rectangle, rounded corners=3pt,
    inner sep=0.0mm, minimum width=.9cm, minimum height=.6cm,
    draw=orange,fill=blue!10,thick},
  vbuffer/.style={rectangle,rounded corners=3pt,
     inner sep=0.0mm,  minimum width=.6cm, minimum height=.8cm,
     draw=orange,fill=blue!10,thick},
   routing/.style={circle,inner sep=0pt,minimum size=.5cm,
     draw=red,minimum width=.1cm, fill=red!30},
   state/.style={inner sep=1.0mm, rounded corners=2pt,
    draw=green,fill=green!10,thick},
  dot/.style={circle, inner sep=0.0mm, minimum width=.6cm,
    draw=blue,fill=blue!10,thick},
  task/.style={circle, inner sep=0.0mm, minimum width=.6cm,
    draw=blue,fill=blue!10,thick},
  data/.style={rectangle, inner sep=0.0mm, minimum width=.5cm,
    minimum height=.4cm,
    draw=red,fill=red!10,thick}
}

\tikzset{
  pics/openrectangle/.style n args={3}{
    code = { %
      \pgfmathsetmacro\x{.5}
      \pgfmathsetmacro\y{4}
      \pgfmathsetmacro\z{.55}
      \coordinate (SN) at #1;
      \coordinate (SS) at ($(0,-.3)+#1$);
      \draw[thick,red] ([shift={(-\x,\y)}]SN)--([shift={(-\x, \z)}]SN)
      --([shift={(\x,\z)}]SN)--([shift={(\x,\y)}]SN);

      \foreach \i/\t in #2
      {
        \node[task] (\t) at ([shift={(0,\i)}]SN)  {\t};
      }

      \foreach \i/\d in #3
      {
        \node[data] (\d) at ([shift={(0,-\i)}]SS)  {C\d};
      }
    }
  }
}

\tikzset{
  pics/rack/.style n args={3}{
    code = { %
      \pgfmathsetmacro\x{1.7}
      \pgfmathsetmacro\y{4}
      \pgfmathsetmacro\z{-2}
      \pgfmathsetmacro\zz{-2.5}
      \coordinate (S) at ($#1!.5!#2$);
      \draw[thick,blue] ([shift={(-\x,\y)}]S)--([shift={(-\x, \z)}]S)
      --([shift={(\x,\z)}]S)--([shift={(\x,\y)}]S);
      \node at ([shift={(0,\zz)}]S) {Rack #3};
    }
  }
}

\begin{document}

\title{Queueing Network Controls via Deep Reinforcement Learning}
\date{\vspace{-5ex}}

\author[1,2]{J. G. Dai}
\author[,3]{Mark Gluzman}
\affil[1]{ School of Data Science,  Shenzhen Research Institute of Big Data, The Chinese
University of Hong Kong, Shenzhen}
\affil[2]{School of Operations Research and Information Engineering, Cornell University, Ithaca, NY}
\affil[3]{Center for Applied Mathematics, Cornell University, Ithaca, NY}

\maketitle
\abstract{

  Novel advanced policy gradient (APG) methods, such as Trust Region
  policy optimization and Proximal policy optimization (PPO), have
  become the dominant reinforcement learning algorithms because of  their
  ease of implementation and good practical performance. A
  conventional setup for notoriously difficult queueing network
  control problems is a Markov decision problem (MDP) that has three
  features: infinite state space, unbounded costs, and long-run
  average cost objective.  We extend the theoretical framework of
  these APG methods for such MDP problems.  The resulting PPO
  algorithm is tested  on a parallel-server system and large-size
  multiclass queueing networks. The algorithm consistently
  generates control policies that outperform state-of-art heuristics
  in literature in a variety of load conditions from light to heavy
  traffic. These policies are demonstrated to be near-optimal when the optimal
  policy can be computed.

  A key to the successes of our PPO algorithm is   the use of three  variance
  reduction techniques in estimating the relative value
  function via sampling. First, we use a discounted relative value function as an
  approximation of the relative value function.  Second, we propose
  regenerative simulation to estimate the discounted relative value
  function. Finally, we incorporate the approximating martingale-process
  method into the regenerative estimator.

}

 \keywords{Multiclass Queueing Network, Control Variate, Reinforcement Learning, Long-Run Average Cost }
\section{Introduction}

For more than 30 years, one of the most difficult problems in applied
probability and operations research is to find a scalable algorithm
for approximately solving the optimal control of stochastic processing
networks, particularly when they are heavily loaded. These optimal
control problems have many important applications including healthcare
\cite{Dai2019a} and communications networks \cite{SrikYing2014, Luong2019},
data centers \cite{McKeown1999,Maguluri2012}, and manufacturing systems
\cite{PerkKuma1989, Kumar1993}. Stochastic processing networks are a broad class of
models that were advanced in \cite{Harr2000} and \cite{Harr2002} and
recently recapitulated and extended in \cite{DaiHarr2020}.

In this paper, we demonstrate that a class of \emph{deep reinforcement
  learning} algorithms known as proximal policy optimization (PPO),
 generalized from \cite{Schulman2015, Schulman2017} to our
  setting, can generate control policies that consistently beat the
performance of all state-of-arts control policies known in the
literature. The superior performances of our control policies appear
to be robust as stochastic processing networks and their load
conditions vary, with little or no problem-specific configurations of
the algorithms.

Multiclass queueing networks (MQNs) are a special class of stochastic
processing networks.  They were introduced in \cite{Harr1988} and have
 been studied intensively  for more than 30 years for performance
analysis and controls; see, for example, \cite{Harrison1993,
  KumaKuma1994, BertPascTsit1994, Bramson1998, Willams1998,
  Bertsimas2011, Bertsimas2015, ChenMeyn1999, Henderson2003,
  Veatch2015}.  Our paper focuses primarily on MQNs  with long-run average  cost objectives for two reasons.
First, these stochastic control problems  are
notoriously difficult due to the size of the state space, particularly
in heavy traffic. Second,  a large body of research has motivated the development of  various algorithms and control policies that are based on
either heavy traffic asymptotic analysis or heuristics. See, for
example, fluid policies \cite{ChenYao1993}, BIGSTEP policies
\cite{Harr1996}, affine shift policies \cite{Meyn1997},
discrete-review policies \cite{Maglaras2000, AtaKuma2005}, tracking
policies \cite{Bauerle2001}, and ``robust fluid''
policies~\cite{Bertsimas2015} in the former group and
\cite{LuRamaKuma1994,KumaKuma1996} in the latter group.  We
demonstrate that our algorithms outperform the state-of-art algorithms
in \cite{Bertsimas2015}. In this paper, we will also consider an
$N$-model that belongs to the family of parallel-server systems,
another special class of stochastic processing networks.  Unlike an MQN
in which routing probabilities of jobs are fixed, a  parallel-server
system allows dynamic routing of jobs in order to achieve
load-balancing among service stations.  We demonstrate that our
algorithms achieves near optimal performance in this setting, again
with little special configuration of  them.

For the queueing networks with Poisson arrival and exponential service
time distribution, the control problems can be modeled within the
framework of Markov decision processes (MDPs) \cite{Puterman2005} via
uniformization.  A typical algorithm for solving an MDP is via \emph{policy
  iteration} or \emph{value iteration}.  However, in our setting, the corresponding MDP suffers
from the usual curse of dimensionality: there are a large number of
job classes, and the buffer capacity for each class is assumed to be
infinite. Even with a truncation of the buffer capacity either as a
mathematical convenience or a practical control technique, the
resulting state space is huge when the network is large and heavily
loaded.

A \emph{reinforcement learning (RL) problem} often refers to a
(\emph{model-free}) MDP problem in which the underlying model
governing the dynamics is not known, but sequences of data (actions,
states, and rewards), known as episodes in this paper, can be observed
under a given policy. It has been demonstrated that \emph{RL
  algorithms} for solving RL problems can successfully overcome the
curse of dimensionality in both the model-free and model-based MDP
problems. Two factors are the keys to the success in a model-based
MDP. First,  Monte Carlo sampling method  is used   to approximately evaluate
expectations. The sampling
method also naturally supports   the exploration  needed in RL algorithms.
Second,  a parametric, low-dimensional representation of a
value function and/or a policy can be used.  In recent years, RL algorithms that
use neural networks as an architecture for value function and policy
parametrizations have shown state-of-art results \cite{Bellemare2013, Mnih2015,
  Silver2017, OpenAI2019a}.

In recent years the Proximal Policy Optimization (PPO)
algorithm \cite{Schulman2017} has become a default algorithm \cite{Openai2017} for control optimization  in  new  challenging environments including robotics \cite{Akkaya2019}, multiplayer video games \cite{Vinyals2019, OpenAI2019a}, neural network architecture engineering \cite{Zoph2018}, molecular design \cite{Simm2020}.
In this paper we extend the PPO algorithm to MDP problems with \emph{unbounded cost} functions and
\emph{long-run average cost} objectives. The original PPO algorithm
\cite{Schulman2017} was proposed for problems with \emph{bounded cost}
function and \emph{infinite-horizon discounted} objective. It was
based on the trust region policy optimization (TRPO) algorithm developed
in \cite{Schulman2015}. We use two separate feedforward neural
networks, one for parametrization of the control policy and the other
for value function approximation, a setup common to \emph{actor-critic}
algorithms \cite{Mnih2015}.  We propose an approximating martingale-process
(AMP) method for variance reduction to estimate policy value functions
and show that the AMP estimation  speeds up convergence of the PPO algorithm in the model-based
setting.  We provide a set of instructions for implementing the PPO
algorithm specifically for MQNs. The instructions include the choice of initial stable
randomized policy, methods for improving   policy value function estimation,  architecture of the value and policy neural
networks, and the choice of hyperparameters.  The proposed instructions can be potentially adapted to other advanced policy optimization RL algorithms, e.g. TRPO. Given the success of the PPO in various domains and its ease of use   we  focus on the PPO algorithm in this paper to illustrate efficiency of the deep RL framework for queueing control optimization.

The actor-critic methods can be considered as a combination of value-based and policy based methods.  In a
\textit{value-based}  approximate dynamic programming (ADP) algorithm, we may assume a low-dimensional
approximation of the optimal value function (e.g.  the optimal value function is a linear combination of known
\textit{features} \cite{DeFarias2003a, Ramirez-Hernandez2007,
  Abbasi_Yadkori2014}).   Value-based ADP algorithm has dominated
in the stochastic control of queueing networks literature;
see, for example, \cite{ChenMeyn1999, MoalKumaVanR2008,
  Chenetal2009, Veatch2015}.  These algorithms however have not achieved
robust empirical success for a wide class of control problems. It is
now known that the optimal value function might have a complex
structure which is hard to decompose on features, especially if
decisions have effect over a long horizon \cite{Lehnert2018}.

In a \textit{policy-based} ADP algorithm, we aim to learn the
optimal policy directly. Policy gradient algorithms are used to optimize the objective value within a
parameterized family of policies via gradient descent; see, for example,
\cite{Marbach2001, Paschalidis2004, Peters2008}. Although they are particularly effective for problems with high-dimensional action
space, it might be difficult to  reliably estimate the gradient of
the value under the current policy.  A direct sample-based estimation
typically suffers from high variance in gradient estimation \cite[Section 3]{Peters2008}, \cite[Section 5]{Marbach2001}, \cite[Section 1.1.2]{Baxter2001}.
Thus, actor-critic methods have been proposed \cite{Konda2003} to estimate the value function and use it
as a baseline and bootstrap for gradient direction approximation. The
actor-critic method with Boltzmann parametrization of policies and
linear approximation of the value functions has been applied for
parallel-server system control in \cite{Bhatnagar2012}.
 The standard policy gradient methods typically perform one gradient update per data sample which yields poor data
efficiency, and robustness, and an attempt to use a finite batch of samples to estimate the gradient and perform multiple steps of optimization ``empirically ... leads to destructively large policy
updates'' \cite{Schulman2017}. In \cite{Schulman2017}, the authors also note that the deep Q-learning algorithm \cite{Mnih2015} ``fails on many simple problems''.

 In \cite{Schulman2015,Schulman2017}, the authors propose
  \textit{``advanced policy gradient''} methods to overcome the
  aforementioned problems by designing novel objective functions that constrain
the magnitude of policy updates to avoid performance collapse caused by
large changes in the policy.  In \cite{Schulman2015} the authors prove
that minimizing a certain surrogate objective function guarantees
decreasing the expected discounted cost. Unfortunately,  their
theoretically justified step-sizes of policy updates cannot be
computed from available information for the RL algorithm.  Trust
Region Policy Optimization (TRPO) \cite{Schulman2015} has been
proposed as a practical method to search for step-sizes of policy
updates, and  Proximal Policy Optimization (PPO) method \cite{Schulman2017} has been proposed to compute
these step-sizes based on a clipped, ``proximal'' objective function.

We summarize the major contributions of our study:
\begin{enumerate}
\item In Section \ref{sec:countable} we theoretically justify that the advanced policy gradient algorithms can
  be applied for long-run average cost MDP problems with countable state spaces and unbounded
  cost-to-go functions. We show that starting from a stable policy it
  is possible to improve long-run average performance with
  sufficiently small changes to the initial policy.

\item In Section \ref{sec:AMP} we discuss a new way to estimate relative value
  and advantage functions if transition probabilities are known. We
  adopt the approximating martingale-process method
  \cite{Henderson2002} which, to the best of our knowledge, has not
  been used in simulation-based approximate
    policy improvement setting.
  \item  In Section \ref{sec:ge} we introduce a biased estimator
      of the relative value function through discounting the future
      costs. We interpret the discounting as the modification to the
      transition dynamics that shortens the regenerative cycles.
      We propose a regenerative estimator of the   discounted relative value function.

 The discounting  combined with the AMP method and regenerative simulation significantly reduces the variance of the relative value function estimation at the cost of a tolerable bias. The use of the proposed variance reduction techniques  speeds up the learning process of the PPO algorithm that we demonstrate by  computational experiments in Section
  \ref{sec:cc}.

\item In Section   \ref{sec:experiments} we conduct extensive computational experiments for multiclass queueing networks and parallel
  servers systems.  We propose to choose architectures of neural
  networks automatically as the size of a queueing network varies. We
  demonstrate the effectiveness of these choices as well as other
  hyperparameter choices such as  the  learning rate used in gradient
  decent. We demonstrate that the performance of control policies resulting
  from the proposed PPO algorithm outperforms other heuristics.

   \end{enumerate}

% \delete{In recent years \textit{deep} reinforcement learning  algorithms that use neural networks as an architecture for value function approximation and policy parametrization have shown tremendous successes.   Decision strategies obtained from deep RL algorithms beat human in one-player games \cite{Bellemare2013, Mnih2015}, two-players games \cite{Silver2017} and  team games \cite{OpenAI2019a}.  Following the trend, many dynamic resource allocation and sequential decision making problems in communications and networking have been solved applying deep RL  methods, but mostly value-based, see the review \cite{CongLuong}.}

% We mostly compare performance of our RL policies with the performance of robust fluid policies reported in \cite{Bertsimas2015}.  Robust fluid policies yield performance that is near-optimal for small-size networks, and have better performance for moderate and large-size networks in comparison with  the best other heuristic policies.

\section{Control of multiclass queueing networks}\label{sec:MQN}

In this section we formulate the   control problems
for multiclass processing networks. We first give a   formulation for the criss-cross network, which serves as an example, and then give a
formulation for a general multiclass queueing network.

   The set of real numbers is denoted by $\R$. The set of nonnegative
   integers is denoted by $\Z_+$. For a vector $a$ and a matrix $A$,
   $a^T$ and $A^T$ denote their transposes.

\subsection{The criss-cross network}
The criss-cross network has been studied in
\cite{Harrison1990},  \cite{Avram1995}, and
  \cite{Martins1996} among others. The network  depicted in Figure
  \ref{fig:cc}, which is taken from Figure~1.2 in \cite{DaiHarr2020}, consists of two stations that process three classes
of jobs. Each job class has its own dedicated buffer where jobs wait
to be processed. All buffers are assumed to have an infinite capacity.

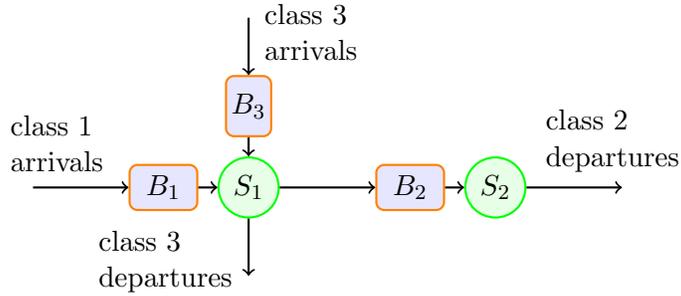
\begin{figure}[H]
  \centering
\begin{tikzpicture}  [scale=1.0,inner sep=2.0mm]
%\draw[step=.5cm, help lines] (-2,0) grid (21,10);
\node (T0) at (14,4) {};
\node[server] (S1) [left=1in of T0] {$S_1$};
\node[buffer] (B1) [left=.1in of S1] {$B_1$};
\node[buffer] (B2) [right=.5in of S1] {$B_2$};
\node[server] (S2) [right=.1in of B2] {$S_2$};
\node[vbuffer] (B3) [above= .1in of S1] {$B_3$};
\node (source1) [left= .5in of B1] {};
\node (sink2) [right= .5in of S2] {};
\node (sink3) [below= .3in of S1] {};
\node (source2) [above= .3in of B3] {};

\draw[thick,->] (B1) -- (S1);
\draw[thick,->] (B3) -- (S1);
\draw[thick,->] (S1) -- (B2);
\draw[thick,->] (B2) -- (S2);

\draw[thick,->] (source1) --(B1) node [above, align=left, near start]{class 1\\ arrivals};
\draw[thick,->] (source2) --(B3) node [right, align=left, near start]{class 3\\ arrivals};
\draw[thick,->] (S2) --(sink2) node [above, align=left, near end]{\mbox{\ \ } class 2 \\  \mbox{\ \ }  departures};
\draw[thick,->] (S1) --(sink3) node [left, align=left,  near end]{\ class 3 \\ \  departures};
 \end{tikzpicture}
  \caption{The criss-cross network}
  \label{fig:cc}
\end{figure}

We assume that the jobs of class 1 and class 3 arrive to the
system following Poisson processes with rates $\lambda_1$ and
 $\lambda_3$, respectively. Server 1 processes both classes one job at a time. After being processed class 1 jobs become  class 2 jobs and wait in buffer $2$ for server 2 to process. Class 2 and class 3
jobs   leave the system after their processings
are completed.  We assume that the processing times for class j jobs are i.i.d.,  having exponential
distribution with mean $m_j$, $j=1, 2, 3$. We denote $\mu_j := 1/m_j$ as
the service rate of class $j$ jobs. We assume that the following load conditions are satisfied:
\begin{align}\label{eq:load_cc}
\lambda_1m_1+\lambda_3m_3<1  ~~\text{and}~~ \lambda_1m_2<1.
\end{align}
Again, we assume each server processes one job at a time. Therefore,
processor sharing among multiple jobs is not allowed for each
server. Jobs within a buffer are processed in the first-in--first-out
order.  A service policy dictates the order in which jobs from
different classes are processed. (See below for a precise definition of a stationary Markov policy.) We assume a decision
time occurs when a new job arrives  or a service is
completed.  For concreteness, we adopt a preemptive service policy --- suppose the service policy dictates that server $1$
  processes a class $3$ job next while it is in the middle of processing
  a class $1$ job, the server preempts the unfinished class $1$ job to start the processing of the leading class $3$ job from buffer $3$. Due to the memoryless property of  an exponential distribution, it does not matter whether the preempted job keeps its remaining processing time or is assigned a new service time sampled from the original exponential distribution.

  Under any service policy,
at each decision time, the system manager needs to simultaneously choose  action $a_1$ from set $\{0, 1, 3\}$ for server $1$ and  action $a_2$ from the set $\{0, 2\}$ for server 2; for server $1$ to choose action $j$, $j=1, 3$, means that   server 1 processes a class $j$ job next (if buffer $j$ is non-empty), and to choose action $0$ means   that server 1 idles. Similarly,
for server $2$ to choose action $2$ means that server $2$ processes a class $2$ job, and to choose action $0$ means  server $2$ idles.  Each server is allowed to choose  action $0$ even if there are waiting jobs at the associated buffers.
Therefore, our service policies are not necessarily non-idling. We define the action set as $\A =\left\{ (a_1, a_2)\in  \{0, 1, 3\}\times \{0, 2\}\right\}$.

The service policy studied in this paper is \textit{randomized}.  By
randomized, we mean each server takes a random action sampled from  a
certain distribution on the action set.  For a set $A$, we use
$\mathcal{P}(A)$ to denote the set of probability distributions on
$A$. Therefore, for a pair $(p_1,p_2)\in \mathcal{P}(\{0,1,3\})\times \mathcal{P}(\{0,2\})$, server $1$ takes a random action sampled from distribution $p_1$ and server $2$ takes  a random action sampled from distribution $p_2$.  For notational convenience, we note that a pair has a one-to-one correspondence to a vector $u$ in the following set
\begin{align*}
  \mathcal{U}=\left\{u= \left(u_1, u_2, u_3\right)\in \R^3_+: u_1+u_3\le 1 \text{ and } u_2\le 1\right\},
\end{align*}
where $p_1 = (1-u_1-u_3, u_1, u_3)\in \mathcal{P}(\{0, 1, 3\})$ is a
probability distribution on the action set $\{0, 1, 3\}$ and
$p_2 = (1-u_2, u_2)\in \mathcal{P}(\{0, 2\})$ is a probability
distribution on the action set $\{0,2\}$. Throughout this subsection, we use
$u\in \mathcal{U}$ to denote  pair $(p_1, p_2)$.

To define a randomized stationary Markovian service policy, let $x_j(t)$ be the number of class $j$ jobs (including possibly one job in service) in the system  at
time $t$, $j=1,2, 3$.  Then $x(t) = \left(x_1(t), x_2(t), x_3(t)\right )$ is
  the vector of jobcounts at time $t$. Clearly, $x(t)\in \Z_+^3$. By convention, we assume the sample path of $\{x(t), t\ge 0\}$ is right continuous, which implies that when $t$ is a decision time (triggered by an external arrival  or a service completion), $x(t)$ has taken into account the  arriving job or completed job at time $t$.

By a randomized stationary Markovian service policy we denote a map
\begin{align*}
  \pi: \Z_+^3 \to \mathcal{U}.
\end{align*}
Given this map $\pi$, at each decision time $t$, the system manager
observes jobcounts $x=x(t)$, computes $\pi(x)\in \mathcal{U}$ and the
corresponding pair
$(p_1(x), p_2(x))\in \mathcal{P}(\{0, 1, 3\})\times \mathcal{P}(\{0,
2\})$. Then server $1$ chooses a random action sampled from $p_1(x)$ and
server $2$ chooses a random action sampled from $p_2(x)$. When each
distribution is concentrated on a single action, the corresponding policy is
a \emph{deterministic} stationary Markovian service policy. Hereafter, we use the term  stationary Markovian service policies to mean  randomized
policies, which include deterministic service policies as special cases.

Under a stationary Markovian policy $\pi$, $\{x(t), t\ge 0 \}$ is a
continuous time Markov chain (CTMC). Hereafter, we call  jobcount
vector $x(t)$ the state at time $t$, and we denote the state space as
$\X = \Z^3_+$.  The objective of our optimal control problem is to find
a stationary Markovian policy that minimizes the long-run average
number of jobs in the network:
\begin{equation}\label{co}
\inf_{\pi} \lim_{T\rightarrow \infty} \frac{1}{T} {\E}_\pi \int_0^T\Big(x_1(t)+x_2(t)+x_3(t)\Big)dt.
\end{equation}
 Because the interarrival and service times are exponentially
  distributed, the optimal control problem (\ref{co}) fits the
  framework of semi-Markov decision process (SMDP). See, for example,
  \cite[Chapter 11]{Puterman2005}.  Indeed, one can easily verify that
   at each state $x$, taking action $a$, the distribution of the
  time interval until next decision time is exponential with rate
  $\beta(x,a)$ to be specified below. We use $P(y|x,a)$ to denote the
  transition probabilities of the embedded Markov decision process,
  where $y\in \X$ is a state at the next decision time. For any
  state $x\in \X$ and any action $a\in \A$, the following transition
  probabilities always hold
\begin{align}
   P\big((x_1+1, x_2, x_3)|x, a\big)=\frac{\lambda_1}{\beta(x,a)}, \quad
   P\big((x_1, x_2, x_3+1)|x,a \big)=\frac{\lambda_3}{\beta(x,a)}. \label{eq:arrivals}
\end{align}
In the following, we specify $\beta(x,a)$ for each $(x,a)\in \X\times \A$ and additional
transition probabilities.  For action $a=(1,2)$ and state
$x=(x_1, x_2, x_2)$ with $x_1\ge 1$ and $x_2\ge 1$,
$\beta(x,a)=\lambda_1+\lambda_3+\mu_1+\mu_2$,
\begin{align*}
  &  P\big((x_1-1, x_2+1, x_3)|x,a \big)=\frac{\mu_1}{\beta(x,a)}, \quad
    P\big((x_1, x_2-1, x_3)|x,a\big)=\frac{\mu_2}{\beta(x,a)};
\end{align*}
for action $a=(3,2)$ and state $x=(x_1, x_2, x_2)$ with $x_3\ge 1$ and
$x_2\ge 1$, $\beta(x,a)=\lambda_1+\lambda_3+\mu_3+\mu_2$,
\begin{align*}
    P\big((x_1, x_2, x_3-1)|x,a \big)=\frac{\mu_3}{\beta(x,a)}, \quad
    P\big((x_1, x_2-1, x_3)|x,a\big)=\frac{\mu_2}{\beta(x,a)};
\end{align*}
for action $a=(0,2)$ and state $x=(x_1,x_2, x_3)$ with $x_2\ge 1$, $\beta(x,a)=\lambda_1+\lambda_3+\mu_2$,
\begin{align*}
    P\big((x_1, x_2-1, x_3)|x,a\big)=\frac{\mu_2}{\beta(x,a)};
\end{align*}
for action $a=(1,0)$ and state $x=(x_1, x_2, x_3)$ with $x_1\ge 1$, $\beta(x,a)=\lambda_1+\lambda_3+\mu_1$,
\begin{align*}
   P\big((x_1-1, x_2+1, x_3)|x,a\big)=\frac{\mu_1}{\beta(x,a)};
\end{align*}
for action $a=(3,0)$ and state $x=(x_1, x_2, x_3)$ with $x_3\ge 1$, $\beta(x,a)=\lambda_1+\lambda_3+\mu_3$,
\begin{align*}
  P\big((x_1, x_2, x_3-1)|x,a\big)=\frac{\mu_3}{\beta(x,a)};
\end{align*}
for action $a=(0,0)$ and state $x=(x_1, x_2, x_3)$,
$\beta(x,a)=\lambda_1+\lambda_3$. Also, $x_2=0$ implies that
$a_2=0$, $x_1=0$ implies that $a_1\neq 1$,  and  $x_3=0$ implies that $a_1\neq 3$.

Because the time between state transitions are exponentially
distributed, we   adopt the method of uniformization for solving the SMDP;
see, for example,  \cite{Serfozo1979} and \cite[Chapter 11]{Puterman2005}.
We denote
\begin{align}
  \label{eq:B}
B=\lambda_1+\lambda_3+\mu_1+\mu_2+\mu_3.
\end{align}
 For the new control
problem under uniformization, the decision times are determined by the
arrival times of a Poisson process with (uniform) rate $B$ that is independent of the underlying state. Given current state $x\in \X$
and action $a\in \A$, new transition probabilities into $y\in \X$ are given by
\begin{align}\label{eq:unif}
\tilde P(y|x, a) =
\begin{cases}
P(y|x, a) \beta(x, a) / B \quad \text{ if } x\neq y,\\
1 - \beta(x, a) / B  \quad \text{otherwise.}
\end{cases}
\end{align}
The transition probabilities $\tilde P$ in (\ref{eq:unif}) will define
a (discrete time) MDP. The  objective is given by
\begin{align}\label{eq:co1}
\inf\limits_{\pi} \lim\limits_{N\rightarrow \infty}\frac{1}{N}\E_{\pi}\left[ \sum\limits_{k=0}^{N-1}\left(x_1^{(k)}+x_2^{(k)}+x_3^{(k)}\right)\right],
\end{align}
where $\pi$ belongs to the family of stationary Markov policies,
and $x^{(k)} = \left(x_1^{(k)}, x_2^{(k)}, x_3^{(k)}\right)$ is the state (vector of jobcounts)
at the time of the $k$th decision (in the uniformized framework).
Under a stationary Markov policy $\pi$, $\{x^{(k)}:k=0, 1, 2, \ldots\}$ is a discrete time Markov chain (DTMC).

The existence of a stationary Markovian policy $\pi^*$ that minimizes
(\ref{eq:co1}) follows from \cite[Theorem 4.3]{Meyn1997} if the load conditions (\ref{eq:load_cc}) are satisfied. Under a mild
condition on $\pi^*$, which can be shown to be satisfied following the
argument in \cite[Theorem 4.3]{Meyn1997}, the policy $\pi^*$ is an
optimal Markovian stationary policy for (\ref{co}) \cite[Theorem
2.1]{Beutler1987}.  Moreover, under policy $\pi^*$, the objective in
(\ref{co}) is equal to  that in (\ref{eq:co1}); see
\cite[Theorem 3.6]{Beutler1987}.

\subsection{General formulation of  a  multiclass queueing network control problem}

We consider a multiclass queueing network that has $L$ stations  and $ J$  job classes.   For
notational convenience we denote $\LL = \{1, ..., L\}$ as  the set
  of stations, $\J = \{1, ...,J \}$ as  the  set of
job classes.
Each station has a single server that
processes jobs from the job classes that belong   to the station.  Each job class belongs to one  station. We use
$\ell= s(j)\in \LL$ to denote the station that class $j$ belongs to.
We assume the function $s: \J \to \LL$ satisfies $s(\J)=\LL$.
Jobs arrive externally  and are processed sequentially at various stations, moving from one class to  the next after each processing step until they exit the network.
Upon arrival if a class $j$ job finds the  associated server busy, the job waits in the corresponding buffer $j$.  We assume that every buffer has an infinite capacity.
For each station $\ell\in \LL$, we define
  \begin{align}\label{def:servers}
 \mathcal{B}(\ell) := \{j\in \J:~ s(j)  = \ell\}
 \end{align} as the set of job classes to be processed by server $\ell$.

Class $j\in \J$ jobs arrive externally to buffer $j$ following a Poisson
 process with rate $\lambda_j$; when $\lambda_j=0$, there are no external arrivals into buffer $j$.
   Class $j$ jobs are processed
 by server $s(j)$ following the service policy as specified below.
 We assume the service times for class $j$ jobs are i.i.d.  having exponential distribution with  mean $1/\mu_j$. Class $j$
 job, after being processed by server $s(j)$, becomes class $k\in \J$ job with probability $r_{jk}$ and leaves the network with probability
  $1 - \sum\limits_{k=1}^K r_{jk}$. We define $J\times J$ matrix $R := (r_{jk})_{j,k=1,..., J}$ as the routing matrix. We assume that the network is open, meaning that $I-R$ is invertible.
We let vector $q= (q_1, q_2, ..., q_J)^T$ satisfy the system of linear
equations
\begin{equation}\label{eq:traffic}
  q = \lambda +R^T q.
\end{equation}
Equation (\ref{eq:traffic}) is known as the traffic equation, and it has a unique solution under the open network assumption.
  For each class $j \in \J$,  $q_j$ is interpreted to be the total arrival rate
  into buffer $j$, considering both the  external arrivals
  and internal arrivals from service completions at stations.
We define \textit{the load} $\rho_\ell$ of station $\ell\in\LL$ as
\begin{equation*}\rho_\ell := \sum\limits_{j\in \B(\ell)} \frac{q_j}{\mu_j}.\end{equation*}
We assume that
\begin{align}\label{eq:load}
\rho_\ell<1 \quad \text{for each  station }\ell\in \LL.
\end{align}

We let $x(t) = \left( x_1(t), .., x_J(t) \right)$ be the vector of jobcounts at time $t$. A decision time occurs when a new job arrives at the system or a service is completed.   Under any service policy, at each decision time, the system manager needs to simultaneously choose an action for each server $\ell\in \LL$.
For each server $\ell\in \LL$ the system manager selects an action from set $\B(\ell)\cup \{0\}$: action $j\in \B(\ell)$ means that  the system manager gives priority to  job class $j$ at station $\ell$; action $0$ means server $\ell$ idles until the next decision time.

We define  set
\begin{align}\label{eq:setP}
  \mathcal{U}=\left\{u= \left(u_1, u_2,..., u_J\right)\in \R^J_+:     \sum\limits_{j\in \B(\ell)} u_j\le1 \text{ for each }\ell\in \LL\right\}.
\end{align}

For each station $\ell\in \LL$ vector $u\in \mathcal{U}$ defines a probability distribution $u_\ell$ on the action set  $\B(\ell)\cup \{0\}$: probability of action $j$ is equal to $u_j$ for $j\in \B(\ell)$, and  probability of action $0$ is equal to $\Big(1 - \sum\limits_{j\in \B(\ell)} u_j\Big)$.
We define a randomized stationary Markovian service policy as a map from a set of  jobcount vectors into set $\mathcal{U}$ defined in (\ref{eq:setP}):
\begin{align*}
\pi:\Z_+^J\rightarrow \mathcal{U}.\end{align*}

Given this map $\pi$, at each decision time $t$, the system manager
observes jobcounts $x(t)$, chooses $\pi(x)\in\mathcal{U}$, and based on $\pi(x)$ computes probability distribution $u_\ell$ for each $\ell\in \LL$.
Then  the system manager  independently samples  one action from $u_\ell$  for each server $\ell\in \LL$.

 The objective is to find a stationary Markovian policy  $\pi$ that   minimizes the  long-run average number of jobs in the network:
\begin{align}\label{eq:obj3}
\inf\limits_{\pi} \lim\limits_{T\rightarrow \infty} \frac{1}{T} \E_\pi \int\limits_0^T\left(\sum\limits_{j=1}^J x_j(t)\right)dt.
\end{align}

Under a stationary Markovian policy $\pi$, we adopt the method of uniformization to obtain a uniformized discrete time Markov chain (DTMC) $\{x^{(k)}:k=0, 1, \ldots \}$. We abuse the notation and denote a system state as $x^{(k)} = \left( x_1^{( k)}, x_2^{(k)}, ..., x_J^{(k)} \right)$ after $k$ transitions of the DTMC.

In this paper, we   develop  algorithms to approximately solve the following (discrete-time) MDP problem:
% The objective (\ref{eq:obj3}) is equivalent to
\begin{align}\label{eq:obj4}
 \inf\limits_\pi \lim\limits_{N\rightarrow \infty} \frac{1}{N} \E_{\pi}\left[ \sum\limits_{k=0}^{N-1}\sum\limits_{j=1}^J x_j^{(k)} \right].
\end{align}
\begin{remark}

It has been proved in \cite{Meyn1997} that the MDP (\ref{eq:obj4}) has an optimal policy
that satisfies the conditions in \cite[Theorem 3.6]{Beutler1987} if the ``fluid limit model'' under
some policy is $L_2$-stable.  Under the load condition (\ref{eq:load}),
conditions in \cite{Meyn1997} can be verified as follows. First, we adopt the randomized
version of the head-of-line static processor sharing (HLSPS) as defined in
\cite[Section 4.6]{DaiHarr2020}. We apply this  randomized  policy to the
discrete-time MDP to obtain the resulting DTMC.  The fluid limit path
of this DTMC can be shown to satisfy the fluid model defined in  \cite[Definition 8.17]{DaiHarr2020} following a procedure that is similar to,
but much simpler than, the proof of \cite[Theorem 12.24]{DaiHarr2020}. Finally, \cite[Theorem 8.18]{DaiHarr2020} shows the fluid model is stable, which is
stronger than the $L_2$-stability needed.

\end{remark}

\section{Reinforcement learning approach for queueing network control}\label{sec:countable}

%	The undeniable success of RL has been reached for games \cite{Mnih2015}, \cite{Bellemare2013}, and physics engines \cite{Duan2016},  \cite{Todorov2012}. These problems are episodic in nature and have long but finite horizon. In the game domain each new episode (game) may require a different amount of time for the player to finish. To compare the effect of different strategies  the objective is
%	normalized by discounting future costs.

 Originally, policy gradient algorithms have been developed  to find
optimal policies which optimize the finite horizon total cost or
infinite horizon discounted total cost objectives. For
  stochastic processing networks and their applications, it is often
   useful  to optimize the long-run average cost. In this section
we develop a version of the  Proximal Policy Optimization algorithm.  See Section 5, which demonstrates the effectiveness of our proposed PPO for finding the  near optimal control policies for stochastic processing networks.

\subsection{Positive recurrence and $\V$-uniform ergodicity}\label{sec:MC}
As discussed in Section~\ref{sec:MQN}, operating under a fixed
randomized stationary control policy, the dynamics of a stochastic
processing network is a DTMC. We restrict policies so that the
resulting  DTMCs are  irreducible and aperiodic. Such a DTMC does not
always have a stationary distribution.  When the DTMC does not have a
stationary distribution, the performance of the control policy is
necessarily poor, leading to an infinite long-run average
cost. It is well known that an irreducible DTMC has a unique stationary distribution if and only if it is positive recurrent.
Hereafter, when the DTMC is positive recurrent, we call
the corresponding control policy \emph{stable}. Otherwise, we call it
\emph{unstable.}  %Clearly, a good RL algorithm should keep each policy
%in the iteration \emph{stable}.

A sufficient condition for an irreducible DTMC to be positive recurrent is
the Foster-Lyapunov drift condition.  The drift condition
(\ref{eq:drift}) in the following lemma is stronger than the classical
Foster-Lyapunov drift condition.
For a proof of the lemma, see  Theorem 11.3.4 and Theorem 14.3.7 in \cite{Meyn2009}.

\begin{lemma}\label{lem:drift}
Consider an irreducible Markov chain on a countable state space $\X$ with a transition matrix $P$ on $\X\times \X$.  Assume   there exists a vector $\V:\X\rightarrow [1, \infty)$ such that the following drift condition holds for some constants  $b\in (0,1)$ and $d\geq0$, and a finite subset $C\subset \X$:
\begin{align}\label{eq:drift}
\sum\limits_{y\in \X}P(y|x)\V(y)\leq b \V(x) +d\I_{C}(x), \quad \text{for each }x\in \X,
\end{align}
where $\I_{C}(x)=1$ if $x\in C$ and $\I_{C}(x)=0$ otherwise. Here,
$P(y|x)=P(x,y)$ is the transition probability from state $x\in \X$ to
state $y\in \X$. Then
(a)  the Markov chain with the transition matrix  $P$ is positive recurrent with a unique stationary distribution  $\mu$;  and  (b)
$\mu^T\V <\infty,
$
where  for any function $f: \X \rightarrow \R$  we define $\mu^T f$ as
\begin{align*}
\mu^Tf:=\sum\limits_{x\in \X} \mu(x)f(x).
\end{align*}

\end{lemma}

Vector $\V$ in the drift condition (\ref{eq:drift}) is called a
\textit{Lyapunov function} for the Markov chain.
For any matrix $M$ on $\X\times \X$, its $\V$-norm is defined to be
 \begin{align*}
   \|M\|_\V^{ } = \sup\limits_{x\in \X} \frac{1}{\V(x)}\sum\limits_{y\in \X}|M(x, y)| \V(y).
\end{align*}

An irreducible, aperiodic Markov chain with transition matrix $P$ is
called \textit{$\V$-uniformly ergodic} if
\begin{align*}\|P^n - \Pi\|_\V^{ }\rightarrow 0 \text{ as } n\rightarrow
  \infty,\end{align*} where every row of $ \Pi$  is equal  to the
stationary distribution $\mu$, i.e. $\Pi(x, y): = \mu(y),$ for any
$x, y\in \X$.  The drift condition (\ref{eq:drift}) is sufficient and
necessary for an irreducible, aperiodic Markov chain to be
$\V$-uniformly ergodic \cite[Theorem 16.0.1]{Meyn2009}.  For an
irreducible, aperiodic Markov chain that satisfies (\ref{eq:drift}),
for any $g:\X\to \R$ with $\abs{g(x)}\le \V(x)$ for $x\in \X$, there
exist constants $R<\infty$ and $r<1$ such that
\begin{align}\label{eq:geo}
\left| \sum\limits_{y\in \X} P^n(y|x) g(y) - \mu^T g  \right| \leq R\V(x) r^n
\end{align}
for any $x\in \X$ and $n\geq 0$; see   \cite[Theorem 15.4.1]{Meyn2009}.

\subsection{Poisson equation}\label{sec:PO}
For an irreducible DTMC on state space $\X$ (possibly infinite) with transition matrix $P$, we
assume that there exists a Lyapunov function $\V$ satisfying (\ref{eq:drift}).
For any cost function $g:\X \rightarrow \R$ satisfying $\abs{g(x)}\le \V(x)$ for each $x\in \X$, it follows from Lemma~\ref{lem:drift} that $\mu^T |g|<\infty$.
Lemma~\ref{lem:poisson_sol} below asserts that the following equation
has a solution $h:\X\to\R$:
\begin{align}\label{eq:Poisson}
g(x) - \mu^T g + \sum\limits_{y\in \X}P(y|x) h(y) - h(x) =0 \quad \text{ for each }x\in \X.
\end{align}
Equation (\ref{eq:Poisson}) is called  a \textit{Poisson equation} of the Markov chain with transition matrix $P$ and cost function $g $. Function $h:\X\rightarrow \R$ that satisfies (\ref{eq:Poisson}) is called a \textit{solution} to the Poisson equation.  The solution is unique up to a constant shift, namely,
if $h_1$ and $h_2$ are two solutions to Poisson equation (\ref{eq:Poisson}) with $\mu^T(|h_1|+|h_2|)<\infty$, then there exists a constant $b\in \R$ such that $h_1(x) = h_2(x) +b$ for each $x\in \X$, see \cite[Proposition 17.4.1]{Meyn2009}.

A solution $h$ to the Poisson equation is called a \textit{fundamental solution} if   $ \mu^T h =0.$    The proof of the following lemma is provided in \cite[Proposition A.3.11]{Meyn2007}.

\begin{lemma}\label{lem:poisson_sol}

Consider a $\V$-uniformly ergodic  Markov chain with transition matrix $P$ and the stationary  distribution $\mu$.
 For any cost function $g:\X \rightarrow \R$ satisfying $|g|\leq \V$,  Poisson equation (\ref{eq:Poisson}) admits a fundamental  solution
\begin{align}\label{eq:h}
h^{(f)}(x) : = \E \left[\sum\limits_{k=0}^\infty \left(g(x^{(k)}) - \mu^Tg\right)~|~x^{(0)} = x\right] ~\text{for each }x\in \X,
\end{align}
where $x^{(k)}$ is the state of the Markov chain after $k$ timesteps.

\end{lemma}

We  define   \textit{fundamental matrix} $Z$ of the Markov chain with transition kernel $P$ as %that maps any cost function $|g|\leq \V$ into a corresponding fundamental solution $h^{(f)}$ as
 \begin{align}\label{eq:Zs}Z :=\sum\limits_{k=0}^\infty \left( P-\Pi \right)^k. \end{align}
 It follows from  \cite[Theorem 16.1.2]{Meyn2009} that the series (\ref{eq:Zs}) converges in $\V$-norm and, moreover,  $\|Z\|_\V<\infty$.
Then, it is easy to see that fundamental matrix $Z$ is an inverse matrix of $(I-P+\Pi)$, i.e. $Z(I-P+\Pi) = (I-P+\Pi)Z = I$.
See Appendix Section \ref{sec:proofs} for the proof of the following Lemma \ref{lem:Zeq}.
 \begin{lemma}\label{lem:Zeq}
Fundamental matrix $Z$   maps any cost function $|g|\leq \V$ into a corresponding fundamental solution $h^{(f)}$ defined by (\ref{eq:h}):
 \begin{align}h^{(f)} = Z\left(g -(\mu^T g) e\right) ,\end{align}
 where $e = (1, 1,.., 1, ...)^T$ is a unit vector.
 \end{lemma}

 \begin{remark}
  Consider matrices $A, B,C$ on $\X\times \X$.
The   associativity property,
\begin{align*}
 ABC =(AB)C= A(BC),
 \end{align*}
does not always hold for matrices defined on  a  countable state space; see a counterexample  in \cite[Section 1.1]{Kemeny1976}. However, if  $\|A\|_\V<\infty, \|B\|_\V<\infty, \|C\|_\V<\infty$ then matrices $A, B, C$ associate, see \cite[Lemma 2.1]{Jiang2017}.
Hence,  there is no ambiguity in the definition of the fundamental matrix (\ref{eq:Zs}):
\begin{align*}
\left( P-\Pi \right)^k =\left( P-\Pi \right) \left( P-\Pi \right)^{k-1}=\left( P-\Pi \right)^{k-1}\left( P-\Pi \right),~~\text{for }k\geq 1
 \end{align*}
where $\|P-\Pi\|_\V<\infty$ holds due to the drift condition (\ref{eq:drift}).
\end{remark}

 \subsection{Improvement guarantee for average cost objective}\label{sec:TRPOforAC}

 We consider an  MDP problem with a countable state space $\X$, finite action space $\A$, one-step cost function $g(x)$, and  transition function $P(\cdot|x, a)$. For each  state-action pair $(x, a)$, we assume  that the chain can transit to a finite number of distinguished states, i.e. set $\{y\in \X: P(y| x, a)>0\}$ is finite for each $(x, a)\in \X\times \A.$

 We consider  $\Theta\subset \R^d$ for some integer $d>0$ and
  $\Theta$ is open.  With every $\theta\in \Theta,$ we associate a
randomized Markovian policy $\pi_{\theta}$, which at any state
$x\in \X$ chooses action $a\in \A$ with probability
$\pi_\theta(a|x)$. Under   policy $\pi_{\theta}$, the corresponding DTMC has transition matrix $P_{\theta}$ given by
% The corresponding family of transition kernels is
% given by
% \begin{align*}\left\{P_\theta(\cdot|x) :~P_\theta(\cdot|x) = \sum\limits_{a\in A} \pi(a|x)P(\cdot|x, a),~ \theta\in \Theta\right\}.\end{align*}
\begin{align*} P_\theta(x, y) = \sum\limits_{a\in \A} \pi_\theta(a|x)P(y|x, a) \text{ for } x, y\in \X.
\end{align*}
For each $\theta\in \Theta$ we assume that the resulting Markov chain with transition probabilities $P_{\theta}$ is irreducible and aperiodic.

We also assume there exists $\eta \in \Theta$ such that the drift
condition (\ref{eq:drift}) is satisfied for the transition matrix
$P_\eta$ with a Lyapunov function $\V:\X\rightarrow [1, \infty).$ By
Lemma \ref{lem:poisson_sol} the corresponding fundamental matrix
$Z_\eta$ is well-defined. The following lemma says that if $P_{\eta}$ is positive
recurrent and $P_\theta$ is ``close'' to $P_\eta$, then $P_\theta$ is
also positive recurrent.
See Appendix Section \ref{sec:proofs} for the proof.

%When  $\theta\in \Theta$ is sufficiently close to $\eta$, we can define an operator
%\begin{align*}
%H_{\theta, \eta}: = (I- [P_{\theta} - P_{\eta}] Z_{\eta})^{-1}.
%\end{align*}
%A sufficient condition to ensure that the inverse is well-define is $\|  [P_{\theta} - P_{\eta}] Z_{\eta}\|_\V < 1$. Then $H_{\theta, \eta}$ exists in the sense that  $H_{\theta, \eta}: = \sum\limits_{k=0}^\infty [P_{\theta} - P_{\eta}]^k Z^k_{\eta}.$
%
%
%Existence of operator $H_{\theta, \eta}$  guarantees that the Markov chain $P_{\theta}$ is positive recurrent and allows to find its stationary distribution.

\begin{lemma}\label{lem:st}
  Fix an $\eta \in \Theta$.
  We assume  that drift condition (\ref{eq:drift}) holds for $P_\eta$. Let some $\theta\in \Theta$ satisfies,
\begin{align*}
\| (P_{\theta} - P_{\eta}) Z_{\eta}\|_\V^{ } < 1 , \end{align*}
then the Markov chain with transition matrix $P_\theta$ has a unique stationary distribution $\mu_{\theta}$.
\end{lemma}

Assume the assumption of Lemma~\ref{lem:st} holds.
For any cost function $|g|\le \V$, we denote
the corresponding fundamental solution to the Poisson equation as
$h_\eta$ and the long-run average cost
\begin{align}\label{eq:ac}
 \mu_\eta^Tg  := \sum\limits_{x\in \X} \mu_\eta(x) g(x).
\end{align}

 The following theorem provides a bound on the difference of long-run average performance of policies $\pi_\theta$ and $\pi_{\eta}.$ See  Appendix Section \ref{sec:proofs} for the proof.

\begin{theorem}
\label{thm:main}
Suppose that the Markov chain with transition matrix $P_{\eta}$ is  an irreducible  chain such that the drift condition (\ref{eq:drift}) holds for some function $\V\geq1$ and the cost function satisfies $|g|<\V$.

 For any $\theta\in \Theta$ such that
\begin{align}\label{eq:D}
 D_{\theta,\eta} : = \|  (P_{\theta} - P_{\eta}) Z_{\eta}\|_\V^{ } < 1
\end{align}
the difference of long-run average costs of policies $\pi_\theta$ and $\pi_{\eta}$ is bounded by:
   \begin{align}\label{eq:ineq2}
\mu_\theta^Tg- \mu_\eta^Tg ~\leq~ &N_1(\theta, \eta)+  N_2(\theta, \eta),
\end{align}
where $N_1(\theta, \eta)$, $N_2(\theta, \eta)$ are finite and equal to
\begin{align}
N_1(\theta, \eta) &:= \mu_{\eta}^T( g -(\mu_\eta^Tg) e   +P_{\theta}h_{\eta} - h_{\eta} ), \label{eq:M1} \\
N_2(\theta, \eta) &:= \frac{ D_{\theta,\eta}^2}{1- D_{\theta, \eta}}  \left\| g - (\mu_\eta^Tg) e  \right\|^{ }_{\infty, \V}(\mu_{\eta}^T\V), \label{eq:M2}
\end{align}
where,  for a vector $\nu$ on $\X$, $\V$-norm is defined as
\begin{align}\label{eq:Vnorm}
 \|\nu\|_{\infty, \V} := \sup\limits_{x\in \X} \frac{| \nu(x)|}{\V(x)}.
\end{align}

\end{theorem}

It follows from Theorem \ref{thm:main} that the negativity of the right side of inequality (\ref{eq:ineq2}) guarantees   that policy $\pi_\theta$  yields an improved performance compared  with the initial policy $\pi_{\eta}.$
Since
\begin{align*}
 \min\limits_{\theta \in \Theta: ~D_{\theta, \eta}<1}[N_1(\theta, \eta) +N_2(\theta, \eta)] \leq N_1(\eta, \eta) +N_2(\eta, \eta) = 0,
\end{align*}
we want to find  $\theta=\theta^*$:
\begin{align}
  \label{eq:argmin}
  \theta^* =\argmin \limits_{\theta \in \Theta: ~D_{\theta, \eta}<1}[N_1(\theta, \eta) +N_2(\theta, \eta)]
\end{align}
 to achieve the maximum improvement in the upper bound  (\ref{eq:ineq2}).    In the setting
of finite horizon and infinite discounted RL problems,
\cite{Kakade2002, Schulman2015} propose to  fix the maximum change
between policies $\pi_\theta$ and $\pi_{\eta}$ by bounding  the
$N_2(\theta, \eta)$ term and  to minimize  $N_1(\theta, \eta)$. Below, we discuss the motivation for developing the PPO algorithm proposed in Section \ref{sec:ppo},
leading to a practical algorithm to approximately solve optimization
(\ref{eq:argmin}).

First, we  observe that
\begin{align*}
|N_1(\theta, \eta)|:&= \left|\mu_{\eta}^T( g -(\mu_\eta^Tg) e   +P_{\theta}h_{\eta} - h_{\eta} ) \right|\\
 &\leq (\mu_{\eta}^T\V)\| g - (\mu_\eta^Tg) e  +P_{\theta}h_{\eta} - h_{\eta } \|^{ }_{\infty, \V}\\
& =  (\mu_{\eta}^T\V)\| (P_{\theta} - P_{\eta})h_{\eta} \|^{ }_{\infty, \V}\\
&= (\mu_{\eta}^T\V)\| (P_{\theta} - P_{\eta}) Z_{\eta}\left (g- (\mu_\eta^Tg) e\right)\|^{ }_{\infty, \V}\\
&\leq (\mu_{\eta}^T\V )\| g -(\mu_\eta^Tg) e \|_{\infty, \V}^{ } D_{\theta, \eta}.
\end{align*}
Therefore, when $\theta$ is close to $\eta$, we expect that $D_{\theta, \eta}$
in (\ref{eq:D}) is small and  $N_1(\theta, \eta) =O(D_{\theta, \eta})$.

From (\ref{eq:M2}), $N_2(\theta, \eta)$ is nonnegative and
$N_2(\theta, \eta) = O(D_{\theta, \eta}^2)$.  Therefore, when
$N_1({\theta, \eta})<0$ and $D_{\theta, \eta}$ is small enough,
$\pi_\theta$ is a strict improvement over $\pi_\eta$.   Lemma \ref{lem:policies}  shows
that the distance $D_{\theta, \eta}$ can be controlled by the
probability ratio
\begin{align}
  \label{eq:ratio}
 r_{\theta, \eta}(a|x): =\frac{\pi_\theta(a|x)}{ \pi_{\eta}(a|x)}
\end{align}
between the two policies.
\begin{lemma}\label{lem:policies}

Suppose that the Markov chain with transition matrix $P_{\eta}$ is  an irreducible  chain such that the drift condition (\ref{eq:drift}) holds for some function $\V\geq1$ and the cost function satisfies $|g|<\V$. Then for any $\theta\in \Theta$
   \begin{align*}
D_{\theta, \eta} \leq   \|Z_{\eta}\|^{ }_\V\sup\limits_{x\in \X}    \sum\limits_{a\in \A} \left| r_{\theta, \eta}(a|x)-  1 \right|    G_{\eta}(x, a),
\end{align*}

where $G_{\eta}(x, a): =   \frac{1}{\V(x)} \sum\limits_{y\in \X}  \pi_{\eta}(a|x) P(y|x, a) \V(y)$.
\end{lemma}

 Lemma \ref{lem:policies} implies that
$D_{\theta, \eta}$ is small when the ration  $r_{\theta, \eta}(a|x)$  in (\ref{eq:ratio}) is close to 1
for each state-action pair $(x, a)$. Note that
$r_{\theta, \eta}(a|x) = 1$ and $D_{\theta, \eta}=0$ when
$\theta=\eta.$ See Appendix Section \ref{sec:proofs} for the proof.

\subsection{Proximal Policy Optimization}\label{sec:ppo}

We rewrite the first term of the right-hand side of (\ref{eq:ineq2}) as:
\begin{align}\label{eq:MA}
N_1(\theta, \eta) & =  \mu_{\eta}^T(g -(\mu_\eta^Tg) e   +P_{\theta}h_{\eta} - h_{\eta} ) \nonumber \\
 & =  \underset{\substack{ x\sim \mu_{\eta}\\ a\sim \pi_{\theta}(\cdot|x) \\ y\sim P(\cdot|x, a)}  }{\E} \left[    g(x ) -(\mu_\eta^Tg) e   + h_{\eta}(y) - h_{\eta}(x)   \right]  \nonumber\\                  &
                    =  \underset{\substack{ x\sim \mu_{\eta}\\ a\sim \pi_{\theta}(\cdot|x)  }  }{\E} A_{\eta} (x, a)\nonumber\\
                  & =  \underset{\substack{ x\sim \mu_{\eta}\\ a\sim \pi_{\eta}(\cdot|x)  }  }{\E} \left[ \frac{\pi_{\theta}(a|x)}{\pi_{\eta}(a|x)}A_{ \eta} (x, a) \right]
 =  \underset{\substack{ x\sim \mu_{\eta}\\ a\sim \pi_{\eta}(\cdot|x)  }  }{\E}\Big[  r_{\theta, \eta}(a|x)A_{\eta} (x, a)\Big],
\end{align}
where we define an advantage function $A_{\eta}:\X\times \A\rightarrow \R$ of policy $\pi_\eta, $ $\eta\in \Theta$ as:
\begin{align}\label{eq:A}
	A_{\eta}(x, a): =   \underset{\substack{ y\sim P(\cdot|x, a)}  }{\E} \left[g(x) - \mu_\eta^T g + h_{\eta} (y) -h_{\eta} (x) \right].
\end{align}

Equation (\ref{eq:MA}) implies that if we want to minimize $N_1(\theta, \eta)$, then the ratio $r_{\theta, \eta}(a|x)$ should be minimized (w.r.t. $\theta$) when $A_\eta(x,a)>0$, and maximized when $A_\eta(x,a)<0$  for each $x\in \X$ .
The end of Section~\ref{sec:TRPOforAC} suggests
that we should strive to (a)
\begin{align}\label{eq:minM1}
\text{minimize } \quad N_1(\theta, \eta),
\end{align}
and (b) keep the ratio $r_{\theta, \eta}(a|x)$ in (\ref{eq:ratio})
close to 1. In \cite{Schulman2017} the authors propose to minimize
(w.r.t. $\theta\in \Theta$) the following clipped surrogate objective
\begin{align}\label{eq:PO}
L(\theta, \eta):=\underset{\substack{ x\sim \mu_{\eta}\\ a\sim \pi_{\eta}(\cdot|x)  }  }{\E}   \max \left[  r_{\theta, \eta}(a|x) A_{\eta} (x, a) ,  ~ \text{clip} (r_{\theta, \eta}(a|x),  1-\epsilon, 1+\epsilon)  A_{\eta} (x, a)  \right],
\end{align}
where  $\epsilon\in(0, 1)$ is a hyperparameter, and clipping function is defined as
\begin{align*}
  \text{clip}(c,  1-\epsilon, 1+\epsilon):= \begin{cases} \max(1-\epsilon, c),~\text{if } c\leq1,\\  \min(1+\epsilon, c),~\text{otherwise.}\end{cases}
\end{align*}
 In \cite{Schulman2017} the authors  coined the term,
proximal policy optimization (PPO), for their algorithm, and
demonstrated its ease of implementation and its ability to find good control
policies.

The objective term
$\text{clip} (r_{\theta, \eta}(a|x), 1-\epsilon, 1+\epsilon) A_{\eta}
(x, a) $ in (\ref{eq:PO}) prevents changes to the policy that
move $ r_{\theta, \eta}(a|x) $ far from 1. Then the objective function
(\ref{eq:PO}) is a upper bound (i.e. a pessimistic bound) on the
unclipped objective (\ref{eq:minM1}).  Thus, an improvement on the
objective (\ref{eq:PO}) translates to an improvement on
$N_1(\theta, \eta)$ only when $\theta\in \Theta$ satisfies
$r_{\theta, \eta}\in (1-\epsilon, 1+\epsilon)$. The alternative
heuristics proposed in \cite{Schulman2015, Wang2016, Schulman2017, Wu2017} to solve
optimization problem (\ref{eq:argmin}): each defines a loss function that controls $N_2(\theta, \eta)$ and
minimizes the $N_1(\theta, \eta)$ term.
Following \cite{Schulman2017}, we use loss function
(\ref{eq:PO})  because of its implementation simplicity.

To compute objective function in (\ref{eq:PO}) we first evaluate the expectation and precompute advantage functions in (\ref{eq:PO}). We assume that an approximation $\hat A_\eta:\X\times \A\rightarrow \R$ of the advantage function (\ref{eq:A}) is available and focus on estimating the objective from simulations.
See Section \ref{sec:M1} below for estimating $\hat A_\eta$.

 %We  use Monte-Carlo simulation method and replace the expectations in (\ref{eq:PO}) and (\ref{eq:A}) by sample averages.

% With available function approximation of $h_\eta$, the advantage function (\ref{eq:A}) can be estimated as
% \begin{align}\label{eq:Aes1}
% 	\hat A_{\eta}(x_k, a_k): =   g(x_k) - \widehat{\mu_\eta^T g} + \sum\limits_{y\in \X} P(y|x_k, a_k) f_{\psi} (y) -  f_{\psi}(x_k).
% \end{align}

 Given an episode with length $N$ generated  under policy $\pi_\eta$  we compute  the advantage function estimates $\hat A_{\eta}\left(x^{({k, q})}, a^{({k, q})}\right)$   at the observed state-action pairs:
\begin{align*}
D^{(0:N-1)}: = \left \{\left(  x^{(0)}, a^{(0)}, \hat A_{\eta}(x^{(0)}, a^{(0)})  \right),
\left (   x^{(1)}, a^{(1)}, \hat A_{\eta}(x^{(1)}, a^{(1)} )\right) ,\cdots,
\left(x^{({N-1})}, a^{({N-1})}, \hat A_{\eta}(x^{({N-1})}, a^{({N-1})}) \right)\right\},
   \end{align*}
   and estimate the loss function (\ref{eq:PO}) as a sample average over the state-action pairs from the episode:
 \begin{align}\label{eq:popt}
   \hat L\left(\theta, \eta,D^{(0:N-1)} \right) = \sum\limits_{ k=0}^{N-1} \max\left[ \frac{\pi_{\theta}\left( a^{({k})}| x^{({k})} \right)}{\pi_{\eta}\left( a^{({k})}|x^{({k})} \right)}  \hat A_{\eta}\left(x^{(k)}, a^{(k)}\right) ,\text{clip}\left(\frac{\pi_{\theta}\left(a^{(k)}| x^{(k)}\right)}{\pi_{\eta}\left(a^{(k)}|x^{(k)}\right)},     1-\epsilon, 1+\epsilon \right  ) \hat  A_{\eta}\left(x^{(k)}, a^{(k)}\right)  \right].
 \end{align}
 In theory, one long episode  under policy $\pi_\eta$  starting from any initial state $x^{(0)}$
 is sufficient because the following SLLN for Markov chains holds:
 with probability $1$,
 \begin{align*}
   \lim_{N\to\infty} \frac{1}{N}   \hat L\left(\theta, \eta,D^{(0:N-1)} \right)
   = L(\theta, \eta).
 \end{align*}

\section{Advantage function estimation}\label{sec:M1}

The computation of objective function (\ref{eq:PO}) relies on the availability of  an estimate of  advantage function
$A_{\eta} (x, a)$ in (\ref{eq:A}).  We assume our MDP model is
known. So the expectation on (\ref{eq:A}) can be computed
exactly, and we can perform the computation in a timely manner. In this section, we explain how to
estimate $h_\eta$, a solution to the Poisson equation
\eqref{eq:Poisson} with $P=P_\eta$ and
$\mu=\mu_\eta$.

%Algorithm~\ref{alg1} relies on the availability of advantage function
%$A_{\eta} (x, a)$ in (\ref{eq:A}).  We assume our MDP model is
%known. So the expectation on (\ref{eq:A}) can be computed
%exactly. Furthermore, we assume this expectation computation can be
%done in a time-efficient manner. Therefore, our focus is on how to
%estimate $h_\eta$, a solution to the Poisson equation
%\eqref{eq:Poisson} with $P=P_\eta$ and
%$\mu=\mu_\eta$. Lemma~\ref{lem:poisson_sol} provides a representation
%\eqref{eq:h} for $h_\eta$.

To compute expectation $A\left(x^{(k)}, a^{(k)}\right)$ in (\ref{eq:A}) for a given
state-action pair $\left(x^{(k)},a^{(k)}\right)$, we need to evaluate $h_\eta(y)$ for each $y$ that
is reachable from $x^{(k)}$. This requires one to estimate
$h_\eta(y)$ for some states $y$ that have not been visited in the simulation. Our strategy is to
use Monte Carlo method to estimate $h_\eta(y)$ at a  selected
subset of $y$'s, and then use an approximator $f_\psi(y)$ to
replace $h_\eta(y)$ for an arbitrary $y\in \X$. The latter is standard
in deep learning. Therefore, we focus on finding a good estimator
$
  \hat h(y)
$
for $h_\eta(y)$.

\subsection{Regenerative estimation}

 Lemma~\ref{lem:poisson_sol} provides a representation of the fundamental solution
\eqref{eq:h} for $h_\eta$. Unfortunately, the known unbiased Monte Carlo estimators of the fundamental solution rely on obtaining  samples from the stationary distribution of the Markov chain \cite[Section 5.1]{Cooper2003}.

We define the following solution to the Poisson equation (\ref{eq:Poisson}).
\begin{lemma}\label{lem:poisson_sol2}

Consider the $\V$-uniformly ergodic  Markov chain with transition matrix $P$ and the stationary  distribution $\mu$. Let $x^*\in X$ be an arbitrary state of the positive recurrent Markov chain.
 For any cost function $g:\X \rightarrow \R$ such that $|g|\leq \V$, the Poisson's equation (\ref{eq:Poisson}) admits  a  solution
\begin{align}\label{eq:h2}
h^{(x^*)}(x) : = \E \left[\sum\limits_{k=0}^{\sigma(x^*)-1} \left(g(x^{(k)}) - \mu^Tg\right)\Big|~x^{(0)} = x\right] ~\text{for each }x\in \X,
\end{align}
where $\sigma(x^*) = \min\left\{k>0~|~x^{(k)} = x^*\right\}$ is the first future time when state $x^*$ is visited.
 Furthermore, the solution has a finite $\V$-norm: $\|h^{(x^*)}\|_{\infty, \V}<\infty.$
\end{lemma}
See \cite[Proposition A.3.1]{Meyn2007} for the proof. Here, we refer to state $x^*$ as a \textit{regeneration state}, and to the times $\sigma(x^*)$ when the regeneration state is visited   as \textit{regeneration times}.

The value of  advantage function (\ref{eq:A}) does not depend on a particular choice of a solution of the Poisson equation (\ref{eq:Poisson}) since if $h_1$ and $h_2$ are two solutions  such that  $\mu^T(|h_1|+|h_2|)<\infty$, then there exists a constant $b\in \R$ such that $h_1(x) = h_2(x) +b$ for each $x\in \X$,  \cite[Proposition 17.4.1]{Meyn2009}. Therefore, we use representation (\ref{eq:h2})  for $h_\eta$ in computing (\ref{eq:A}).

We assume that an episode consisting  of $N$ regenerative cycles
\begin{align*}
  \left\{x^{(0)}, x^{(1)},\cdots, x^{(\sigma_1)},\cdots,   x^{(\sigma_{N}-1)}   \right \}
\end{align*}
has been generated  under policy $\pi_\eta$, where $x^{(0)}=x^*.$

We compute an estimate of the long-run average cost based on   $N$ regenerative cycles  as
\begin{align}\label{es_av}
\widehat{\mu^T_\eta  g}  : =\frac{1}{\sigma(N)}  \sum\limits_{k=0}^{\sigma(N)-1} g(x^{(k)}),
\end{align}
where $\sigma(n)$ is the $n$th time when  regeneration state $x^*$ is visited.
Next, we consider an arbitrary state $x^{(k)}$ from the generated episode.  We define a one-replication estimate of the solution  to the Poisson equation (\ref{eq:h2}) for a state $x^{(k)}$ visited at time $k$  as:
\begin{align}\label{eq:es1}
\hat h_k: =  \sum\limits_{t=k}^{\sigma_k-1} \left(g(x^{(t) }) - \widehat {\mu_\eta^Tg}  \right) ,
\end{align}
where $\sigma_k = \min\left\{t>k~|~x^{(t)}=x^*\right\}$  is the first time when the regeneration state $x^*$ is visited after time $k$.
We note that the one-replication estimate (\ref{eq:es1}) is computed for every timestep.
   The estimator (\ref{eq:es1}) was    proposed in  \cite[Section 5.3]{Cooper2003}.

We use   function $f_\psi:\X\rightarrow \R$ from  a family of function approximators $\{f_\psi, \psi\in \Psi \}$ to represent function $h_{  \eta}$ and choose function $f_\psi$  from $\{f_\psi, \psi\in \Psi \}$ to minimize the mean square distance to the one-replication estimates  $\{\hat  h_k\}_{k=0}^{\sigma(N)-1}:$
\begin{align}\label{eq:Vappr}
\psi^* = \arg\min\limits_{\psi \in \Psi} \sum\limits_{k=0}^{\sigma(N)-1} \left ( f_{\psi}(x^{(k)}) - \hat h_k   \right)^2.
\end{align}

With available function approximation $f_{\psi^*}$ for $h_\eta$,  we estimate the advantage function (\ref{eq:A})  as:
 \begin{align}\label{eq:Aes1}
 	\hat A_{\eta}(x^{(k)}, a^{(k)}): =   g(x^{(k)}) - \widehat{\mu_\eta^T g} + \sum\limits_{y\in \X} P\left(y|x^{(k)}, a^{(k)}\right) f_{\psi^*} (y) -  f_{\psi^*}(x^{(k)}).
 \end{align}

  We assume that $Q$ episodes, $Q\geq 1$, can be simulated in parallel, and each of $q=1, ..., Q$ (parallel) actors collect an episode
\begin{align}\label{eq:episodes}
\left\{x^{(0, q)}, a^{(0, q)}, x^{(1 , q)}, a^{(1, q)}, \cdots, x^{(k , q)}, a^{(k, q)},\cdots, x^{(\sigma^q(N)-1, q)}, a^{(\sigma^q(N)-1, q)}\right\}
\end{align}
with $N$ regenerative cycles, where $\sigma^q(N)$  is the $N$th regeneration time in the simulation of $q$th actor and $x^{(0, q)} = x^*$ for each $q=1, ..., Q$.
 Given  the episodes (\ref{eq:episodes})  generated under policy $\pi_\eta$,  we compute  the advantage function estimates $\hat A_{\eta}(x^{(k, q)}, a^{(k, q)})$   by (\ref{eq:Aes1}):
\begin{align*}
D^{ (0: \sigma^q(N)-1)_{q=1}^Q } = \Big\{\Big( x^{(0, q)}, a^{(0, q)},&  \hat A_{\eta}(x^{(0, q)}, a^{(0, q)}) \Big), \cdots, \\
 &\Big( x^{(\sigma^q(N)-1, q)}, a^{(\sigma^q(N)-1, q)}, \hat A_{\eta}(x^{(\sigma^q(N)-1, q)}, a^{(\sigma^q(N)-1, q)})  \Big)\Big\}_{q=1} ^{Q }.
\end{align*}
  We estimate the loss function (\ref{eq:PO})  as a sample average over these $\sum\limits_{q=1}^Q \sigma^q(N)$  data-points:
   \begin{align}\label{eq:popt}
    \hat L\left(\theta, \theta_i, D^{(1:Q), (0:\sigma^q(N)-1) }\right)  =\sum\limits_{q=1}^Q \sum\limits_{k=0}^{\sigma^q(N)-1} \max\Big[ &\frac{\pi_{\theta}(a^{(k, q)}| x^{(k, q)})}{\pi_{\theta_{i}}(a^{(k, q)}|x^{(k, q)})}  \hat A_{\theta_{i}}(x^{(k, q)}, a^{(k, q)}) ,\\ &\text{clip}\left(\frac{\pi_{\theta}(a^{(k,q)}| x^{(k, q)})}{\pi_{\theta_{i}}(a^{(k, q)}|x^{(k, q)})},     1-\epsilon, 1+\epsilon  \right) \hat A _{\theta_{i}}(x^{(k, q)}, a^{(k,q)})  \Big] \nonumber
   \end{align}
  Optimization of the loss function yields a new policy for the next iteration.

\begin{algorithm}[H]
\SetAlgoLined
\KwResult{policy $\pi_{\theta_I}$ }
 Initialize  policy $\pi_{\theta_0}$ \;
 \For{ policy iteration $i= 0, 1, ..., I-1$}{
  \For{ actor  $q= 1, 2, ..., Q$}{
  Run policy $\pi_{\theta_{i}}$  until it reaches $N$th regeneration time on $\sigma^q(N)$ step: collect an episode
   $\left\{x^{(0, q)}, a^{(0, q)}, x^{(1 , q)}, a^{(1, q)}, \cdots, x^{(\sigma^q(N)-1, q)}, a^{(\sigma^q(N)-1, q)}, x^{(\sigma^q(N), q)}\right\} $\;
  }

  Compute the average cost estimate $\widehat{\mu^T_{\theta_i}   g}$ by (\ref{es_av})\  (utilizing $Q$ episodes) ;

        Compute  $\hat h_{k, q}$,  the estimate of $h_{\theta_i}(x^{(k, q)})$, by (\ref{eq:es1})  for each $q = 1, .., Q$, $k=0, .., \sigma^q(N)-1$\;
        Update $\psi_{i}: = \psi$, where $\psi\in \Psi$ minimizes $  \sum\limits_{q=1}^Q \sum\limits_{k=0}^{\sigma^q(N)-1 } \left ( f_{\psi}(x^{(k, q)}) -\hat h_{k, q}   \right)^2$ following (\ref{eq:Vappr}) \;
  Estimate the advantage functions $\hat A_{\theta_{i}}\left(x^{(k, q)}, a^{(k, q)}\right)$ using (\ref{eq:Aes1}) for each $q = 1,..., Q$, $k=0,..., \sigma^q(N)-1$:
 \begin{align*}
D^{(1:Q), (0:\sigma^q(N)-1) } = \left\{\left( x^{(0, q)}, a^{(0, q)}, \hat A_{0,q}\right), \cdots,  \left( x^{(\sigma^q(N)-1, q)}, a^{(\sigma^q(N)-1, q)}, \hat A_{\sigma^q(N)-1,q}\right)\right\}_{q=1} ^{Q }.
\end{align*}\\
 Minimize the surrogate objective function w.r.t. $\theta\in \Theta$:
 \begin{align*}
    \hat L\left(\theta, \theta_i, D^{(1:Q), (0:\sigma^q(N)-1) }\right)  =\sum\limits_{q=1}^Q \sum\limits_{k=0}^{\sigma^q(N)-1} \max\Big[ &\frac{\pi_{\theta}(a^{(k, q)}| x^{(k, q)})}{\pi_{\theta_{i}}(a^{(k, q)}|x^{(k, q)})}  \hat A_{\theta_{i}}(x^{(k, q)}, a^{(k, q)}) ,\\ &\text{clip}\left(\frac{\pi_{\theta}(a^{(k,q)}| x^{(k, q)})}{\pi_{\theta_{i}}(a^{(k, q)}|x^{(k, q)})},     1-\epsilon, 1+\epsilon  \right) \hat A _{\theta_{i}}(x^{(k, q)}, a^{(k,q)})  \Big]
   \end{align*}\\
 Update $\theta_{i+1}: = \theta$.
 }
 \caption{Base proximal policy optimization algorithm  for long-run average cost problems}\label{alg1}
\end{algorithm}

%
%\begin{remark}
%Note that the advantage function $\hat A_{\theta_{i+1}}$ is computed using $h_{\psi_{i+1}}$ which is an approximation of the value function of policy $\pi_{\theta_i}$, not $\pi_{\theta_{i+1}}$.
%Additional bias might be introduced if the value function is updated
%first, see \cite[Section 6.1]{Schulman2016}. For example, if   $N=1$ and approximation in (\ref{eq:Vappr}) overfits the data, i.e.  $h_{\psi}(x_{k}) =\hat h_\eta(x_k)$ for each $x_k$,  then estimates of the advantage function   $\hat A_{\theta_{i+1}}$ become zero for all generated state-action pairs $(x_k, a_k)$.
%\end{remark}

% Let $x^*$ be a regeneration state. We define a   state-action relative value function $Q _{\eta} :\X\times \A\rightarrow \R$ of policy $\pi_\theta, $ $\theta\in \Theta$ as
%
%\begin{align*}
%Q(x, a) : = \E \left[\sum\limits_{k=0}^{\sigma(x^*)-1} \left(g(x_k) - (\mu^Tg)\right)\Big|x_0 = x\right], ~\text{for each }x\in \X
%\end{align*}
%where $x=x_0$ and action $a = a_0$ is taken at time $k=0$;  the following actions starting from time $k=1$ are taken according to policy $\pi_\theta$;  $\sigma(x^*) = \min\{k>0|x_k = x^*\}$ is the first future time when state $x^*$ is visited.
%
%We note that  $A_{\eta}(x, a) =Q_\eta(x, a)  - h_\eta(x )  $  for each $(x, a)\in \X\times \A$. We also note that the Poisson's solution function  $h_\eta(x ) $ does not depend on  actions and serves as a normalization \textit{baseline} for the state-action value functions \cite[Section 13.4]{Sutton2018}.

In practice a naive (standard) Monte Carlo estimator  (\ref{eq:es1}) fails to improve in PPO policy iteration
  Algorithm~\ref{alg1}  because of the large variance (i.e. the estimator is unreliable).  Therefore, we progressively develop a sequence of estimators in the next subsections. We end this section with   two remarks.

\begin{remark}
For any state $x\in \X$ the one-replication estimate (\ref{eq:es1}) is computed each time  the state is visited (the \textit{every-visit} Monte-Carlo method). It is also possible to implement  a \textit{first-visit}  Monte Carlo method which implies that
a one-replication estimate is computed when state $x$ is visited for the first time within a cycle
 and that the next visits to state $x$ within the same cycle are ignored. See \cite[Section 5.1]{Sutton2018} for more details of  every-visit and first-visit Monte-Carlo methods.
\end{remark}
\begin{remark}
  In regression problem (\ref{eq:Vappr}), each data point
    $(x^{(k)}, \hat h_k)$ is used  in the quadratic loss function, despite
    that many of the $x^{(k)}$'s represent the same state. It is possible to  restrict that
    \emph{only distinct} $x^{(k)}$'s are used in the loss function, with
    corresponding $\hat h_k$'s properly averaged. It turns out that this
    new optimization problem yields the same optimal solution as the one in (\ref{eq:Vappr}).
The equivalence of the optimization problems follows from the fact that for an arbitrary sequence of real numbers $a_1, ..,a_n\in \R$:
\begin{align*}%\label{eq:opt}
 \arg\min_{x\in B}\left[ \sum\limits_{i=1}^n (x-a_i)^2 \right]= \arg\min_{x\in B}\left[ \left( x-\frac{1}{n}\sum\limits_{i=1}^n a_i\right)^2\right],
\end{align*}
where $B$ is an arbitrary subset of $\R$.
 \end{remark}

\subsection{Approximating martingale-process method}\label{sec:AMP}

%The main goal of the advantage function $A_{\eta} (x, a)$ is to predict the effect of action $a$ realization at state $x$. If approximation  $f_\psi$ is not sufficiently close to a true solution of  Poisson's equation, the effect of the action on the future can be evaluated incorrectly by estimator (\ref{eq:Aes1}) that can be a cause of suboptimal policy update.

%One can estimate state-action function considering
%\begin{align}\label{eq:es2}
%\hat Q (x_k, a_k): =  \sum\limits_{t=0}^{\sigma_k(x^*)-1} \left(g(x_{k+t }) - \widehat {\mu^Tg}  \right),
%\end{align}
%where $\sigma_k(x^*)$ is the first regeneration time after time $k$.

Estimator (\ref{eq:es1}) of  the solution to the Poisson equation suffers from the high variance when the regenerative cycles are long  (i.e. the estimator is a sum of many random terms $g(x^{(k)}) - \mu_\eta^Tg$).  In this section we explain how to  decrease  the variance by reducing the magnitude of summands in (\ref{eq:es1}) if  an approximation $\zeta$  of the solution to Poisson's equation $ h_\eta$ is available.

We assume  an episode $\left\{x^{(0)}, a^{(1)}, x^{(1)}, a^{(2)}, \cdots, x^{(K-1)}, a^{(K-1)}, x^{(\sigma(N))}\right\}$ has been generated under policy $\pi_\eta$. From the definition of a solution to the Poisson equation (\ref{eq:Poisson}):
\begin{align*}
g(x^{(k) }) - \mu_\eta^Tg =  h_\eta(x^{(k)})-\sum\limits_{y\in \X}P_\eta (y| x^{(k)} ) h_\eta(y) \text{ for each state } x^{(k)}  \text{ in the simulated episode.}
 \end{align*}
If the approximation $\zeta$ is sufficiently close to $ h_\eta$, then the correlation between
\begin{align*}
g(x^{(k) }) -  \widehat {\mu^T_\eta g}  \quad \text{ and }\quad  \zeta(x^{(k)})-\sum\limits_{y\in X} P_\eta\left(y| x^{(k)}\right) \zeta(y)
\end{align*}
is positive and we can use the control variate  to reduce the variance. This idea gives rise to the approximating martingale-process (AMP) method proposed in \cite{Henderson2002}; also see \cite{Andradottir1993}.

Following \cite[Proposition 7]{Henderson2002}, for some approximation $\zeta$ such that   $\mu_\eta^T \zeta<\infty$ and $\zeta(x^*)=0$,
we consider the martingale process starting from an arbitrary state $x^{(k)}$  until the first regeneration time:
\begin{align}\label{eq:M}
M_{\sigma_k} (x^{(k)})=\zeta(x^{(k)}) +\sum\limits_{t=k}^{\sigma_k-1} \left[\sum\limits_{y\in \X} P_\eta\left(y|x^{(t)}\right)\zeta(y)  - \zeta(x^{(t)})\right],
\end{align}
where $\sigma_k = \min\left\{t>k~|~x^{(t)}=x^*\right\}$  is the first time when the regeneration state $x^*$ is visited after time $k$.
% we define a martingale process $M = (M_n:n\geq 0)$ , where for $x = x_0$
% \begin{align}\label{eq:M}
% M_n(x) :=  \zeta(x_0)-\zeta(x_n) +\sum\limits_{k=0}^{n-1} \left[\sum\limits_{y\in \X} P_\eta(y|x_k) \zeta(y)  - \zeta(x_k)\right].
% \end{align}
The martingale process (\ref{eq:M}) has zero expectation $\E M_n = 0$ for all $n\geq 0$; therefore we use it as   a control variate to define a new estimator. %From Lemma \ref{eq:h2}  $h_\eta(x^*)=0$, and we assume that $\zeta(x^*)=0$.
%Consider the martingale process starting from an arbitrary state $x_k$  until the first regeneration time:
% \begin{align*}
% M_{\sigma_k} (x_k)=\zeta(x_k) +\sum\limits_{t=k}^{\sigma_k-1} \left[\sum\limits_{y\in \X} P_\eta(y|x_{t})\zeta(y)  - \zeta(x_{t})\right],
% \end{align*}
% where $\sigma_k = \min\{t>k~|~x_t=x^*\}$  is the first time when the regeneration state $x^*$ is visited after time $k$.
Adding $M_{\sigma_k}$ to estimator (\ref{eq:es1}) we get the AMP estimator of the solution to the Poisson equation:
\begin{align}\label{eq:es2}
\hat h_\eta^{AMP(\zeta)} (x^{(k)} )&: =\zeta(x^{(k)}) +  \sum\limits_{t=k}^{\sigma_k-1} \left(g(x^{(t) }) - \widehat {\mu_\eta^Tg} +\sum\limits_{y\in \X} P_\eta\left(y|x^{(t)}\right)\zeta(y) -\zeta(x^{(t)})  \right) .
\end{align}
We  assume that the estimation of the average cost is accurate (i.e. $ \widehat {\mu_\eta^Tg}  = \mu_\eta^Tg$). In this case estimator (\ref{eq:es2}) has zero variance if the approximation is exact $\zeta = h_\eta$.

% We get the AMP  estimator of the advantage function
%\begin{align}\label{eq:es3}
%\hat A^{AMP(\tilde h)} (x_k, a_k) &:=\hat Q^{AMP(\tilde h)}(x_k, a_k) -\tilde h(x_k) \nonumber\\
%&=  \sum\limits_{t=0}^{\sigma_k(x^*)-1} \left(g(x_{k+t }) - \widehat {\mu^Tg} +\sum\limits_{x'\in \X} P_\eta(x'|x) \tilde h(x') -\tilde  h(x_k)  \right).
%\end{align}

 %The empirical evidence  of sample complexity improvement with AMP estimator is provided in Section \ref{sec:AMPexp}.

Now we want to replace the standard regenerative estimator (\ref{eq:es1})  used in line 8 of Algorithm \ref{alg1} with AMP estimator (\ref{eq:es2}). As the approximation $\zeta$ needed in (\ref{eq:es2}), we use   $f_{\psi_{i-1}}$ that approximates a solution to the Poisson equation corresponding to previous policy $\pi_{\theta_{i-1}}$. In line 8  of Algorithm \ref{alg1} we replace $\hat h (x^{(k)})$ with the estimates $\hat h^{AMP(f_{\psi_{i-1}} )} (x^{(k)})$ that are computed by (\ref{eq:es2}).

 \begin{algorithm}[h]
\SetAlgoLined
\KwResult{policy $\pi_{\theta_I}$ }
 Initialize  policy $\pi_{\theta_0}$ and value function $f_{\psi_{-1}}\equiv 0$ approximators \;
 \For{ policy iteration $i= 0, 1, ..., I-1$}{
  \For{ actor  $q= 1, 2, ..., Q$}{
  Run policy $\pi_{\theta_{i}}$  until it reaches $N$th regeneration time on $\sigma^q(N)$ step: collect an episode
   $\left\{x^{(0, q)}, a^{(0, q)}, x^{(1 , q)}, a^{(1, q)}, \cdots, x^{(\sigma^q( N)-1, q)}, a^{(\sigma^q(N)-1, q)}, x^{(\sigma^q(N), q)}\right\} $\;
  }

  Compute the average cost estimate $\widehat{\mu^T_{\theta_i}  g}$ by (\ref{es_av})\;
        Compute  $\hat h^{AMP(f_{\psi_{i-1}})}_{k, q}$,  the estimate of $h_{\theta_i}(x^{(k, q)})$, by (\ref{eq:es2})  for each $q = 1, .., Q$, $k=0, .., \sigma^q(N)-1$\;
        Update $\psi_{i}: = \psi$, where $\psi\in\Psi$ minimizes $  \sum\limits_{q=1}^Q \sum\limits_{k=0}^{\sigma^q(N)-1 } \left ( f_{\psi}(x^{(k, q)}) -\hat h^{AMP(f_{\psi_{i-1}})}_{k, q}   \right)^2$ following (\ref{eq:Vappr})\;
  Estimate the advantage functions $\hat A_{\theta_{i}}(x^{(k, q)}, a^{(k, q)})$ using (\ref{eq:Aes1}) for each $q = 1, .., Q$, $k=0,..., \sigma^q(N)-1$:
 \begin{align*}
D^{ (0:\sigma^q(N)-1)_{q=1}^Q } = \left\{\left( x^{(0, q)}, a^{(0, q)}, \hat A_{0,q}\right), \cdots,  \left( x^{(\sigma^q(N)-1, q)}, a^{(\sigma^q(N)-1, q)}, \hat A_{\sigma^q(N)-1,q}\right)\right\}_{q=1} ^{Q }.
\end{align*}\\
 Minimize the surrogate objective function w.r.t. $\theta\in \Theta$:
 \begin{align*}
  \hat L\left(\theta, \theta_i, D^{ (0:\sigma^q(N)-1)_{q=1}^Q}\right) =\sum\limits_{q=1}^Q \sum\limits_{k=0}^{\sigma^q(N)-1} \max\Big[ &\frac{\pi_{\theta}(a^{(k, q)}| x^{(k, q)})}{\pi_{\theta_{i}}(a^{(k, q)}|x^{(k, q)})}  \hat A_{\theta_{i}}(x^{(k, q)}, a^{(k, q)}) ,\\ &\text{clip}\left(\frac{\pi_{\theta}(a^{(k,q)}| x^{(k, q)})}{\pi_{\theta_{i}}(a^{(k, q)}|x^{(k, q)})},     1-\epsilon, 1+\epsilon  \right) \hat A _{\theta_{i}}(x^{(k, q)}, a^{(k,q)})  \Big]
   \end{align*}\\
 Update $\theta_{i+1}: = \theta$.

 }
 \caption{Proximal policy optimization with AMP method}\label{alg1amp}
\end{algorithm}

\subsection{Variance reduction through discounting}\label{sec:ge}

 Unless an approximation $\zeta$ is exact, each term in  the summation in (\ref{eq:es2}) is random with nonzero variance.  When the expected length of a regeneration cycle is large, the   cumulative variance  of estimator (\ref{eq:es2})   can be devastating.

In this subsection, we describe a commonly used solution:  introduce a forgetting   factor $\gamma\in (0,1)$
 to discount the future relative costs,  \cite{Jaakkola1994, Baxter2001, Marbach2001, Kakade2001, Thomas2014, Schulman2016}. % When a regeneration cycle are long and in average its length $\E G$ is large   one may sum relative costs $g(x_k) - \mu^Tg$ up to some fixed timestep $N$ in estimator (\ref{eq:es1})  s.t. $N<\E G$. The choice of $N$ can be generalized assuming that the
 % number of summands $N$  follows geometrical distribution with parameter $\gamma<1$ as in TD($\lambda$) method \cite[Section 12]{Sutton2018}.   }

%\begin{align}\label{eq:J}
%J_{\eta}^{(\gamma)}(x): = \E \left[ \sum\limits_{t=0}^{\infty} \gamma^{t}  \left(g(x_t) -  \mu^T g \right) ~\Big|~ x_0 = x\right] \text{ for each }x\in \X,
%\end{align}
% where  $x_k$ is the state of the Markov chain with transition matrix $P_\eta$ at time $k$, $\gamma\in (0,1]$ is a discount factor.
% We note that $J^{(\gamma)}_{\eta} \rightarrow h_\eta$  in $\V$-norm as $\gamma\rightarrow 1$  for $\V$-uniformly ergodic Markov chain, see Lemma \ref{lem:disc}.
We let
\begin{align}\label{eq:r}
r(x^*):= (1-\gamma)\E \left[\sum\limits_{t=0}^\infty \gamma^t g(x^{(t)}) ~|~x^{(0)}=x^*\right]
\end{align}
be a \textit{present discounted value} at state $x^*$; the term
``present discounted value'' was proposed in \cite[Section
11.2]{Wagner1975}. We define the \textit{regenerative
  discounted relative value function}   as:
\begin{align}\label{eq:Vreg}
V^{(\gamma)}(x): = \E \left[ \sum\limits_{t=0}^{\sigma(x^*)-1} \gamma^{t}  \left(g(x^{(t)}) -  r(x^*) \right) ~\Big| ~x^{(0)} = x\right] \text{ for each }x\in \X,
\end{align}
where $x^{(k)}$ is the state of the Markov chain with transition matrix
$P$ at time $k$, $x^*$ is the prespecified regeneration
state, and $\gamma\in (0,1]$ is a discount factor.  We note that
$V^{(\gamma)}(x^*) = 0$ by definition.  It follows from
  \cite[Corollary 8.2.5.]{Puterman2005} that under the drift condition,
  $r(x^*)\to \mu^T g$ as $\gamma\uparrow 1$.  Furthermore, by
  Lemma \ref{lem:disc}, for each $x\in \X$,
  \begin{align*}
   V^{(\gamma)}(x) \to h^{(x^*)}(x)  \text{ as } \gamma\uparrow 1,
  \end{align*}
  where $h^{(x^*)}$ is a solution to the Poisson equation given in (\ref{eq:h2}).

  The proof of Lemma \ref{lem:disc} can be found in Appendix Section \ref{sec:disc}.

  \begin{lemma}\label{lem:disc}
  We consider  irreducible, aperiodic   Markov chain  with transition matrix $P$ that satisfies drift condition (\ref{eq:drift}).
We let (\ref{eq:Vreg})  be a regenerative
  discounted relative value function  for discount factor $\gamma$ and one-step cost function $g$,  such that  $|g(x)|<\V(x)$ for each $x\in \X$. We let $h^{(x^*)} $ be a solution of the Poisson equation (\ref{eq:Poisson}) defined by (\ref{eq:h2}).

 Then for some constants $R<\infty$ and $r\in(0,1)$  we have
\begin{align*}
\left|V^{(\gamma)}(x) - h^{(x^*)}(x) \right|\leq \frac{rR(1-\gamma)}{(1-r)(1-\gamma r)}(\V(x)+\V(x^*))
\end{align*}
for each $x\in \X.$
\end{lemma}

  We let
\begin{align}\label{eq:VdiscVarEst}
 \hat V^{(\gamma )}(x):=  \sum\limits_{k=0}^{\sigma(x^*)-1} \gamma^k\left(g (x^{(k)} )- r(x^*) \right )
\end{align}
   where $x^{(0)} = x$ and $x^{(k)}$ is the $k$th step of the Markov chain with transition matrix $P$, be an one-replication estimate of (\ref{eq:Vreg}).
    By Lemma \ref{lem:var} the variance of this estimator $Var[ \hat V^{(\gamma )}(x)] $
converges to zero with rate $\gamma^2$ as $\gamma\downarrow 0$ for each $x\in \X$. See Appendix Section \ref{sec:disc} for the proof.

   \begin{lemma}\label{lem:var}
   We consider  irreducible, aperiodic   Markov chain  with transition matrix $P$ that satisfies drift condition (\ref{eq:drift})
and assume that one-step cost function $g:\X\rightarrow \R$ satisfies $g^2(x)\leq\V(x)$ for each $x\in \X$.
We let (\ref{eq:Vreg})  be a regenerative
  discounted relative value function for discount factor $\gamma$, and assume that   regeneration state  $x^*$ is such that set \begin{align*}
\Big\{x\in \X:b\V(x)\leq \V(x^*)\Big\}
\end{align*} is a finite set, where function $\V$ and constant $b$ are from (\ref{eq:drift}). We consider an arbitrary $x\in \X$ and let (\ref{eq:VdiscVarEst})
 be an one-replication estimate of   (\ref{eq:Vreg}).

If $\gamma<1$, then there exist   constants $R< \infty$, $B<\infty$, and $r\in (0,1)$ independent of $\gamma$ such that the variance of estimate (\ref{eq:VdiscVarEst}) is bounded as
 \begin{align*}
 Var[ \hat V^{(\gamma )}(x)] \leq \gamma^2\left( R\V(x)\frac{1}{1-\gamma^2r} + (\mu^T\V)B \frac{1}{1-\gamma^2}  \right), \text{ for each }x\in \X,
\end{align*}
where   $\mu$ is a stationary distribution of transition matrix $P$.
\end{lemma}

For a fixed  $\gamma\in (0,1)$ and state $x\in \X$  any unbiased
  estimator of $V^{(\gamma)}(x)$ is a biased estimator of
  $h^{(x^*)}(x)$. It turns out that the discount counterparts of the
  estimators (\ref{eq:es1}) and (\ref{eq:es2}) for
  $V^{(\gamma)}(x)$ have smaller variances than the two estimators for $h^{(x^*)}(x)$. This variance
  reduction can be explained intuitively as follows.  Introducing the
discount factor $\gamma$ can be interpreted as a modification of the
original transition dynamics; under the modified dynamics, any action
produces a transition into a regeneration state with probability at
least $1 - \gamma$, thus shortening the length of   regenerative
cycles. See Appendix Section \ref{sec:disc} for details.

%Both function $J_{\eta}^{(\gamma)}$ and $V_{\eta}^{(\gamma)}$ are solutions to the same Poisson's equation, see Section \ref{sec:disc}. Therefore, there exists a constant $b\in \R$ such that $J_{\eta}^{(\gamma)}(x) = V_{\eta}^{(\gamma)}(x)+b$ for each $x\in \X$.

We define a \textit{discounted} advantage function for policy $\pi_\eta$ as:
\begin{align}\label{eq:Adisc}
A^{(\gamma)}_{\eta}(x, a):& = \underset{\substack{ y\sim P(\cdot|x, a)}  }{\E} \left[g(x) - \mu_\eta^T g + V_{\eta}^{(\gamma)}(y) -V_{\eta}^{(\gamma)} (x)\right].
\end{align}

For two policies $\pi_\theta$ and  $\pi_\eta$, $\eta, \theta \in \Theta$, we define an approximation of $N_1(\theta, \eta)$ (\ref{eq:MA}) as:
\begin{align}\label{eq:Ndisc}
N_1^{(\gamma)}(\theta, \eta):=   \underset{\substack{ x\sim \mu_{\eta} \\ a\sim \pi_{\theta}(\cdot|x)  }  }{\E}  \Big[ r_{\theta, \eta}(a|x) A^{(\gamma)}_{\eta} (x, a)\Big].
\end{align}

%We note that if drift condition (\ref{eq:drift}) holds for $P_\eta$, and $P_\theta$ satisfies (\ref{eq:D})
%then  $N_1^{(\gamma)}(\theta, \eta) \rightarrow N_1(\theta, \eta)$ as $\gamma\rightarrow 1$, see Lemma \ref{lem:M1disc}.

We use the function approximation   $f_\psi$ of $V_{\eta}^{(\gamma)}$ to   estimate the advantage function (\ref{eq:Adisc})  as (\ref{eq:Aes1}).

We now present the discounted version of the AMP estimator (\ref{eq:es2}).
We let $\zeta$ be an approximation of the discounted value function $V_{\eta}^{(\gamma)}$ such that   $ \mu_\eta^T \zeta <\infty$ and $\zeta(x^*)=0$.  We define the sequence $(M_{\eta}^{(n)}:n\geq 0)$:
\begin{align}\label{eq:mart}
M_{\eta}^{(n)}(x):=\sum\limits_{t=k}^{n-1} \gamma ^{t-k+1}\left[ \zeta   (x^{(t+1)}) - \sum\limits_{y\in \X} P_{\eta}\left(y|x^{(t)}\right)  \zeta(y)  \right],
\end{align}
where   $x = x^{(k)}$ and $x^{( t)}$ is a state of the Markov chain   after $ t$ steps.

We define a one-replication of the AMP estimator for the discounted value function:
\begin{align}\label{eq:es4}
\hat V^{ AMP( \zeta), (\gamma)}_{\eta}(x^{(k)} ) :&=  \sum\limits_{t=k}^{\sigma_k-1} \gamma^{t-k}  \left(g(x^{(t)}) - \widehat{ r_\eta(x^*)} \right)  - M^{(\sigma_k )}_\eta(x^{(k)})\\
&=    \zeta(x^{(k)}) +    \sum\limits_{t=k}^{\sigma_k - 1}\gamma^{t-k} \left(g(x^{(t)} ) - \widehat{r_\eta(x^*)}+  \gamma \sum\limits_{y\in \X} P_{\eta}\left(y|x^{(t)}\right)  \zeta(y)   -   \zeta(x^{(t)} )  \right)  - \gamma^{\sigma_k-k} \zeta(x^*) \nonumber\\
&=    \zeta(x^{(k)}) +    \sum\limits_{t=k}^{\sigma_k - 1}\gamma^{t-k} \left(g(x^{(t)} ) -\widehat{r_\eta(x^*)}+  \gamma \sum\limits_{y\in \X} P_{\eta}\left(y|x^{(t)}\right)  \zeta(y)   -   \zeta(x^{(t)} )  \right),  \nonumber
\end{align}
where $ \widehat{r_\eta(x^*)}$ is an estimation of $r(x^*),$ and $\sigma_k = \min\left\{t>k~|~x^{(t)}=x^*\right\}$ is the first time  the regeneration state $x^*$ is visited after time $k$.

The AMP estimator (\ref{eq:es4}) does not introduce any bias subtracting $M_\eta$ from $\hat V^{(\gamma)}_\eta$ since    $\E M_{\eta}^{(n)} =0$ for any $n>0$ by \cite{Henderson2002}.
Function $V_{\eta}^{(\gamma)}$ is a solution of the following equation (see Lemma  \ref{lem:2sol}):
\begin{align}\label{eq:Poiss_reg}
g(x) - r_\eta(x^*) + \gamma\sum\limits_{y\in \X}P_\eta (y|x) h(y) - h(x) =0 \quad \text{ for each }x\in \X.
\end{align}
Therefore, similar to (\ref{eq:es2}), estimator (\ref{eq:es4}) has zero variance if approximation is exact  $\widehat{r_\eta(x^*)} = r_\eta(x^*)$ and $\zeta = V _\eta$,  see Poisson equation (\ref{eq:Poiss_reg}).

Further variance reduction is possible via \textit{$T$-step truncation} \cite[Section 6]{Sutton2018}. We consider an   estimate of  the value function  (\ref{eq:Vreg})  at a state $x\in \X$ as   the sum of the discounted costs before time $T$, where $T<\sigma(x^*)$,   and the discounted costs after time $T$:

\begin{align}\label{eq:Vst}
\hat V^{(\gamma)}(x) =  \sum\limits_{t=0}^{T-1} \gamma^t \left(g(x^{(t)})-\widehat{r(x^*)}\right) + \gamma^T\sum\limits_{t=0}^{\sigma(x^*)-1} \gamma^t \left(g(x^{(T+t)})-\widehat{r(x^*)}\right),
\end{align}
where   $ x^{(0)}=x$,  $x^{(t)}$ is a state of the Markov chain  after   $t$  steps and $\sum\limits_{t=0}^{\sigma(x^*)-1} \gamma^t g(x^{(T+t)})$ is a standard one-replication estimation of the value function at state $x^{(T)}$.
Instead of estimating the value at state $x^{(T)}$ by a random roll-out (second term in (\ref{eq:Vst})), we can use the value of deterministic approximation function $\zeta$ at state $x^{(T)}.$ The $T$-step truncation reduces the variance of the standard estimator but introduces bias unless the approximation is exact $\zeta(x^{(T)}) =  V^{(\gamma)}(x^{(T)})$.

A \textit{$T$-truncated} version of the AMP estimator is
\begin{align}\label{eq:esT}
\hat V^{AMP( \zeta), (\gamma, T)}_k :&= \sum\limits_{t=k}^{T\wedge \sigma_k-1} \gamma^{t-k} \left(g(x^{(t)})-\widehat{r_\eta(x^*)}\right) + \gamma^{T\wedge \sigma_k -k }\zeta(x^{(T\wedge \sigma_k) }) - M_\eta^{(T\wedge \sigma_k )}(x^{(k)}) \\
                                     &=\sum\limits_{t=k}^{T\wedge \sigma_k -1} \gamma^{t-k} \left(g(x^{(t)})-\widehat{r_\eta(x^*)}\right) - \sum\limits_{t=k}^{T\wedge \sigma_k -1} \gamma^{t-k+1}\left(\zeta(x^{(t+1) }) - \sum\limits_{y\in \X} P_\eta\left(y|x^{(t)} \right)\zeta(y) \right) \nonumber \\
  & \quad { } + \gamma^{T\wedge \sigma_k -k}\zeta(x^{(T\wedge \sigma_k)}) \nonumber \\
&=   \zeta(x^{(k)}) + \sum\limits_{t=k}^{T\wedge \sigma_k -1}  \gamma ^{t-k} \left(g(x^{(t)}) -\widehat{r_\eta(x^*)}+  \gamma \sum\limits_{y\in \X} P_{\eta}\left(y|x^{(t)}\right)  \zeta(y)   -   \zeta(x^{(t)})  \right),\nonumber
\end{align}
where $T\wedge \sigma_k = \min(T, \sigma_k)$.
We note that if the value function approximation and present discounted value approximation  are exact, estimator (\ref{eq:esT}) is unbiased for $V^{(\gamma)}(x^{(k)})$ and has zero variance.  We generalize the $T$-truncated estimator by taking  the
  number of summands $T$  to follow the geometrical distribution with parameter $\lambda<1$ as in the TD($\lambda$) method \cite[Section 12]{Sutton2018}, \cite[Section 3]{Schulman2016}:
\begin{align}\label{eq:es4}
\hat V^{AMP( \zeta), (\gamma, \lambda)}_k  :&=  \E_{T\sim Geom(1-\lambda)} \hat V^{AMP( \zeta), (\gamma, T)}_k \\
&=  (1-\lambda)\Big( \hat V^{AMP( \zeta), (\gamma, 1)}_k +\lambda  \hat V^{AMP( \zeta), (\gamma, 2)}_k  + \lambda^2 \hat V^{AMP( \zeta), (\gamma, 3)}_k+\cdots\nonumber\\
&\quad\quad +\lambda^{\sigma_k}V^{AMP( \zeta), (\gamma, \sigma_k)}_k  + \lambda^{\sigma_k+1}V^{AMP( \zeta), (\gamma, \sigma_k)}_k  +\cdots\Big)\nonumber\\
&= \zeta(x^{(k)}) + \sum\limits_{t=k}^{\sigma_k-1} (\gamma\lambda)^{t-k} \left(g(x^{(t)}) -\widehat{r_\eta(x^*)}+  \gamma \sum\limits_{y\in \X} P_{\eta}\left(y|x^{(t)}\right)  \zeta(y)   -   \zeta(x^{(t)})  \right).\nonumber
\end{align}

The regenerative cycles can be very long.  In practice we want to control/predict the  time and memory amount  allocated for the algorithm execution. Therefore, the simulated episodes should have  finite lengths. We use the following estimation for the first $N$ timesteps if an episode with finite length $N+L$ is generated:
\begin{align}\label{eq:esf}
\hat V^{AMP( \zeta), (\gamma, \lambda, N+L)}_k  := \zeta(x^{(k)}) + \sum\limits_{t=k}^{(N+L) \wedge \sigma_k -1} (\gamma\lambda)^{t-k} \left(g(x^{(t)})  - \widehat{r_\eta(x^*)}+  \gamma \sum\limits_{y\in \X} P_{\eta}\left(y|x^{(t)}\right)  \zeta(y)   -   \zeta(x^{(t)})  \right),
\end{align}
 where $k=0,..., N-1$, and integer $L$ is large enough. We note that if an episode has a finite length $N+L$, regeneration $\sigma_k$  may have not been observed in the generated episode (i.e. $\sigma_k>N+L$).  In this case, we summarize  (\ref{eq:esf}) up to the end of the episode. %As a result, each episode generates $N$ estimates (\ref{eq:esf}). \red{ do you mean at $x_0$, $\ldots$, $x_{N-1}$?}

We provide the PPO algorithm where each of $q=1, ..., Q$ parallel actors simulates an episode with length $N+L$: $\left\{x^{(0, q)}, a^{(0,q)}, x^{(1, q)}, a^{(1, q)}, \cdots, x^{(N+L-1, q)}, a^{(N+L-1, q)} \right\}.$   Since we need the approximation $\zeta$ in (\ref{eq:esf}), we use   $f_{\psi_{i-1}}$ that approximates a regenerative discounted value function corresponding to previous policy $\pi_{\theta_{i-1}}$. See Algorithm \ref{alg2}.

\begin{algorithm}[H]
\SetAlgoLined
\KwResult{policy $\pi_{\theta_I}$ }
 Initialize  policy $\pi_{\theta_0}$ and value function $f_{\psi_{-1}}\equiv 0$ approximators \;
 \For{ policy iteration $i=0, 1,  ..., I-1$}{
  \For{ actor  $q= 1, 2, ..., Q$}{
  Run policy $\pi_{\theta_{i}}$  for $N+L$ timesteps: collect an episode
   $\left\{x^{({0, q})}, a^{({0, q})}, x^{({1 , q})}, a^{({1, q})}, ...., x^{({ N+L-1, q})}, a^{({N+L-1, q})} \right\} $\;
  }
  Estimate the average cost   $\widehat{\mu^T_{\theta_i}  g}$ by (\ref{es_av}), the present discounted value $\widehat{r_{\theta_i}(x^*)}$ by (\ref{eq:es_r}) below\;
        Compute  $\hat V^{AMP( f_{\psi_{i-1}}), (\gamma, \lambda)}_{k, q}  $ estimates by (\ref{eq:esf})  for each $q = 1, ..., Q$, $k=0, .., N-1$\;
        Update $\psi_{i}: = \psi$, where $\psi\in \Psi$ minimizes $\sum\limits_{q=1}^Q \sum\limits_{k=0}^{N-1 } \left ( f_{\psi}\left(x^{({k, q})}\right) -\hat V^{AMP(  f_{\psi_{i-1}}), (\gamma, \lambda)}_{k, q} \right)^2$ following (\ref{eq:Vappr}) \;
  Estimate the advantage functions $\hat A_{\theta_{i}}\left(x^{({k, q})}, a^{({k, q})}\right)$ using (\ref{eq:Aes1}) for each $q = 1, .., Q$, $k=0, ..,N-1$:\
  \begin{align*}
  D^{(0:N-1)_{q=1}^Q} = \left\{\left(x^{({0, q})}, a^{({0, q})}, \hat A^{(\gamma)}_{\theta_{i}} (x^{({0, q})}, a^{({0, q})})\right),  \cdots, \left(x^{({N-1, q})}, a^{({N-1, q})}, \hat A^{(\gamma)}_{\theta_{i}} (x^{({N-1, q})}, a^{({N-1, q})})\right) \right\}_{q=1}^Q
   \end{align*}\\
 Minimize the surrogate objective function w.r.t. $\theta\in \Theta$:\
 \begin{align*}
   \hat L^{(\gamma)}\left(\theta, \theta_i,  D^{(0:N-1)_{q=1}^Q} \right) =\sum\limits_{q=1}^Q \sum\limits_{k=0}^{ N-1} \max\Big[ &\frac{\pi_{\theta}\left(a^{({k, q})}| x^{({k, q})}\right)}{\pi_{\theta_{i}}\left(a^{({k, q})}|x^{({k, q})}\right)}  \hat A^{(\gamma)}_{\theta_{i}}\left(x^{({k, q})}, a^{({k, q})}\right) ,\\ &\text{clip}\left(\frac{\pi_{\theta}\left(a^{({k,q})}| x^{({k, q})}\right)}{\pi_{\theta_{i}}\left(a^{({k, q})}|x^{({k, q})}\right)},     1-\epsilon, 1+\epsilon  \right) \hat A^{(\gamma)}_{\theta_{i}}\left(x^{({k, q})}, a^{({k,q})}\right)  \Big];
   \end{align*}\\
 Update $\theta_{i+1}: = \theta$.
 }
 \caption{Proximal policy optimization with discounting}\label{alg2}
\end{algorithm}

In Algorithm \ref{alg2}, we assume that state $x^*$ has been visited $N_q$ times in the $q$th generated episode, when $q=1,..., Q$ parallel actors are available. Therefore, we estimate the present discounted value $r(x^*)$ as:
\begin{align}\label{eq:es_r}
\widehat{r(x^*)} = (1-\gamma)\frac{1}{\sum_{q=1}^Q N_q} \sum\limits_{q=1}^Q\sum\limits_{n=1}^{N_q} \sum\limits_{k=\sigma^q(n)}^{\sigma^q(n)+L} \gamma^{k-\sigma^q(n)} g(x^{({k,q})}),
\end{align}
where $\sigma^q(n)$ is the $n$th time when state $x^*$ is visited in the $q$th episode, and integer $L$ is a large enough.
If many parallel actors are available, we recommend starting the episodes from state $x^*$ to ensure that  state $x^*$ appears in the generated episodes  a sufficient number of times.

 \begin{remark}\label{rem:regVSinf}
Use the following \textit{discounted value function} as an  approximation of the solution $h_\eta$ of the Poisson equation:
\begin{align}\label{eq:disc_inf}
J^{(\gamma)}(x):=\E\left[ \sum\limits_{k=0}^{\infty} \gamma^k\left(g(x^{(k)})- \mu^Tg \right)~|~x^{(0)}=x \right] \text{ for each }x\in \X.
\end{align}
 We note that  the discounted value function  (\ref{eq:disc_inf}) and the regenerative discounted value function (\ref{eq:Vreg}) are solutions of the same Poisson equation (Lemma \ref{eq:newPoisson}). Therefore, the bias of advantage function estimator (\ref{eq:Adisc}) does not change when $h_\eta$ in (\ref{eq:A}) is replaced   either by $J^{(\gamma)}$ or by $V^{(\gamma)}$.
  The variance of the  regenerative discounted value function estimator (\ref{eq:esf}) can be \textit{potentially} smaller than the variance of the analogous $J^{AMP(\zeta), (\gamma, \lambda, N+L)}$ estimator:
  \begin{align}\label{eq:Jesf}
\hat J^{AMP( \zeta), (\gamma, \lambda, N+L)}_k  := \zeta(x^{(k)}) + \sum\limits_{t=k}^{N+L-1} (\gamma\lambda)^{t-k} \left(g(x^{({t})})  - \widehat{\mu^T_\eta g}+  \gamma \sum\limits_{y\in \X} P_{\eta}\left(y|x^{({t})}\right)  \zeta(y)   -   \zeta(x^{({t})})  \right).
\end{align}
Since the upper bound of summation in  (\ref{eq:esf}) is $\min(\sigma_k, N+L)-1$, it includes fewer summands than the summation in (\ref{eq:Jesf}) if the regeneration  frequently occurs.
See Appendix Section \ref{sec:regVSinf}, which describes a numerical experiment for the criss-cross network implying that the choice of  $J^{AMP(\zeta), (\gamma, \lambda, N+L)}$ and  $V^{AMP(\zeta), (\gamma, \lambda,N+ L)}$  estimators can affect the PPO algorithm convergence rate  to the optimal policy.
\end{remark}

%\subsection{Comparison of AMP and GAE estimation methods}\label{sec:AMPvsGAE}

The connection between our proposed AMP estimator and the GAE estimator \cite{Schulman2016} suggests another motivation for introducing the GAE estimator.
An accurate estimation of the value function $V^{(\gamma)}_\eta$  by the standard estimator  may require  a large number of state-action pairs samples from the current policy $\pi_\eta$ \cite[Section 13]{Sutton2018}, \cite{Schulman2016, Ilyas2020}. In this paper we apply the  AMP method to propose the discounted AMP estimator in (\ref{eq:esf}) that has  a smaller variance than the estimators   (\ref{eq:Vst}). %The AMP estimator relies on the construction of
  %an internal control variate  incorporating an approximation $\zeta$ of the value function.
  The AMP method, however, requires knowledge  of transition probabilities $P(y|x, a)$ in order to exactly compute the expected values for each $(x,~a,~y)~\in~\X~\times~\A~\times~\X$:
  \begin{align*}\underset{\substack{ a\sim \pi_{\eta}(\cdot|x)  \\ y\sim P(\cdot|x, a)}}{\E}  \zeta(y) =  \sum\limits_{y\in \X} \sum\limits_{a\in \A}P(y|x , a)\pi_\eta(a|x )\zeta(y). \end{align*}

One can relax the  requirement by replacing each expected value
 $\underset{\substack{ a\sim \pi_{\eta}\left(\cdot|x^{(t)}\right)  \\ y\sim P(\cdot|x, a)}}{\E}  \zeta(y)$  by its \textit{one-replication estimate}  $\zeta\left(x^{(t+1)}\right)$, where $\left(x^{(t)}, x^{(t+1)}\right)$ are two sequential states from an episode \begin{align*}
 \left(x^{(0)}, \cdots, x^{(t)}, x^{(t+1)}, \cdots\right)
 \end{align*} generated under policy $\pi_\eta$.

%\begin{remark}
%If a \textit{generative model}   of the environment is available, see definition in \cite[Section 1.1]{Kakade2003}, one can compute $n$-replications
% estimate $\frac{1}{n}\sum\limits_{i=1}^n \zeta(x_{t+1}^{(i)})$ for each state $x_t$ from a trajectory $(x_0, ..., x_t, x_{t+1}, ...)$ generated under
%  policy $\pi_\eta$. We denote $x_{t+1}^{(i)}$  as a state generated in the $i$th replication of executing action $a_i\sim\pi_\eta(\cdot|x_t)$ at state $x_t$.
%   The need of additional sampling at each state $x_t$ and the requirement of the generative model  make the \textit{n-replications AMP estimator}
%\begin{align*}
%\hat V_\eta^{\text{n-replications AMP}(\zeta), (\gamma)} &(x_k, a_k):= \zeta(x_k) +  \sum\limits_{t=0}^{\infty} \left(g(x_{k+t })  + \gamma \frac{1}{n}\sum\limits_{i=1}^n \zeta(x_{t+k+1}^{(i)}) -\zeta(x_{k+t})  \right)
%\end{align*}
%less attractive in comparison  with the standard estimator (\ref{eq:Vst2}).
%\end{remark}

%We conclude that  there is no unbiased estimation method, that we are aware of,  that can efficiently use  availability of an approximation function $\zeta$ to improve the estimation of the solution to the Poisson's equation  \textit{and does not require the knowledge of the transition probabilities}. If the transition probabilities are available the AMP method incorporates the approximation function into the control variate which  can potentially reduce the variance of the standard estimator.

In Section \ref{sec:ge} we propose  the   AMP estimator of the  regenerative discounted value function (\ref{eq:es4}):
\begin{align}\label{eq:Ves2}
\hat V^{AMP( \zeta), (\gamma, \lambda)}_k =   \zeta(x^{(k)}) + \sum\limits_{t=k}^{\sigma_k-1} (\gamma\lambda)^{t-k} \left(g(x^{(t)})  - \widehat{ r(x^*)}+  \gamma \sum\limits_{y\in\X} P_{\eta}(y|x^{(t)})  \zeta(y)   -   \zeta(x^{(t)})  \right).
\end{align}
%We want to remind the reader that (\ref{eq:Ves2}) is the unbiased AMP estimator (\ref{eq:es4}) of the  infinite-horizon discounted value function when $\lambda=1$ .

Replacing expectations $\sum\limits_{y\in \X} P_{\eta}\left(y|x^{(t)}\right)  \zeta(y)$ by the one-replication estimates $\zeta\left(x^{(t+1)}\right)$  we obtain the following estimator for the value function:
\begin{align}\label{eq:GAE}
\hat V^{GAE( \zeta), (\gamma, \lambda)}_k :=  \zeta(x^{(k) }) + \sum\limits_{t=k}^{\sigma_k-1} (\gamma\lambda)^{t-k} \left(g(x^{(k+t)})  - \widehat{ r(x^*)}+  \gamma \zeta(x^{(t+1)})   -   \zeta(x^{(t)})  \right),
\end{align}
which is a part of  the general advantage estimation (GAE) method \cite{Schulman2016}. We note that the advantage function   is estimated as $
\hat A^{GAE( \zeta), (\gamma, \lambda)}_k : = \hat V^{GAE( \zeta), (\gamma, \lambda)}_k - \zeta\left(x^{(k)}\right)$ in  \cite{Schulman2016}.

If $\lambda=1$, the estimator (\ref{eq:es4}) is transformed into  the standard estimator and does not depend on $\zeta$:
\begin{align*}%\label{eq:Vst2}
\hat V_\eta^{(\zeta), (\gamma)} (x^{(k)}):&= \zeta(x^{(k)}) \nonumber
 +  \sum\limits_{t=k}^{\sigma_k-1} \gamma^{t-k} \left(g(x^{(t) }) -\widehat{ r(x^*)}  +  \gamma \zeta(x^{(t+1)}) -\zeta(x^{(t)})  \right) \\
 &=\sum\limits_{t=k}^{\sigma_k-1}\gamma^{t-k} \left(g(x^{(t)})-\widehat{ r(x^*)}\right) .
\end{align*}
 The martingale-based control variate has no effect on the
  GAE estimator when $\lambda=1$, but it can produce significant variance reductions in the AMP estimator. See Section \ref{sec:cc} for numerical experiments.

%Therefore, the GAE estimator with $\lambda <1$ not only 'downweightes'   td-errors corresponding to  long delays but also allows to incorporate approximation of the value function $\zeta$, unlike
%the GAE estimator   with $\lambda = 1$, see (\ref{eq:Vst2}).
%We believe that  further     theoretical and empirical studies of the GAE method are required to disclose a major reason   of the GAE method efficiency. We hope that its analysis in conjunction with the AMP method can be helpful in further research.

\section{Experimental results for multiclass queueing networks}\label{sec:experiments}
In this section we evaluate the performance of the proposed proximal policy optimization Algorithms \ref{alg1},  \ref{alg1amp}, and \ref{alg2} for the multiclass queueing networks control optimization task discussed in Section \ref{sec:MQN}.

We use two separate fully connected feed-forward neural networks to represent  policies $\pi_\theta,$ $\theta\in \Theta$ and value functions $f_\psi$, $\psi\in \Psi$ with the architecture details given in Appendix Section \ref{sec:nn}.  We  refer to the neural network used to represent a policy as \textit{the  policy NN} and to  the neural network used to approximate a value function as \textit{the value function NN}. We  run the algorithm for $I=200$ policy iterations for each experiment. The algorithm uses  $Q=50$ actors to simulate data in parallel for each iteration.    See Appendix Section \ref{sec:par} for the details.

\subsection{Criss-cross network}\label{sec:cc}

 We first study the PPO algorithm and compare its base version Algorithm \ref{alg1} and its modification Algorithm \ref{alg1amp} that incorporates the  AMP method.  We check the robustness of the algorithms for the criss-cross system  with  various load (traffic) intensity regimes, including    I.L. (imbalanced light), B.L. (balanced light), I.M. (imbalanced medium), B.M. (balanced medium), I.H. (imbalanced heavy), and B.H. (balanced heavy) regimes.   Table \ref{t:lp} lists the corresponding arrival and service rates. The criss-cross network in any of these traffic regimes is stable under any work-conserving policy  \cite{Dai1996}. Since we want an initial policy to be stable, we forbid each server in the network to idle unless all its associated buffers are empty.

Table \ref{tab:cc}  summarizes the control policies proposed in the literature. Column 1 reports the load regimes, column 2 reports the
optimal performance obtained by dynamic programming (DP), column 3 reports the performance of a target-pursuing
policy (TP) \cite{Paschalidis2004} , column 4 reports the performance of a threshold policy \cite{Harrison1990}, columns 5 and 6 report the performance of fluid (FP) and robust fluid (RFP) policies respectively \cite{Bertsimas2015}, and column 7 reports the performance and the half width of the $95\%$ confidence intervals (CIs) of the PPO policy  $\pi_{\theta_I}$  resulting from the last iteration of Algorithm  \ref{alg1amp}.

We initialize the policy NN parameters $\theta_0$  using standard Xavier initialization \cite{Glorot2010}.  The resulting policy $\pi_{\theta_0}$ is close to the policy that chooses actions uniformly at random.  We take the empty system state $x^* = (0,0,0)$ as a regeneration state and simulate $N=5,000$ independent regenerative cycles  per actor in each iteration of the algorithm. Although the number of generated cycles is fixed for all traffic regimes, the length of the regenerative cycles varies  and highly depends on the load.

To show the learning curves in Figure \ref{fig:cc_opt},  we  save policy parameters $\{\theta_i\}_{i=0, 10, ..., 200}$ every 10th iteration over the course of learning. After the algorithm terminates we independently
 simulate   policies $\{    \pi_{\theta_i}: i = 0, 10, ..., 200 \}$ (in parallel) starting from the regeneration state $x = (0,..,0)$ until a fixed number of regenerative events occurs. For light, medium, and heavy  traffic regimes, we run the simulations  for $5\times10^7$, $5\times 10^6$, and  $10^6$ regenerative cycles respectively.
 We compute the $95\%-$confidence intervals   using the strongly consistent estimator of asymptotic variance. See  \cite[Section VI.2d]{Asmussen2003}.

  \begin{figure}[H]
    \subfloat[Imbalanced low (IL) traffic \label{subfig-6:IL}]{%
       \includegraphics[ height=0.19\textwidth, width=0.35\textwidth]{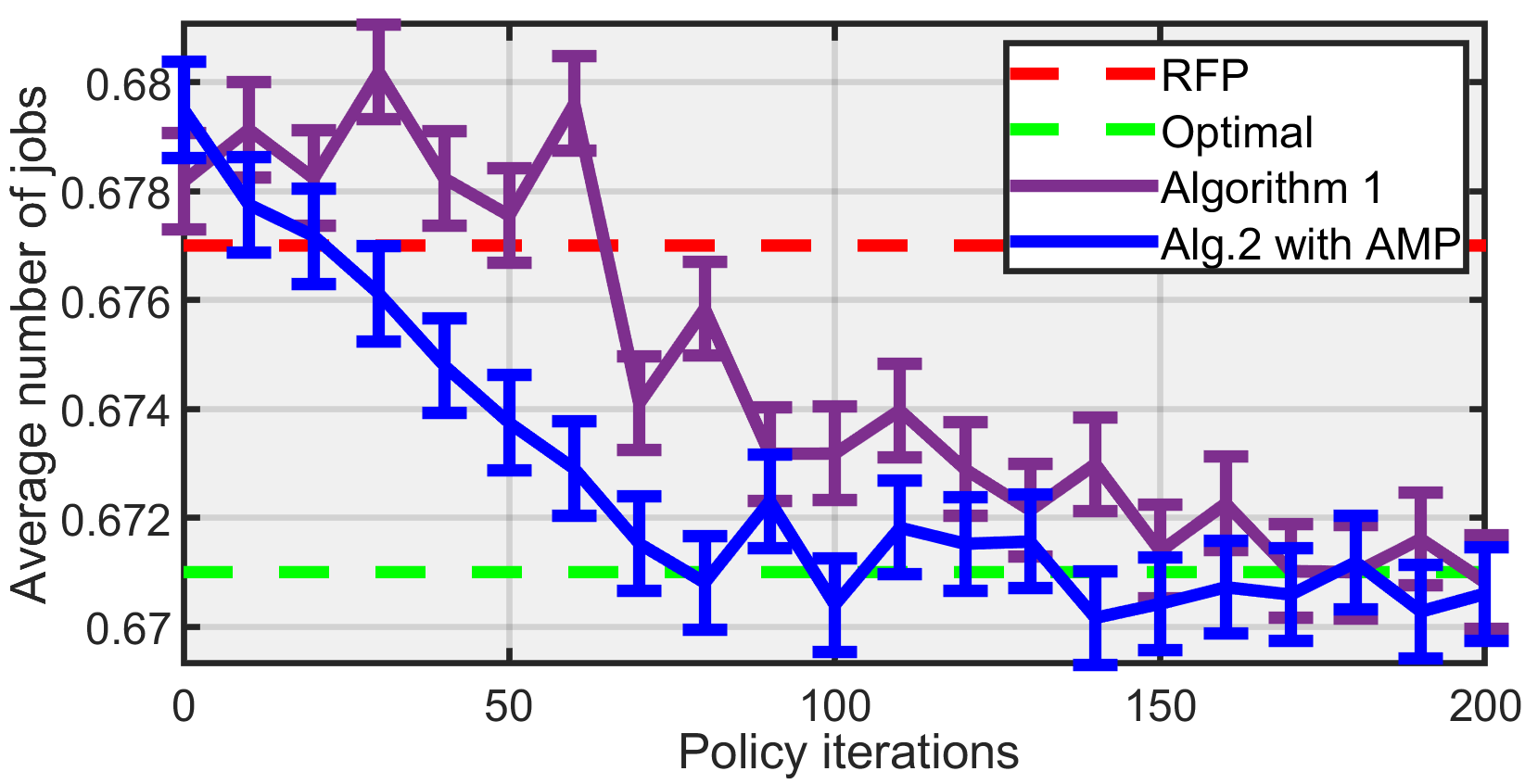}
     }
          \subfloat[Imbalanced medium (IM) traffic \label{subfig-4:IM}]{%
       \includegraphics[ height=0.19\textwidth, width=0.35\textwidth]{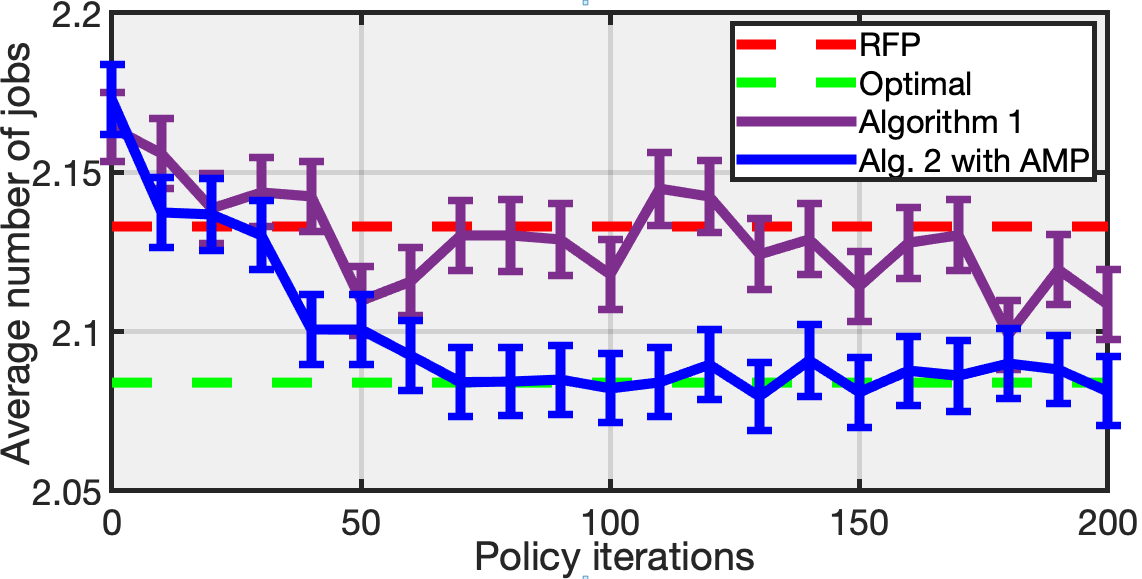}
     }
     \subfloat[Imbalanced heavy (IH) traffic \label{subfig-2:IH}]{%
       \includegraphics[ height=0.19\textwidth, width=0.35\textwidth]{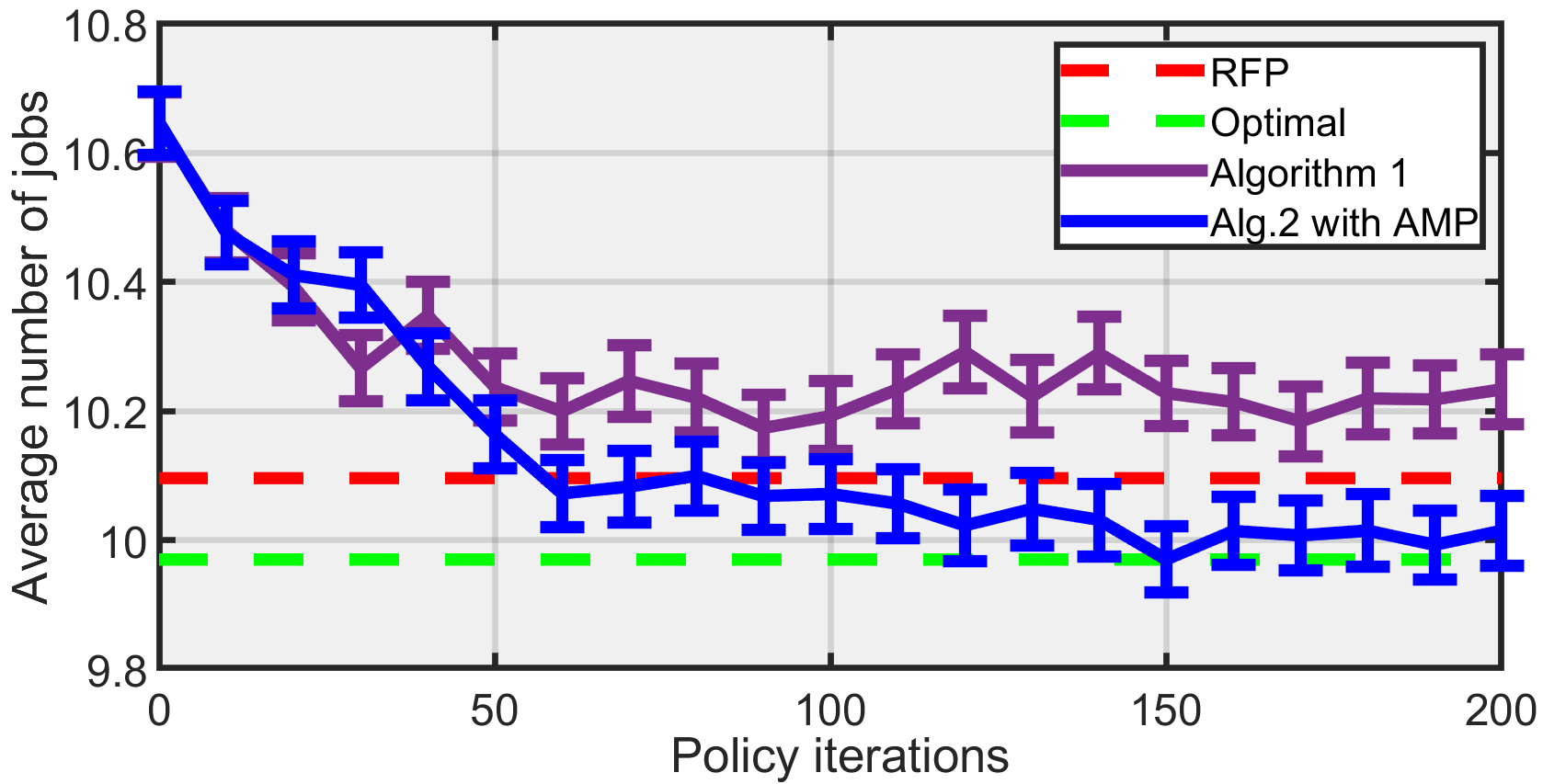}
     }\\
           \subfloat[Balanced low (BL) traffic\label{subfig-5:BL}]{%
       \includegraphics[  height=0.19\textwidth, width=0.35\textwidth]{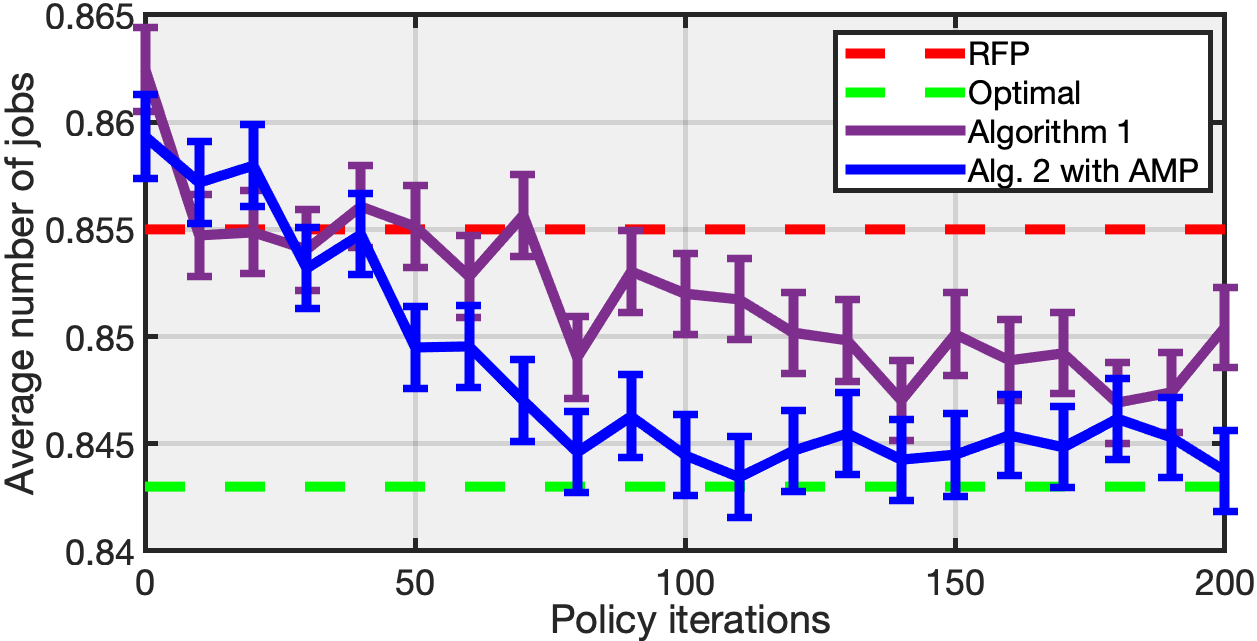}
     }
          \subfloat[Balanced medium (BM) traffic\label{subfig-3:BM}]{%
       \includegraphics[  height=0.19\textwidth, width=0.35\textwidth]{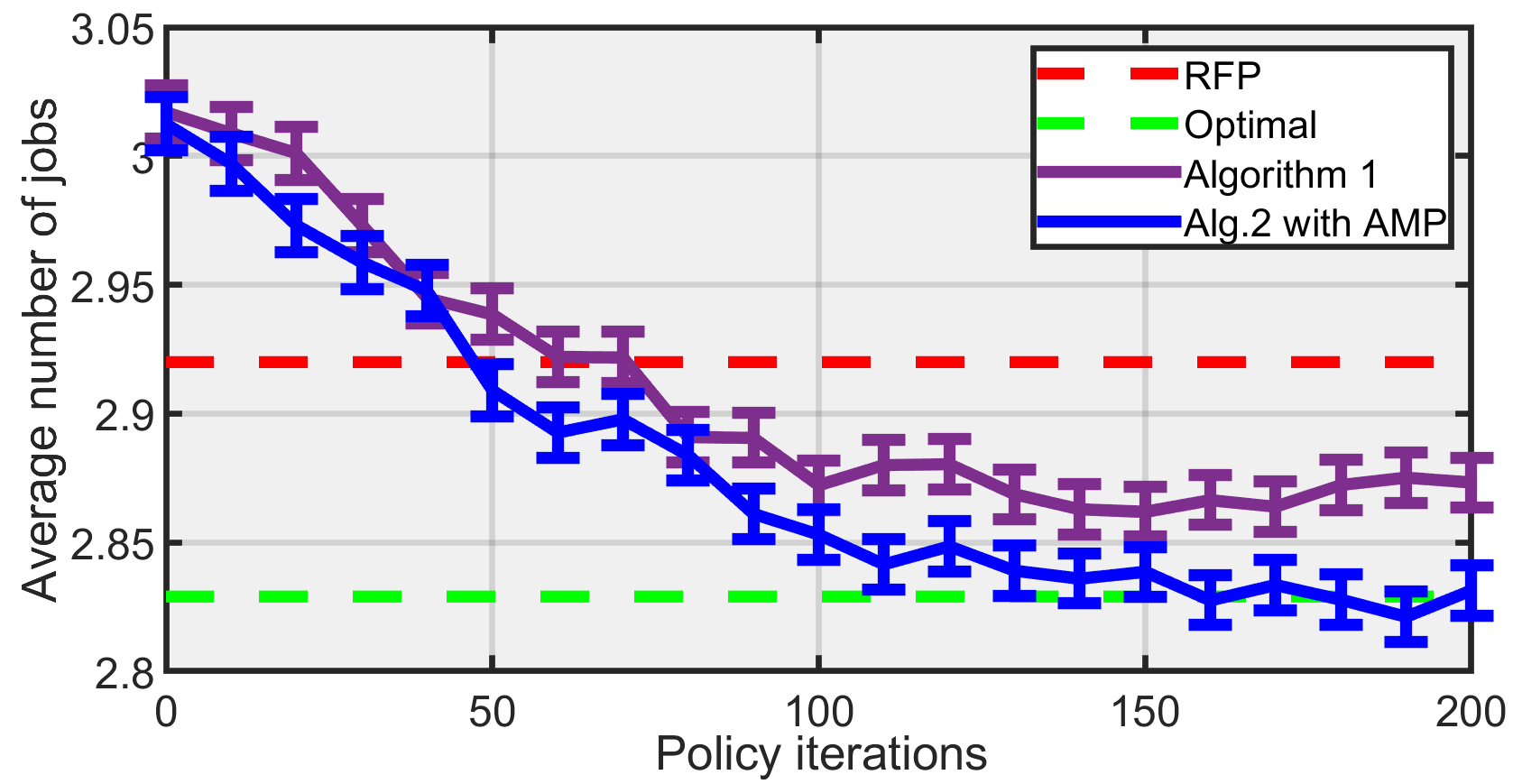}
     }
     \subfloat[Balanced heavy (BH) traffic\label{subfig-1:BH}]{%
       \includegraphics[ height=0.19\textwidth, width=0.35\textwidth]{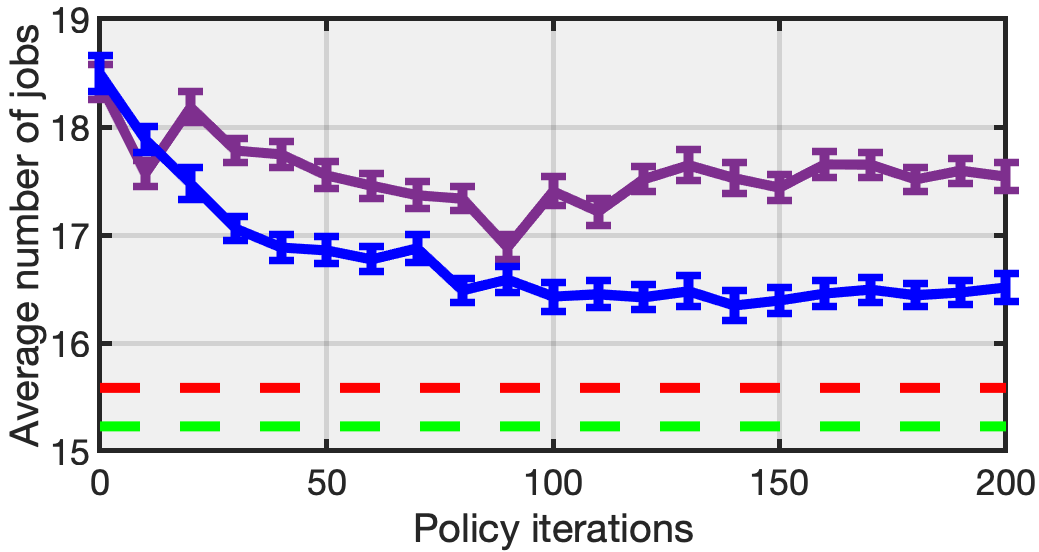}
     }

     \caption{
     Comparison of learning curves from Algorithm \ref{alg1} and Algorithm \ref{alg1amp}  on the criss-cross network with different traffic regimes.
    The solid purple and blue lines show the performance of  the PPO policies obtained at the end of every 10th iterations of Algorithm \ref{alg1} and Algorithm \ref{alg1amp}, respectively; the dashed red lines show the performance of the robust fluid policy (RFP), and the dashed green lines show the  performance of the optimal policy.  }
     \label{fig:cc_opt}
   \end{figure}
\begin{table}[H]
\centering
\begin{tabular}{|c|c|c|c|c|c|c|c|}
  \hline
  % after \\: \hline or \cline{col1-col2} \cline{col3-col4} ...
  Load regime  & $\lambda_1$ & $\lambda_3 $ & $\mu_1$ & $\mu_2$ & $\mu_3$ & $\rho_1$ & $\rho_2$\\\hline
  I.L. & 0.3 & 0.3& 2 & 1.5 &2& 0.3&0.2\\\hline
  B.L. & 0.3 & 0.3 & 2 & 1 &2 &0.3&0.3\\\hline
  I.M. & 0.6 & 0.6  & 2 & 1.5 &2& 0.6&0.4\\\hline
  B.M. & 0.6 & 0.6  & 2 & 1 &2& 0.6&0.6\\\hline
  I.H. & 0.9 & 0.9 & 2 & 1.5 &2& 0.9&0.6\\\hline
  B.H. & 0.9 & 0.9 & 2 & 1&2&0.9&0.9 \\

  \hline

\end{tabular}
 \caption[]{Load parameters for the criss-cross network of Figure \ref{fig:cc}
 }\label{t:lp}
\end{table}

\begin{table}[H]
\centering
\begin{tabular}{|c|c|c|c|c|c|c|}
  \hline
  % after \\: \hline or \cline{col1-col2} \cline{col3-col4} ...
  Load regime  & DP (optimal) & TP & threshold & FP & RFP & PPO (Algorithm \ref{alg1amp}) with CIs\\\hline
  I.L. & 0.671 & 0.678 & 0.679 & 0.678 &0.677&   $0.671\pm 0.001$ \\\hline
  B.L. & 0.843 & 0.856 & 0.857 & 0.857 &0.855&  $0.844\pm 0.004$\\\hline
  I.M. & 2.084 & 2.117 & 2.129 & 2.162 &2.133& $2.084\pm0.011$ \\\hline
  B.M. & 2.829 & 2.895 & 2.895 & 2.965 &2.920 &  $2.833\pm 0.010$\\\hline
  I.H. & 9.970 & 10.13 & 10.15 & 10.398 &10.096 & $10.014\pm0.055$\\\hline
  B.H. & 15.228 & 15.5 & 15.5 & 18.430 &15.585&   $16.513\pm 0.140$ \\

  \hline

\end{tabular}

 \caption[]{Average number of jobs per unit time in the criss-cross network under different policies.
  Column 1 reports the variances in the load regimes. }\label{tab:cc}%
\end{table}

   We observe that   Algorithm \ref{alg1amp} is not robust enough and converges to a  suboptimal policy when the criss-cross network operates in a balanced heavy load regime.
 We run Algorithm \ref{alg2} that uses discount factor $\gamma=0.998$, TD parameter $\lambda = 0.99$, and the AMP method for the value function estimation   at step 7.
  For each iteration we use $Q = 50$ parallel processes to generate trajectories, each with length $N=50,000$.
   We observe that Algorithm \ref{alg2} uses approximately 10 times fewer samples per iteration than  Algorithm \ref{alg1amp}. Algorithm \ref{alg2} outputs policy $\pi_{\theta_{200}}$ whose long-run average
   performance is $15.353\pm0.138$ jobs, which is lower than the RFP
     performance in \cite{Bertsimas2015}.
     We repeat the experiment with  Algorithm \ref{alg2} with the discounting, but we disable the AMP method in the value function estimation   at step 7.
      Figure \ref{fig:cc23} shows  that both variance reduction techniques (i.e. discounting, AMP method) have been necessary in Algorithm \ref{alg2} to achieve near-optimal performance.

%\begin{figure}[H]
%\centering%
%\includegraphics[width=.5\linewidth]{DiscVsUndisc}
%\caption[]{Comparison of learning curves from Algorithm \ref{alg1amp} and Algorithm \ref{alg2}  on the criss-cross network with the balanced heavy regime. }
%\label{fig:cc23}%
%\end{figure}

\begin{figure}[H]
\centering%
\includegraphics[width=.5\linewidth]{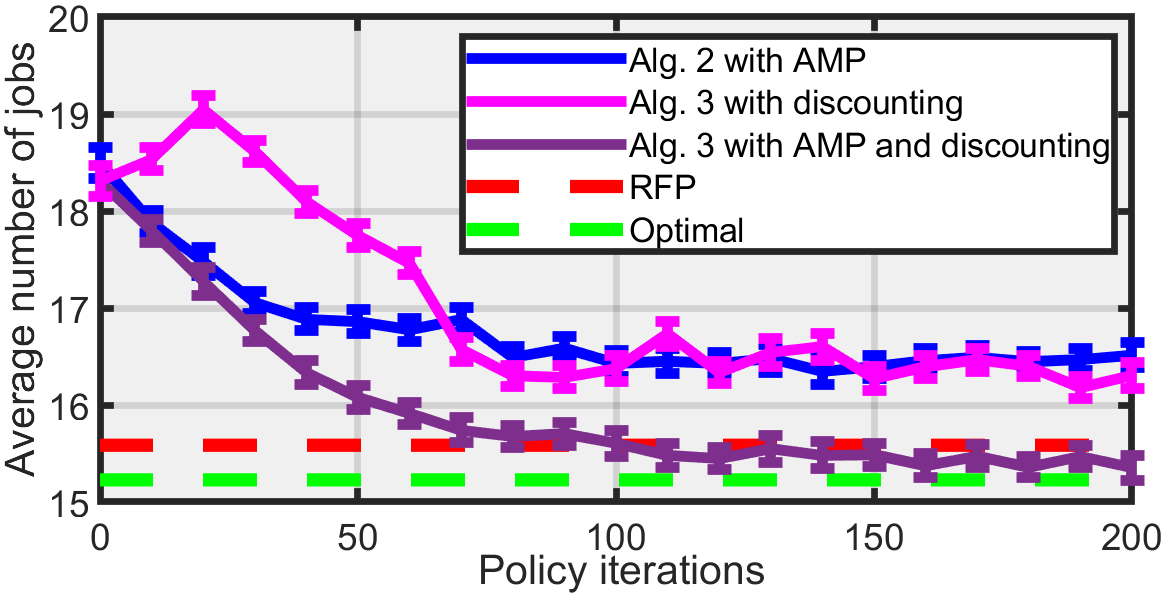}
\caption[]{ Comparison of learning curves from Algorithm \ref{alg1amp} and Algorithm \ref{alg2}  on the criss-cross network with the balanced heavy regime. The solid blue and purple lines  show the performance of  the PPO policies obtained at the end of every 10th iterations of Algorithm \ref{alg1amp} and Algorithm \ref{alg2}, respectively. The solid pink   line shows the performance of  the PPO policies obtained at the end of every 10th iterations of Algorithm \ref{alg2} without the  AMP method. The dashed red line shows the performance of the robust fluid policy (RFP), and the dashed green line shows the  performance of the optimal policy.}
\label{fig:cc23}%
\end{figure}

Recall that in the uniformization procedure, the transition matrix
$\tilde P$ in (\ref{eq:unif}) allows ``fictitious'' transitions when
$\tilde P(x|x)>0$ for some state $x\in \X$.  The PPO algorithm
approximately solves the discrete-time MDP (\ref{eq:obj4}), which
allows a decision at every state transition, including fictitious ones.
The algorithm produces randomized stationary Markovian
policies.  In evaluating the performance of any such policies, we
actually simulate a DTMC $\left\{x^{(k)}:k=0,1, 2, \ldots\right\}$ operating under
the policy, estimating the corresponding long-run average cost as in
(\ref{eq:obj4}).  There are two versions of DTMCs depending on how
often the randomized policy is sampled to determine the next
action. In version 1, the policy is re-sampled only when a \emph{real}
state transition occurs. Thus, whenever a fictitious state transition
occurs, no new action is determined and server priorities do not change in this version.
In version 2, the policy is re-sampled at  \emph{every}
transition. Unfortunately, there is no guarantee that these two
versions of DTMCs yield the same long-run average cost. See
\cite[Example 2.2]{Beutler1987} for a counterexample.

When the randomized stationary Markovian policy is optimal for the
discrete-time MDP (\ref{eq:obj4}), under an additional mild condition,
the two DTMC versions yield the same long-run average cost
\cite[Theorem 3.6]{Beutler1987}. Whenever simulation is used to estimate
the performance of a randomized policy in this paper, we use
version 1 of the DTMC. The reason for this choice is that the long-run
average for this version of DTMC is the same as the continuous-time
long-run average cost in (\ref{eq:obj3}), and the latter performance
has been used as  benchmarks in literature.
Although the final randomized stationary Markovian policy from our PPO
algorithm is not expected to be optimal, our numerical experiments
demonstrate that the performance results of the DTMC versions 1 and 2 are statistically
identical.  Table \ref{tab:cc2} reports  the performance of the DTMC versions 1 and 2  in the criss-cross network.   In Table \label{tab:cc2} column 2, the performance of Version 1 is identical to Table \ref{tab:cc} column 7.
 \begin{table}[H]
\centering
\begin{tabular}{|c|c|c|c|c|c|c|}
  \hline
  % after \\: \hline or \cline{col1-col2} \cline{col3-col4} ...
  Load regime   & Version 1 performance with CIs &   Version 2 performance with CIs
  \\\hline
  I.L. &$0.671\pm 0.001$&       $0.671\pm 0.001$ \\\hline
  B.L. &  $0.844\pm 0.004$&    $0.844\pm 0.004$\\\hline
  I.M. & $2.084\pm0.011$ &       $2.085\pm0.011$ \\\hline
  B.M. &  $2.833\pm 0.010$ &      $2.832\pm 0.010$\\\hline
  I.H. & $10.014\pm0.055$ &   $9.998\pm0.054$\\\hline
  B.H. &   $16.513\pm 0.140$ &       $16.480\pm 0.137$ \\

  \hline

\end{tabular}

 \caption[]{Average number of jobs per unit time in the discrete-time MDP model of the criss-cross network under PPO policies.}\label{tab:cc2}%
\end{table}

%In the next section we test the robustest of Algorithm \ref{alg2} to the size of the queueing network.

\subsection{Extended  six-class queueing network}\label{sec:ext}

In this subsection, we consider the family of extended six-class networks from \cite{Bertsimas2015} and apply Algorithm \ref{alg2} to find good control policies.

Figure \ref{fig1} shows the structures of the extended six-class networks.
%Job classes 1 and 3 arrive externally  to the network according to a Poisson process with a rate $\lambda_1$ and $\lambda_2$, respectively. Then jobs from class 1 and class 3 follow two separate routes. Jobs from class 3 are sequentially processed in each of $L$ servers and then leave the network.  Jobs from class 1 after being processed in each server fed back into 2nd buffer associated with the first server. After the jobs of class 2 are served in each server again and then leave the system.
We run experiments for 6 different extended six-class queueing networks with the following traffic parameters: $\lambda_1 = \lambda_3 = 9/140$, the service times are exponentially distributed with service rates determined  by the modulus after dividing the class index  by 6 (i.e. classes associated with server 1 are served with rates $\mu_1 = 1/8$, $\mu_2 = 1/2$, $\mu_3 = 1/4$ and classes associated with server 2 are processed with service rates $\mu_4 = 1/6$, $\mu_5 = 1/7$, $\mu_6 = 1$). The service rates
for the odd servers $S_1, ..., S_{\lfloor L/2\rfloor+1}$ are the same as the service rates for server 1, while the service rates for the even
servers $S_2, ..., S_{\lfloor L/2 \rfloor}$ are the same as the service rates for server 2. The load is the same for each station and is equal to $\rho = 0.9$.

\begin{figure}[H]
  \begin{center}
 \begin{tikzpicture}[server/.style={rectangle, inner sep=0.0mm, minimum width=.8cm,
     draw=green,fill=green!10,thick},
  buffer/.style={rectangle, rounded corners=3pt,
    inner sep=0.0mm, minimum width=.9cm, minimum height=.6cm,
    draw=orange,fill=blue!10,thick}]

  \node[server,minimum width=.5in, minimum height=2in] at (6,4) (S1)
  {$S_1$};

  \node[server,minimum width=.5in, minimum height=2in] (S2)
  [right=.75in of S1.east] {$S_2$};

\node[server,minimum width=.5in, minimum height=2in] (S3)
  [right=3.0in of S1.east] {$S_L$};

  \node[buffer] (B1) [left=.2in of $(S1.west)!.5!(S1.north west)$ ] {$B_1$};
  \node[buffer] (B2) [left=.2in of S1.west ] {$B_2$};
  \node[buffer] (B3) [left=.2in of $(S1.west)!.5!(S1.south west)$ ] {$B_3$};
  \node[buffer] (B4) [left=.2in of $(S2.west)!.5!(S2.north west)$ ] {$B_4$};
  \node[buffer] (B5) [left=.2in of S2.west ] {$B_5$};
  \node[buffer] (B6) [left=.2in of $(S2.west)!.5!(S2.south west)$ ] {$B_6$};

  \node[buffer] (B7) [left=.2in of $(S3.west)!.5!(S3.north west)$ ] { $B_{3(L-1)+1}$ };
  \node[buffer] (B8) [left=.2in of S3.west ] {$B_{3(L-1)+2}$};
  \node[buffer] (B9) [left=.2in of $(S3.west)!.5!(S3.south west)$ ] {$B_{3(L-1)+3}$};

\coordinate (source1) at ($ (B1.west) + (-.2in, +.2in)$);
\coordinate (source3) at ($ (B3.west) + (-.2in, -.2in)$);
  \coordinate (departure1) at ($ (B8-|S3.east) + (+.8in, +.0in)$);
 \coordinate (departure3) at ($ (B9-|S3.east) + (+.8in, 0in)$);
  \coordinate (S3minus) at ($(B7-|S3.east)+(0.8in,.7in)$);
  \coordinate (B2plus) at ($(B2.east)+(-1in,0in)$);
 \coordinate (L11) at ($(B7-|S2.east)+(+.2in,0in)$);
 \coordinate (L12) at ($(B7.west)+(-.2in,0in)$);
 \coordinate (L21) at ($(B8-|S2.east)+(+.2in,0in)$);
 \coordinate (L22) at ($(B8.west)+(-.2in,0in)$);
 \coordinate (L31) at ($(B9-|S2.east)+(+.2in,0in)$);
 \coordinate (L32) at ($(B9.west)+(-.2in,0in)$);

  \path[->, thick]
 (B2.east) edge (B2-|S1.west)
 (B3.east) edge (B3-|S1.west)
(B1.east) edge (B1-|S1.west)
(B4.east) edge (B4-|S2.west)
 (B5.east) edge (B5-|S2.west)
(B6.east) edge (B6-|S2.west)
(B7.east) edge (B7-|S3.west)
 (B8.east) edge (B8-|S3.west)
(B9.east) edge (B9-|S3.west)
 (B2plus) edge (B2.west)
(B4-|S1.east) edge (B4.west)
(B5-|S1.east) edge (B5.west)
(B6-|S1.east) edge (B6.west);

\draw[thick, dashed]  (L11) -- (L12);
\draw[thick,->] (L12) |- (B7.west);
\draw[thick] (L11-|S2.east) |- (L11);
\draw[thick, dashed]  (L21) -- (L22);
\draw[thick,->] (L22) |- (B8.west);
\draw[thick] (L21-|S2.east) |- (L21);
\draw[thick, dashed]  (L31) -- (L32);
\draw[thick,->] (L32) |- (B9.west);
\draw[thick] (L31-|S2.east) |- (L31);
\draw [thick,->] (source1) |- (B1.west);
\draw [thick,->] (B8-|S3.east) |-(departure1);
\draw [thick,->] (source3) |- (B3.west);
\draw [thick,->] (B9-|S3.east) |- (departure3);
  \draw [thick] (S3minus) -| (B2plus);
\draw [thick] (B7-|S3.east) -| (S3minus);
  \draw [thick] (B8-|S3.east)  |-(departure1);

\node [above,align=left] at (source1.north) {class 1\\ arrivals};
\node [below,align=left] at (departure1.south) {class $3(L-1)+2$\\ departures};
\node [below,align=left] at (source3.south) {class 3\\ arrivals};
\node [below,align=left] at (departure3.south) {class $3(L-1)+3$\\ departures};

\end{tikzpicture}
  \end{center}
  \caption{The extended six-class queueing network.}
  \label{fig1}
\end{figure}
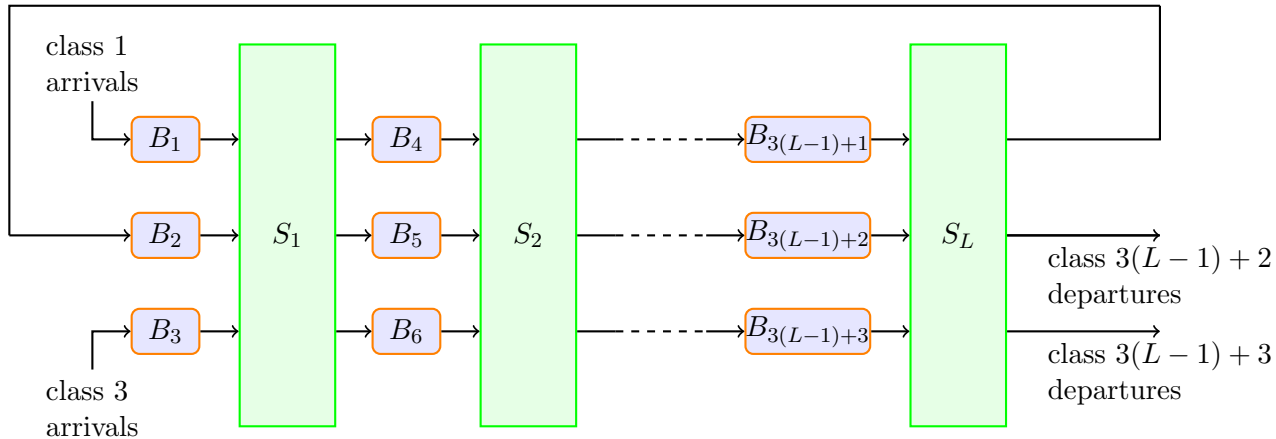

\begin{table}[tbh]
\centering%
\begin{tabular}{|c|c|c|c|c|c|}
  \hline
  % after \\: \hline or \cline{col1-col2} \cline{col3-col4} ...
  No. of classes $3$L  & LBFS & FCFS & FP & RFP & PPO (Algorithm \ref{alg2}) with CIs\\\hline
  6 & 15.749 & 40.173 & 15.422 & 15.286 & $14.130\pm 0.208$\\\hline
  9 & 25.257 & 71.518 & 26.140 & 24.917& $23.269\pm0.251$ \\\hline
  12  & 34.660 & 114.860 & 38.085 & 36.857& $32.171\pm0.556$ \\\hline
  15  & 45.110  & 157.556  & 45.962 & 43.628& $39.300\pm0.612$  \\\hline
  18  & 55.724 & 203.418 & 56.857  & 52.980  & $51.472\pm 0.973$  \\\hline
  21  & 65.980 & 251.657 & 64.713 & 59.051 & $55.124\pm 1.807$\\
  \hline
\end{tabular}
\caption[]{Numerical results for the extended six-class queueing network in Figure \ref{fig1}.}
\label{tab:tab6extRes}
\end{table}

We vary the size of the network between 6 and 21 classes to test the robustness of the PPO policies. In all experiments we generate $Q=50$ episodes with $N=50,000$ timesteps.  Table \ref{tab:rt} in Appendix Section \ref{sec:par} reports the running time of the algorithm depending on the size of the queueing network. We set the discount factor and TD parameter  to $\gamma=0.998$ and $\lambda = 0.99$, respectively.
 Table \ref{tab:tab6extRes} shows the performance of the PPO policy and compares it with other heuristic methods for the extended six-class queueing networks.  In the table FP and RFP refer to fluid and robust fluid policies \cite{Bertsimas2015}.    Table \ref{tab:tab6extRes}  reports the performance of the best RFP corresponding to the best choice of policy  parameters for each class.
LBFS refers to the last-buffer first-serve policy,
where the priority at a server is given to jobs with highest index.  FCFS refers to the first-come first-serve policy, where the priority at a server is given to jobs with the longest waiting time for service.

 For each  extended six-class network we consider a corresponding discrete-time MDP.
We fix a stable, randomized policy for the discrete-time MDP,  since the use of Xavier initialization
could yield an unstable NN policy  for extended six-class networks. We refer to this  stable, randomized policy as an expert policy.
We simulate a long episode of the MDP operating under the expert policy.  At each timestep we save the state at the time
and the corresponding probability distribution over the actions.
We use this simulated data set to train the  initial NN policy $\pi_{\theta_0}$.
 In our numerical experiments, we use the \textit{proportionally  randomized (PR)}  policy as the expert policy.
 If the network operates under the PR policy, when an arrival or service completion event occurs  at station $k$, a nonempty buffer $j$ receives a
priority over other classes at the station with
probability
\begin{align}\label{eq:PRPprob}
\frac{x_j}{\sum\limits_{i\in \B(k)}  x_i},
\end{align}
 where  $\B(k)$ is a set of buffers associated with server $k$ and $x = (x_1,..., x_J)$ is the vector of jobcounts at the time of the event   after accounting for job arrivals and departures.
  The priority  stays fixed   until the next arrival or service completion event occurs.
  The PR policy is  maximally stable for open MQNs, meaning that if the system is unstable under the PR policy, no other policy can stabilize it. See Appendix Section \ref{sec:PR} for the details.  An initial NN policy  for  PPO algorithm plays an important role in its learning process.  PPO algorithm may suffer from an
exploration issue when the initial NN policy  is sufficiently far from the optimal one \cite{Wang2019}.

% \begin{remark}
% The proportionally randomized policy is a modification of HLPPS
%   (head-of-line proportional processor sharing) policy, which allows processor sharing and all nonempty buffers receive  service
% simultaneously proportionally to its queue size -- the fraction of service capacity allocated to class $j$ equals to (\ref{eq:PRPprob}).
%  The HLPPS policy is \textit{maximally stable} for MQNs \cite{Bramson1996}, but, unfortunately, the HLPPS  is not in the set of policies we want to optimize  over.
% \end{remark}

In  each plot in  Figure \ref{fig:ext_ac}, we save policy NN
parameters $\{\theta_i\}_{i=0, 10, ..., 200}$ every 10th policy
iteration.   For each saved policy NN, we conduct a separate long
simulation of the queueing network operating under the policy for accurate performance evaluation by providing a $95\%$
confidence interval of the long-run average cost.   For any of the six queueing networks, when the
load is high, the regeneration is rare.  Thus, we adopt the \textit{batch means} method to estimate the confidence interval
\cite[Section 6]{Henderson1997}.  For each policy from the set
$\{ \pi_{\theta_i}: i = 0, 10, ..., 200 \}$, we simulate an episode
starting from an empty state $x = (0,...,0)$ until $5\times10^6$
arrival events occur.  Then we estimate average performance of the
policy based on this episode. To compute the confidence interval from
the episode, we split the episode into 50 sub-episodes (batches), see
also \cite{Nelson1989}. Pretending that the obtained 50 mean estimates
are i.i.d., we compute the $95\%-$confidence intervals as shown
in Figure~\ref{fig:ext_ac}.
\begin{figure}[tbh]
    \subfloat[6-classes network \label{subfig-6:IL}]{%
       \includegraphics[height=0.2\textwidth, width=0.35\textwidth]{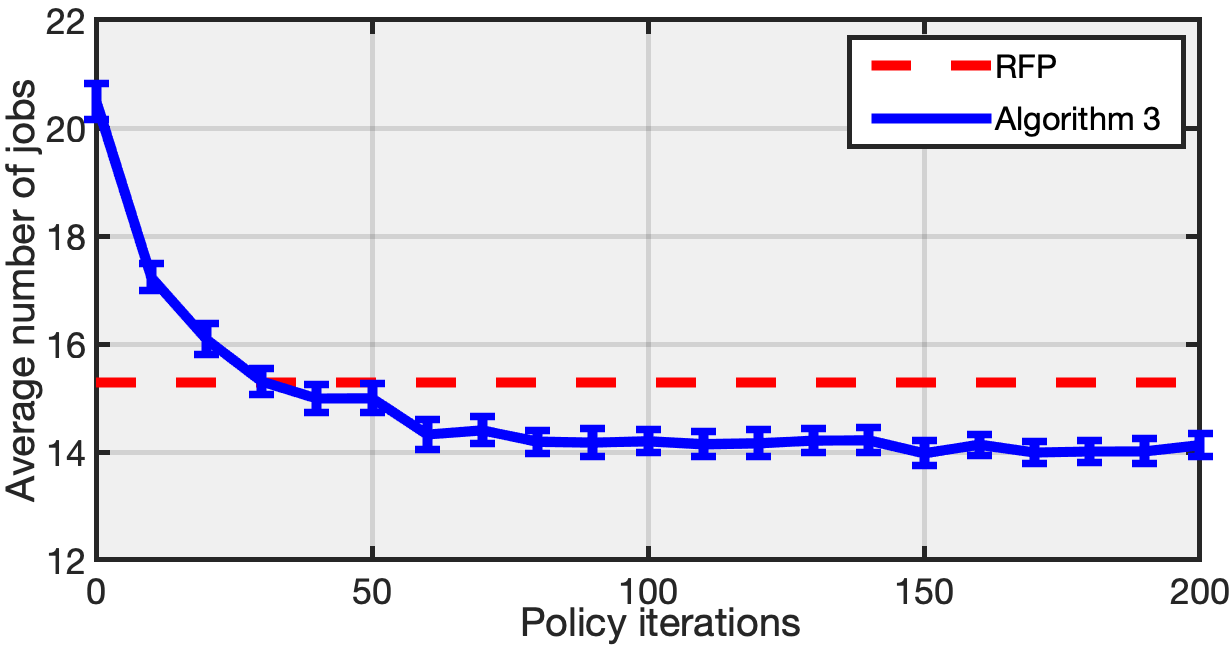}
     }
     \subfloat[9-classes network \label{subfig-4:IM}]{%
       \includegraphics[height=0.2\textwidth, width=0.35\textwidth]{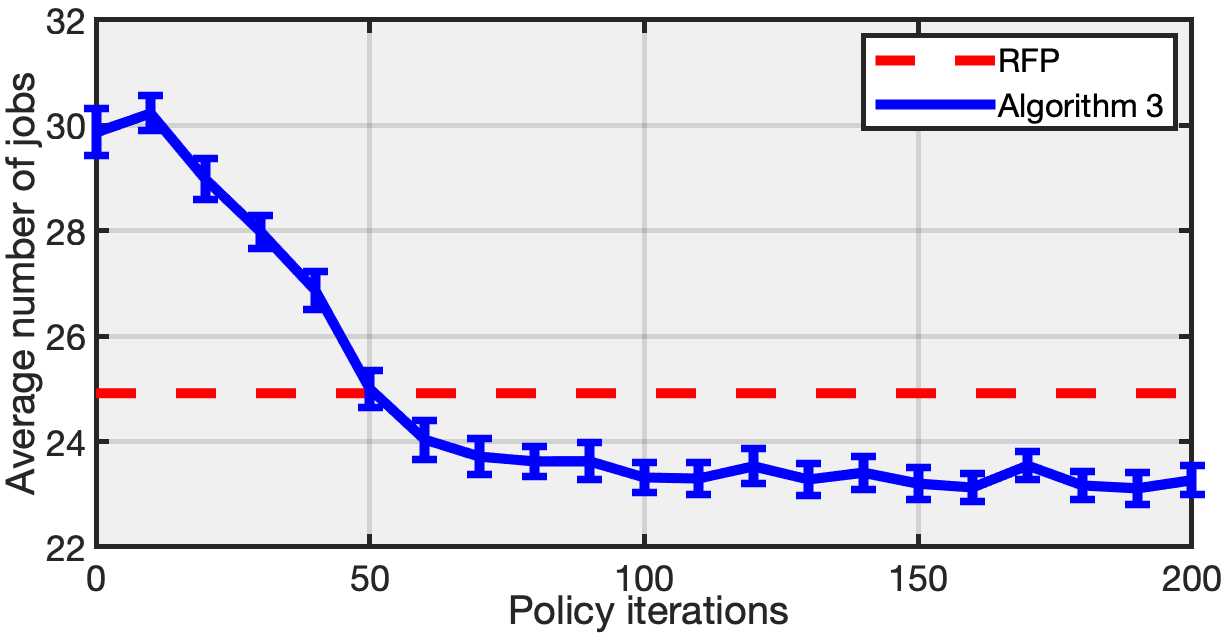}
     }
     \subfloat[12-classes network \label{subfig-2:IH}]{%
       \includegraphics[height=0.2\textwidth, width=0.35\textwidth]{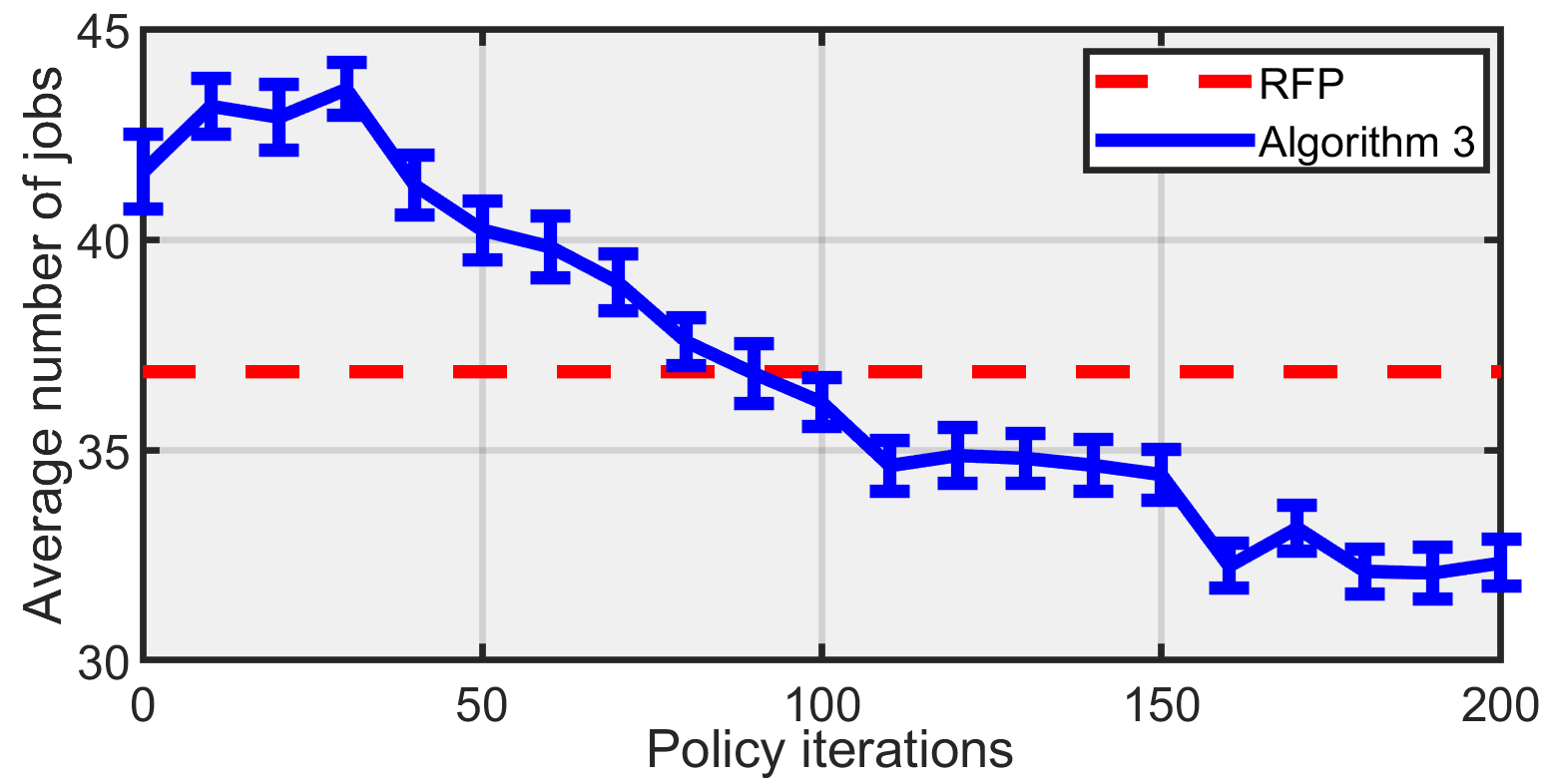}
     }\\
 \subfloat[15-classes network\label{subfig-5:BL}]{%
       \includegraphics[ height=0.2\textwidth, width=0.35\textwidth]{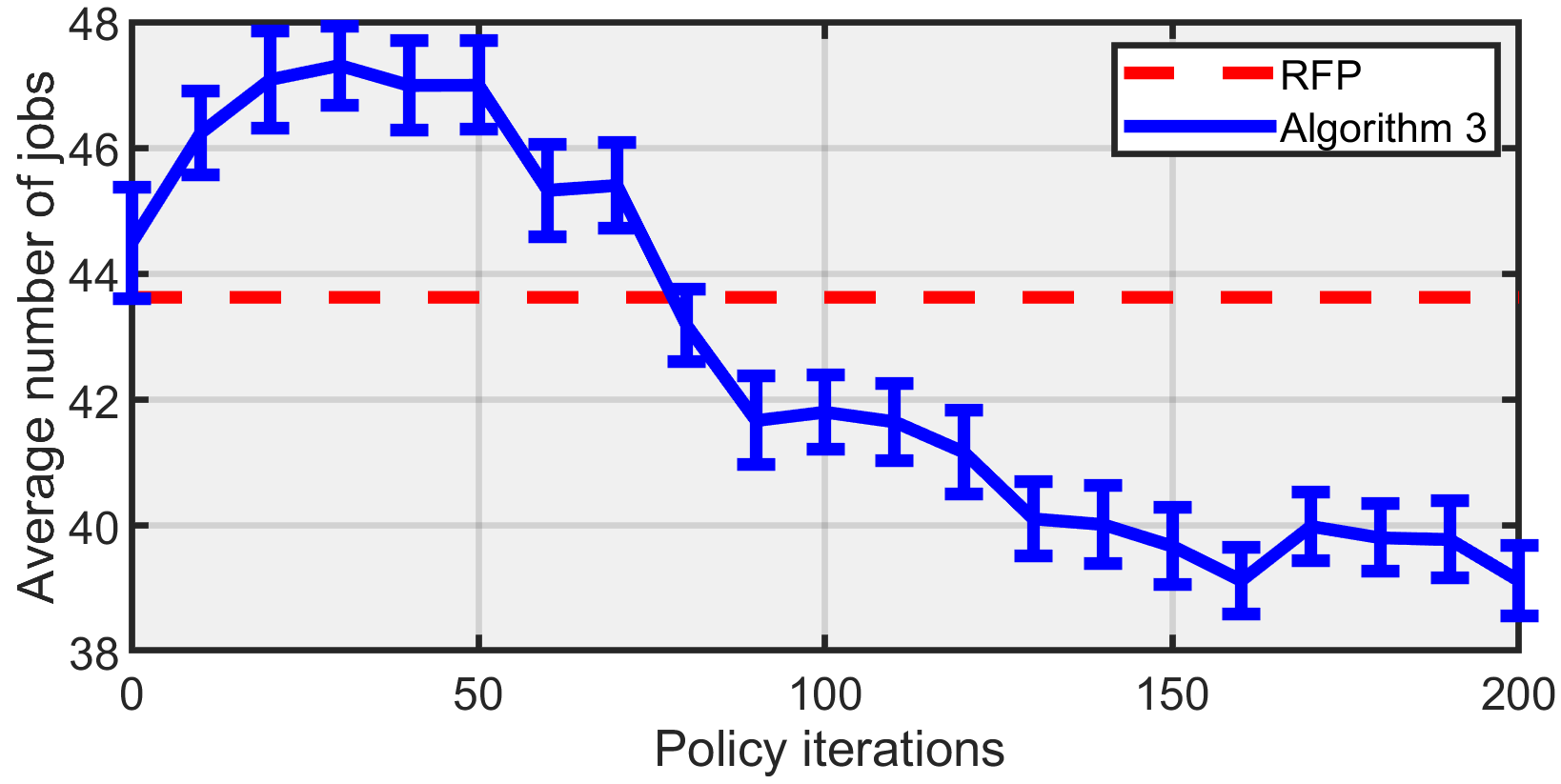}
     }
  \subfloat[18-classes network\label{subfig-3:BM}]{%
       \includegraphics[ height=0.2\textwidth, width=0.35\textwidth]{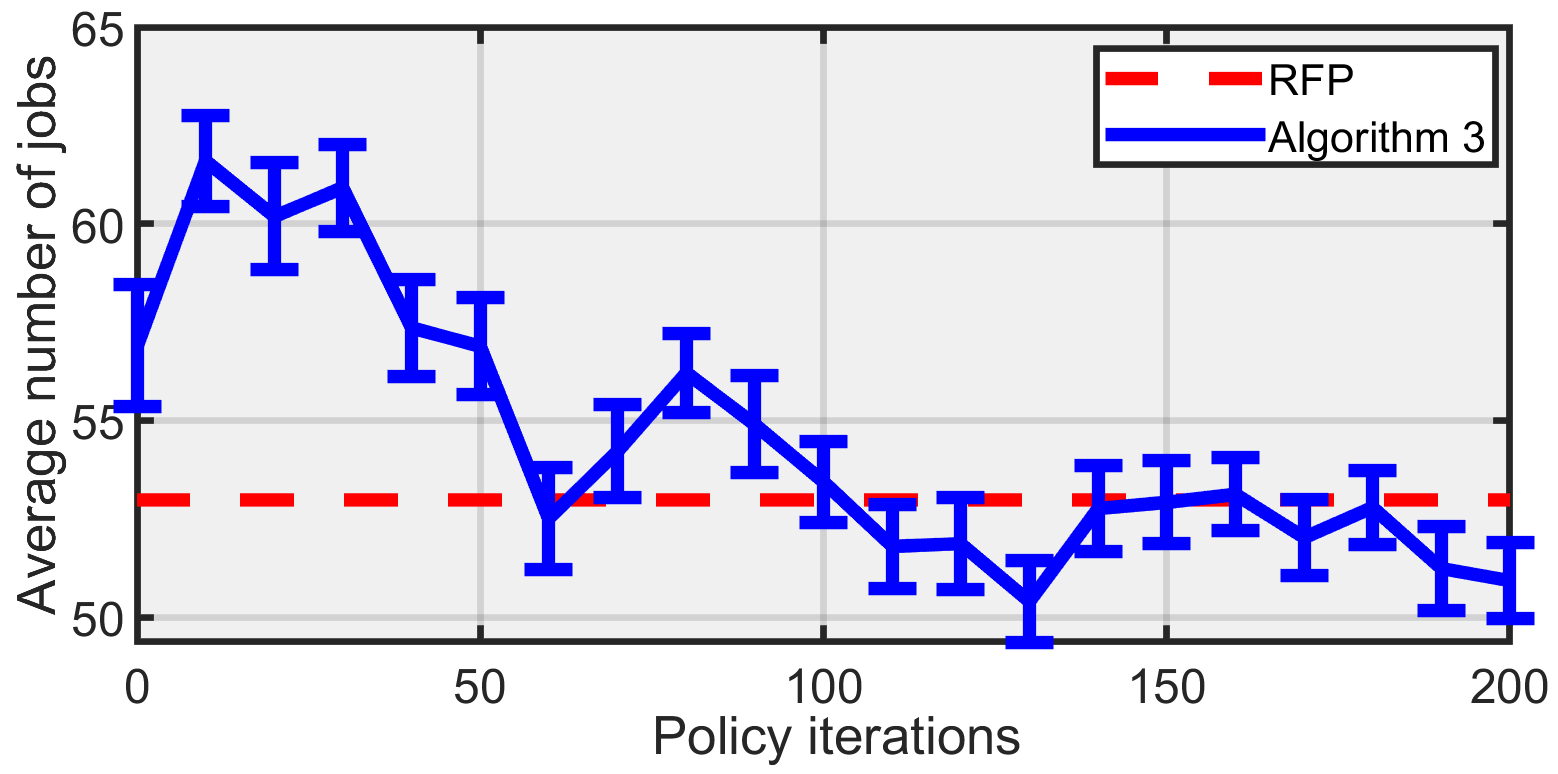}
     }
     \subfloat[21-classes network\label{subfig-1:BH}]{%
       \includegraphics[ height=0.2\textwidth, width=0.35\textwidth]{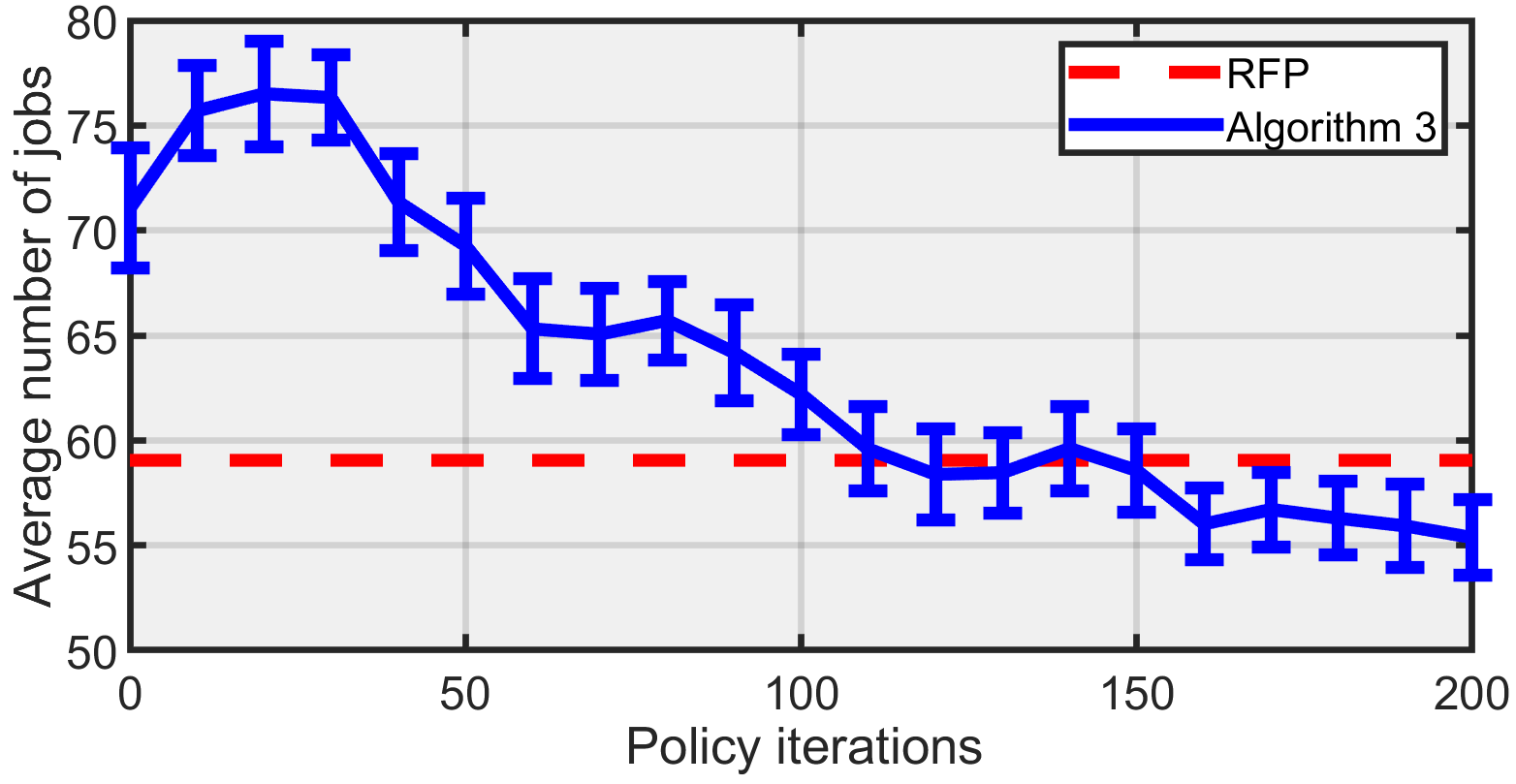}
     }
\caption{Performance of Algorithm~\ref{alg2} on six queueing networks. The solid   blue lines show the performance of  the PPO policies obtained at the end of every 10th iterations of Algorithm \ref{alg2}; the dashed red lines show the performance of the robust fluid policy (RFP).   }
%\caption{Learning curves  from Algorithm \ref{alg2} for the 6-class extended networks.     \textit{The blue solid line} shows the performance of  PPO policies obtained at the end of every 10th
% iterations of Algorithm \ref{alg2},  \textit{red dash line} --performance of the robust fluid policy (RFP).}
     \label{fig:ext_ac}
   \end{figure}

In Section \ref{sec:ge} we discussed the relationship between the
GAE   (\ref{eq:GAE}) and  AMP  (\ref{eq:esf})  estimators.
 Figure \ref{fig:AMPvsGAE}  illustrates the benefits of using the AMP method in Algorithm \ref{alg2} on the 6-classes network.
The learning curve for the GAE estimator is obtained by replacing the value function estimation in line 6 of Algorithm \ref{alg2} with the  GAE estimator (\ref{eq:GAE}).

\begin{figure}[H]
\centering%
\includegraphics[width=.5\linewidth]{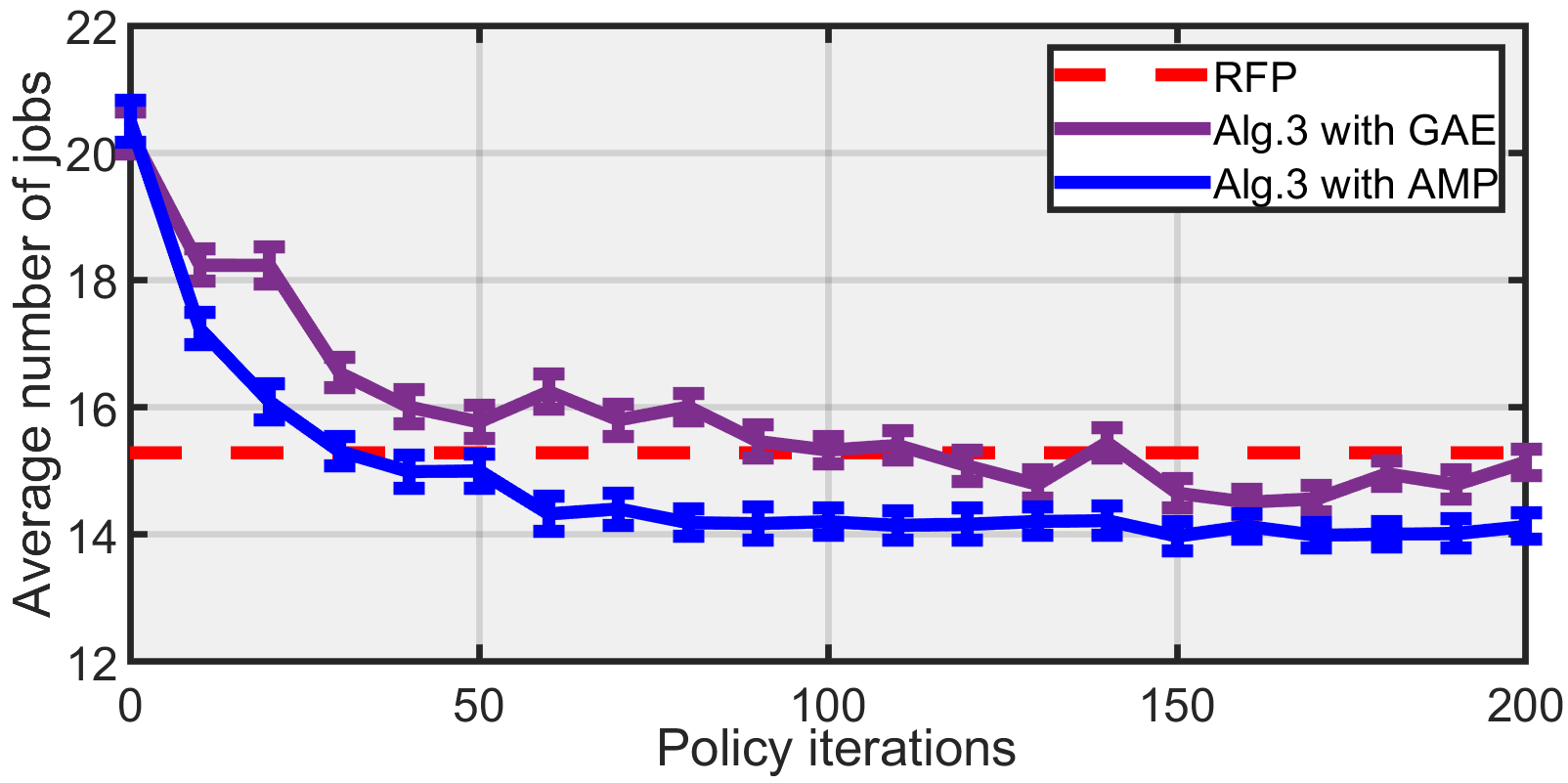}
\caption[]{Learning curves from Algorithm \ref{alg2} for the 6-class network.
The solid blue and purple lines show the performance of  the PPO policies obtained from Algorithm \ref{alg2} in which
the value function estimates are computed by the AMP method   and by the GAE method, respectively; the dashed red line shows the performance of the robust fluid policy (RFP).
}
\label{fig:AMPvsGAE}%
\end{figure}

\subsection{Parallel servers network}\label{sec:nmodel}
In this section we  demonstrate that PPO Algorithm~\ref{alg2} is also effective for a stochastic processing network that is outside the model class of multiclass queueing networks.
Figure \ref{fig:Nmodel}  shows a  processing network system  with two independent Poisson input arrival flows, two servers, exponential
service times, and linear holding costs. Known as the  \textit{N-model network}, it first appeared in \cite{Harrison1998}.

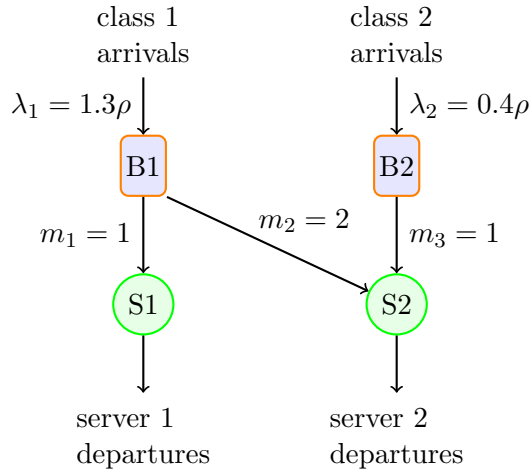
\begin{figure}[H]
  \centering
\begin{tikzpicture}[inner sep=1.5mm] % N-model
  \node[server] at (6,4) (S1)  {S1};
  \node[server] (S2) [right=1in of S1] {S2};
  \node[vbuffer] (B1) [above=.4in of S1.north] {B1};
  \node[vbuffer] (B2) [above=.4in of S2.north] {B2};
  \coordinate (source1) at ($ (B1.north) + (0in, .3in)$) {};
  \coordinate (source2) at ($ (B2.north) + (0in, .3in)$) {};
  \coordinate (departure1) at ($ (S1.south) + (0in, -.3in)$) {};
  \coordinate (departure2) at ($ (S2.south) + (0in, -.3in)$) {};

 \draw[thick,->]    (S1) -- (departure1);
 \draw[thick,->]    (S2) -- (departure2);

  \path[->, thick]
  (source1) edge node [left] {$\lambda_1=1.3\rho$} (B1)
  (source2) edge node [right] {$\lambda_2=0.4\rho$} (B2)
   (B1) edge node [left] {$m_1=1$} (S1.north)
   (B1.south east) edge node [right,near start] {\quad $m_2=2$} (S2)
   (B2) edge node [right] {$m_3=1$} (S2.north);

   \node [inner sep=2.0mm,below,align=left] at (departure1.south east) {server 1 \\ departures};
   \node [inner sep=2.0mm,below,align=left] at (departure2.south west) {server 2\\ departures};
   \node [inner sep=2.0mm,above,align=left] at (source1.north) {class 1\\ arrivals};
   \node [inner sep=2.0mm, above,align=left] at (source2.north) {class 2\\ arrivals};
 \end{tikzpicture}
 \caption{N-model network}
  \label{fig:Nmodel}
\end{figure}

Jobs of class $i$ arrive according to a Poisson process at an average
rate of $\lambda_i$ such that $\lambda_1 = 1.3\rho$ and $\lambda_2 = 0.4\rho$ per unit time, where $\rho=0.95$ is a parameter that specifies the traffic intensity. Each job requires a single service
before it departs, and class 1 can be processed by either server 1 or server 2, whereas
class 2 can be processed only by server 2.   The jobs are processed by three different activities:

\begin{center}
activity 1 = processing of class 1 jobs by server 1,

activity 2 = processing of class 1 jobs by server 2,

activity 3 = processing of class 2 jobs by server 2.
\end{center}

We assume that  the  service at both servers is preemptive and work-conserving; specifically  we assume that activity 2 occurs only if there is at least one class 1 job in the system.

The service times for activity $i$ are exponentially distributed with mean $m_i$, where $m_1=m_3=1$ and $m_2=2$.
 The holding costs are continuously incurred at a rate of $h_j$ for each class $j$ job that remains within the system, with the specific
numerical values $h_1 = 3$ and $h_2 = 1.$ We note that all model parameter values   correspond to those in \cite{Harrison1998}.

We define $x = (x_1, x_2)\in \X$ as a system state, where $x_j$ is number of class $j$ jobs in the system. We use uniformization to convert the continuous-time control problem to a discrete-time control problem.
Under control $a = 1$ (class 1 has preemption high priority for  server 2) the transition probabilities are given by

\begin{align*}
&P\left((x_1+1, x_2)|(x_1, x_2)\right)= \frac{\lambda_1}{\lambda_1+\lambda_2+\mu_1+\mu_2+\mu_3},\\
&P\left((x_1, x_2+1)|(x_1, x_2)\right) = \frac{\lambda_2}{\lambda_1+\lambda_2+\mu_1+\mu_2+\mu_3},\\
&P\left((x_1-1, x_2)|(x_1, x_2),a= 1\right) = \frac{\mu_1\I_{\{x_1>0\}} + \mu_2\I_{\{x_1>1\}}}{\lambda_1+\lambda_2+\mu_1+\mu_2+\mu_3},\\
&P\left((x_1, x_2-1)|(x_1, x_2), a=1\right) = \frac{\mu_3\I_{\{x_2>0, x_1\leq 1\}}}{\lambda_1+\lambda_2+\mu_1+\mu_2+\mu_3},\\
&P\left((x_1, x_2 )|(x_1, x_2) \right) = 1 - P\left((x_1+1, x_2)|(x_1, x_2)\right)  - \\
 &\quad \quad -P\left((x_1, x_2+1)|(x_1, x_2)\right) -P\left((x_1-1, x_2)|(x_1, x_2) \right) - P\left((x_1, x_2 )|(x_1, x_2), 1\right),
\end{align*}
where $\mu_i  = 1/m_i$, $i=1, 2, 3.$

 Under control $a=2$ (class 2 has high priority), the only changes of transition probabilities are
 \begin{align*}
&P\left((x_1-1, x_2)|(x_1, x_2), a=2\right) = \frac{\mu_1\I_{\{x_1>0\}} + \mu_2\I_{\{x_1>1, x_2=0\}}}{\lambda_1+\lambda_2+\mu_1+\mu_2+\mu_3},\\
&P\left((x_1, x_2-1)|(x_1, x_2), a=2\right) = \frac{\mu_3\I_{\{x_2>0\}}}{\lambda_1+\lambda_2+\mu_1+\mu_2+\mu_3}.\\
\end{align*}

We define the cost-to-go function   as $g(x) := h_1x_1 +h_2x_2 = 3x_1+x_2.$  The objective is to find policy $\pi_\theta, \theta\in \Theta$ that minimizes the long-run average holding costs
\begin{align*}
\lim\limits_{N\rightarrow \infty} \frac{1}{N}  \E\left[\sum_{k=0}^{N-1}g (x^{(k)})\right],
\end{align*}
where $x^{(k)}$ is the system state after $k$ timesteps.

We use   Algorithm \ref{alg2} to find a near-optimal policy. Along with a learning curve from Algorithm \ref{alg2}, Figure \ref{fig:NmodelCurves}  shows the performance the best threshold policy with $T=11$ and the optimal policy. The threshold policy was proposed in \cite{Bell2001}. Server 2 operating under the  threshold policy  gives priority to class 1 jobs, if the number of class 1 jobs in the system is larger than a fixed threshold $T$.

 \begin{figure}[H]
\centering%
\includegraphics[width=.5\linewidth]{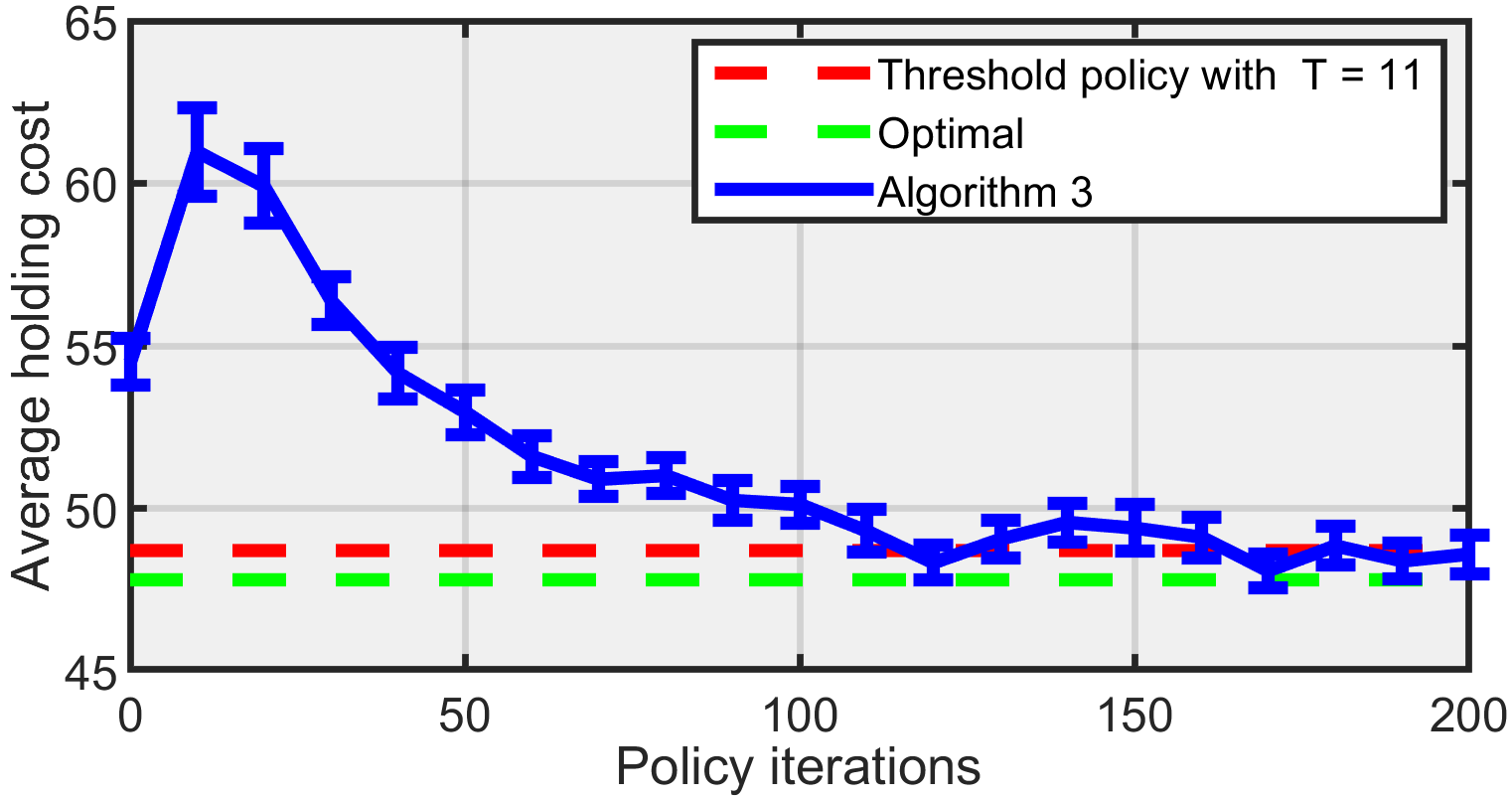}
\caption[]{Learning curves from Algorithm \ref{alg2} for the N-model network.  The blue solid line shows the performance of  PPO policies obtained from Algorithm \ref{alg2};       the dashed red line shows the performance of the threshold policy with $T=11$; the dashed greed line shows the performance of the optimal policy.}
\label{fig:NmodelCurves}%
\end{figure}

In Figure \ref{fig:NmodelPolicies} we show the control of randomized PPO policies obtained after 1, 50, 100, 150, and 200 algorithm iterations.  For each policy  we depict the probability distribution over two possible actions for states $x\in \X$ such that $0\leq x_j\leq 50,$ $j=1, 2$, and compare  the PPO, optimal, and  threshold policies.

\begin{figure}[H]
\begin{center}
    \subfloat[after 1 iteration \label{subfig-1:1}]{%
       \includegraphics[height=0.2\textwidth ]{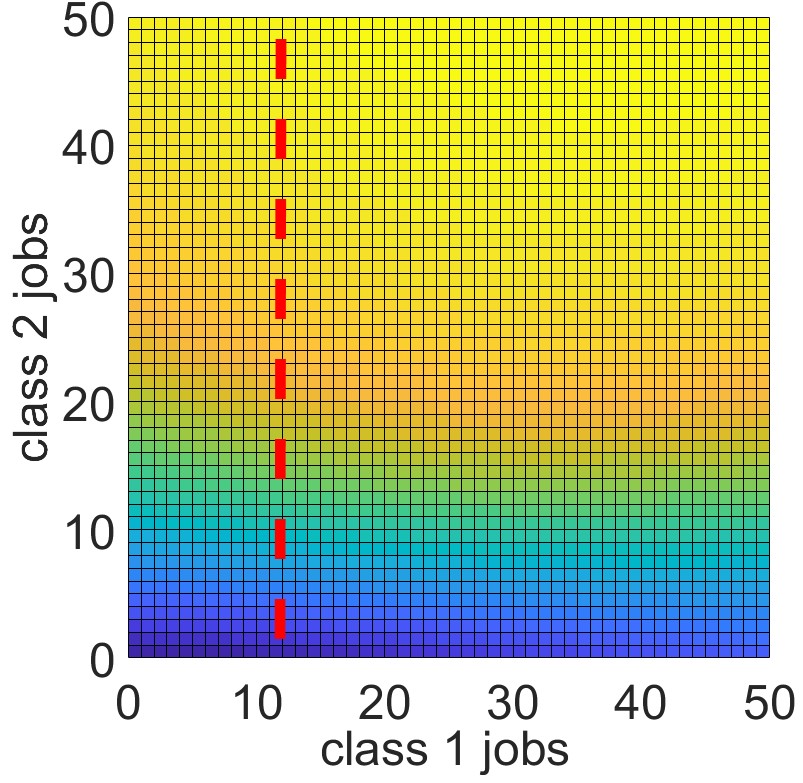}
     }\hspace{8pt}%
     \subfloat[after 50 iterations \label{subfig-2:50}]{%
       \includegraphics[height=0.2\textwidth ]{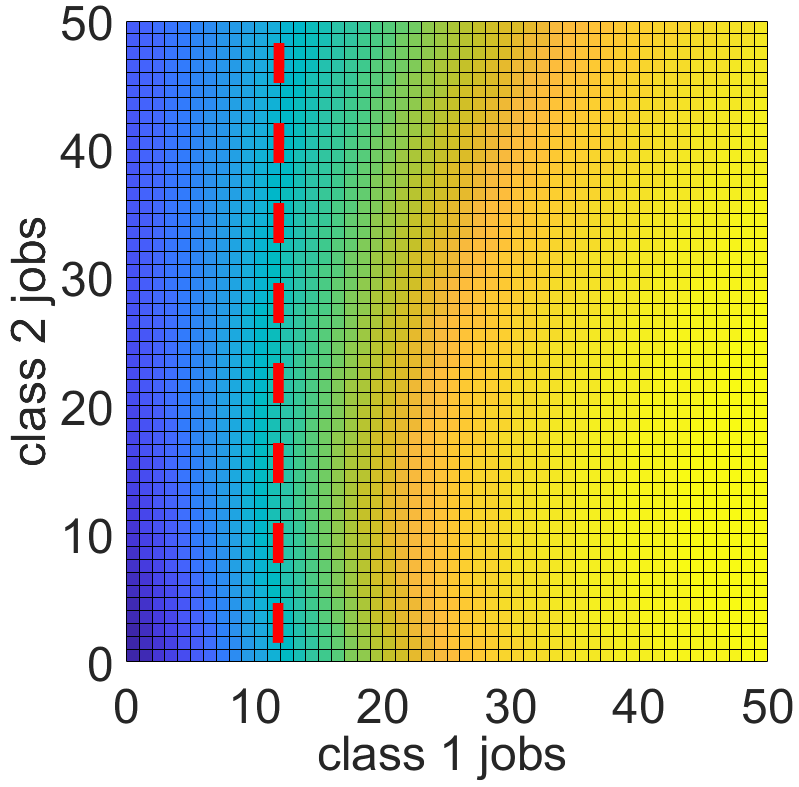}
     }\hspace{8pt}%
     \subfloat[after 100 iterations \label{subfig-3:100}]{%
       \includegraphics[height=0.2\textwidth ]{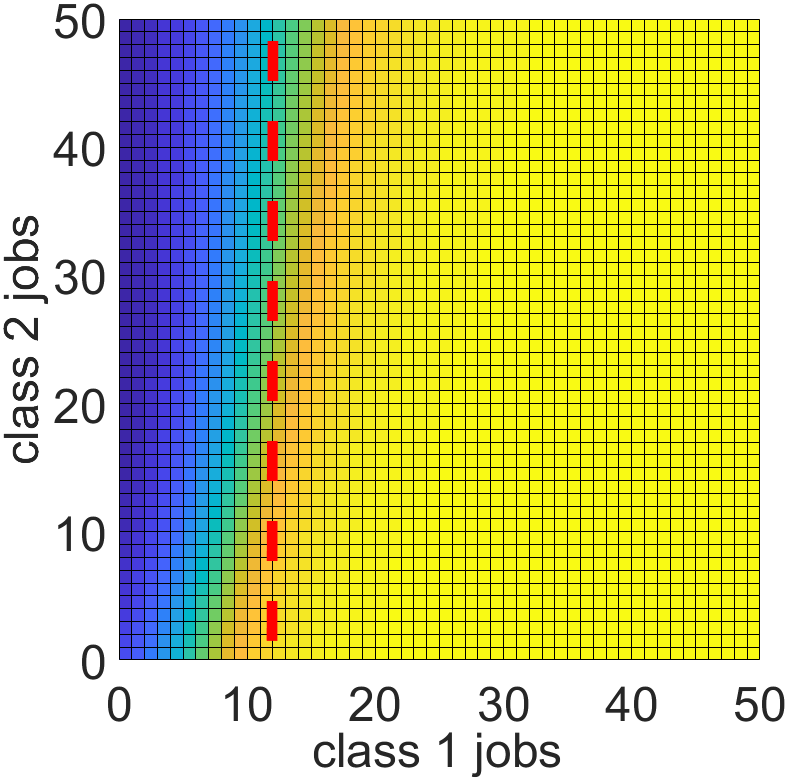}
     }\hspace{8pt}%
 \subfloat[after 150 iterations\label{subfig-4:150}]{%
       \includegraphics[ height=0.2\textwidth ]{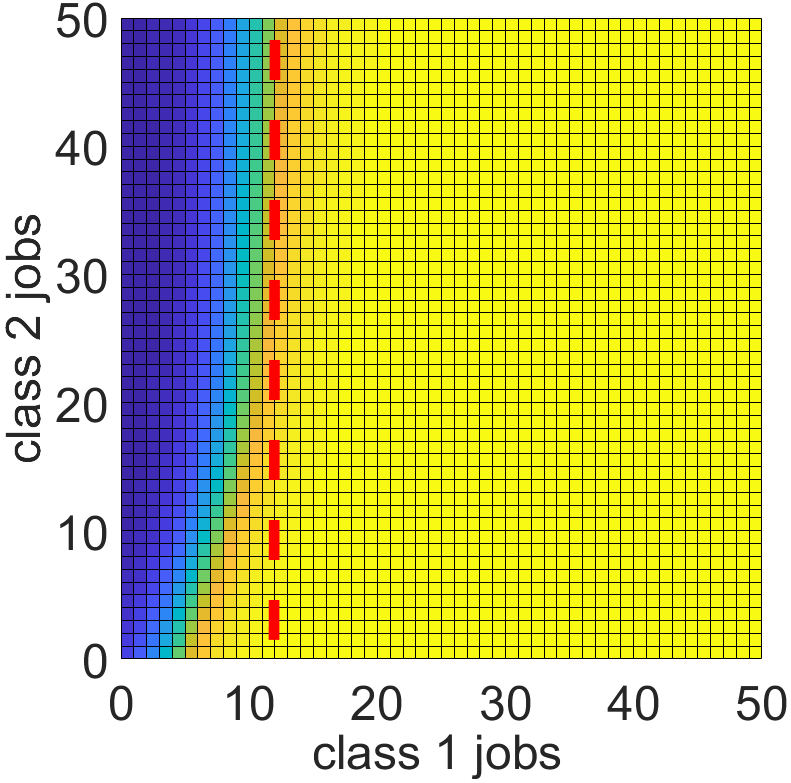}
     }\\
  \subfloat[after 200 iterations\label{subfig-5:200}]{%
       \includegraphics[ height=0.2\textwidth ]{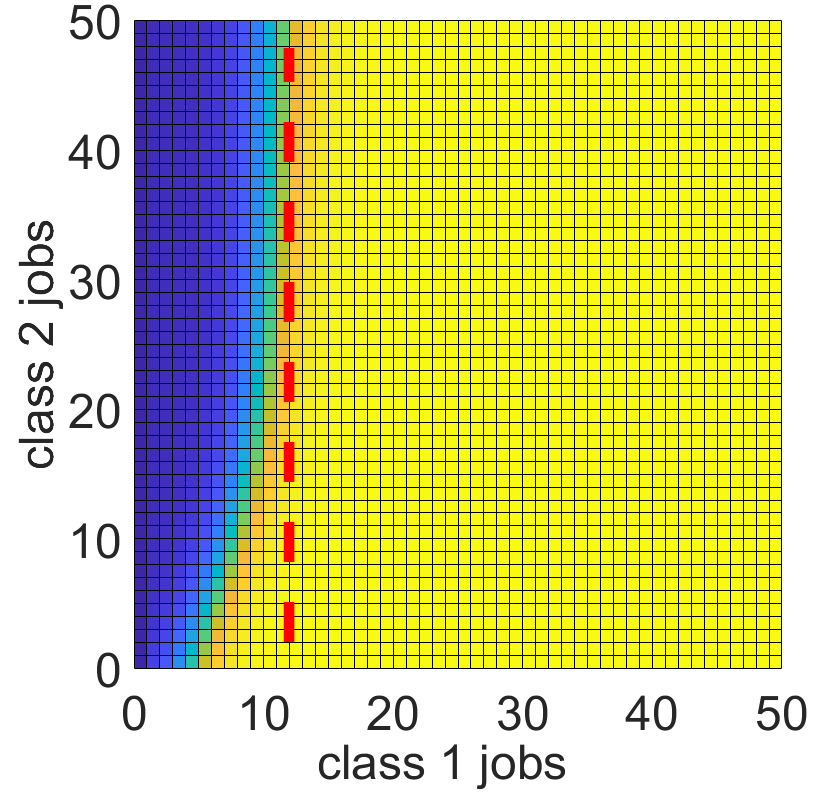}
     }\hspace{8pt}%
     \subfloat[Threshold policy \label{subfig-6:thr}]{%
       \includegraphics[height=0.2\textwidth ]{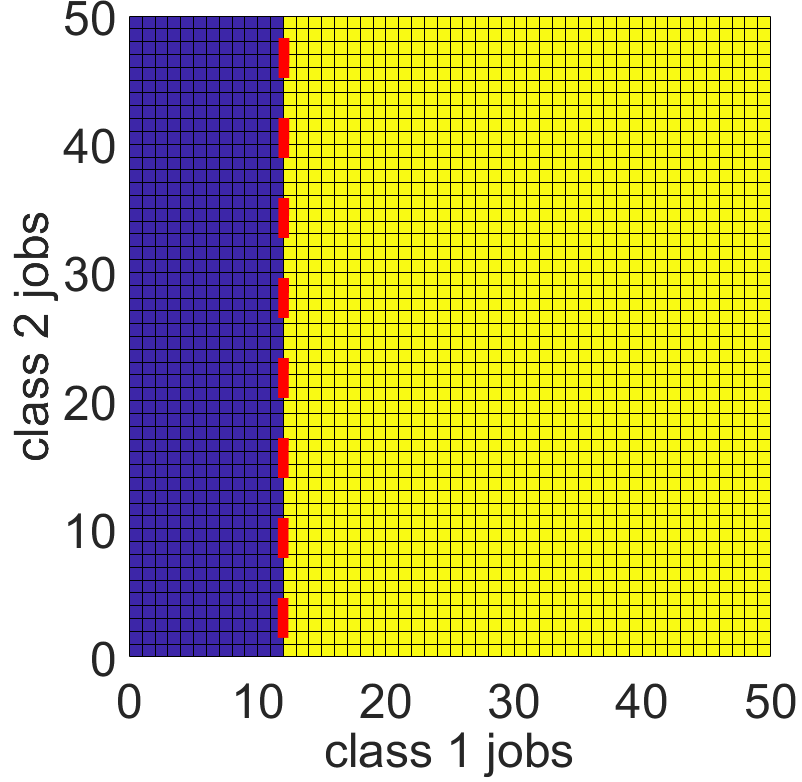}
     }\hspace{8pt}%
\subfloat[Optimal policy \label{subfig-7:opt}]{%
       \includegraphics[height=0.2\textwidth ]{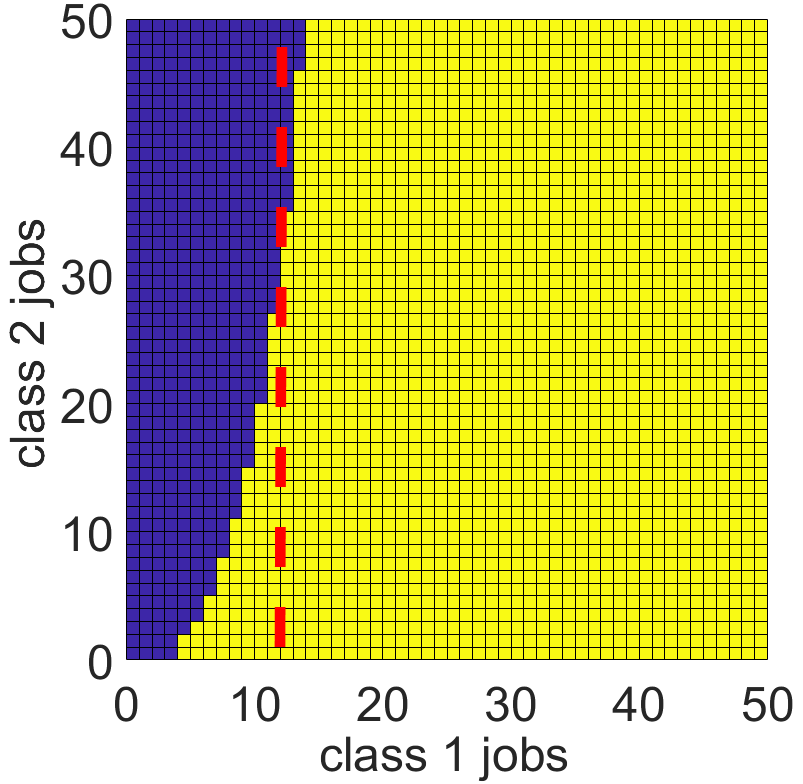}
     }\hspace{8pt}%
     \subfloat {%
       \includegraphics[height=0.2\textwidth ]{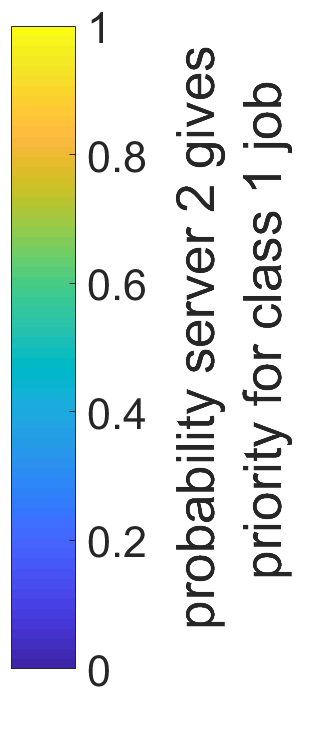}
     }
\end{center}

\caption{Evolution of the PPO policies over the  learning process and comparing them with threshold and optimal policies.  The probability that server 2 gives priority to class 1 jobs is shown by a color gradient for system states that have less than 50 jobs in each buffer:
server 2 gives priority to class 1 jobs (yellow), server 2 gives priority to class 2 jobs (blue), and the dashed red lines represent the threshold policy with $T=11.$
}
     \label{fig:NmodelPolicies}
   \end{figure}

\section{Conclusion}

This paper proposed a method for optimizing the long-run average performance in queueing network control problems. It provided a theoretical justification for extending the PPO algorithm in \cite{Schulman2017} for Markov decision problems with infinite state space,  unbounded costs, and the long-run average cost objective.  Our idea  of applying Lyapunov function approach  has a potential to be adapted for other advanced policy gradient methods. We believe that theoretical analysis and performance comparison   of policy gradient \cite{Marbach2001, Peters2008}, trust-region \cite{Schulman2015}, proximal \cite{Schulman2017}, soft actor-critic \cite{Haarnoja2018}, deep Q-learning \cite{Mnih2015, Hessel2018}, and other policy optimization algorithms  can be of great benefit for the research community and deserve a separate  study.

The success of    the PPO algorithm implementation  largely depends on the accuracy of Monte
  Carlo estimates of the relative value function in each policy
  iteration. Numerical experiments using sets of fixed hyperparameters  showed that introducing an appropriate discount factor has the largest effect on the variance reduction of a value function estimator, even though the discount factor introduces biases,  and that introducing a control variate by using the approximating martingale-process (AMP)  further reduces the variance. Moreover, when system loads vary between light and moderate, regeneration  estimators also reduce  variances. The experiments also implied that   AMP estimation, rather than   GAE estimation, is preferable when the transition probabilities are known. Most of hyperparameters, including the discount factor and length of episodes,  were  chosen by experimental tuning rather than a theoretically justified scheme. It is desirable to investigate how to select the  discount factor,  length of episodes,  and other hyperparameters  based on queueing network load, size, topology.

Our numerical  experiments demonstrated that Algorithm~\ref{alg1amp} applied for the criss-cross network control optimization can obtain policies with long-run average performance within $1\%$ from the optimal.  For large-size networks PPO Algorithm~\ref{alg2} produced effective control policies that either outperform or perform as well as the current alternatives. For an extended six-class queueing network, PPO policies outperform the robust fluid policies on average by $10\%$.  The algorithm can be applied for a processing network control problem if the  processing network admits a uniformization representation under any feasible control policy. A wide class of such processing networks  described in  \cite{DaiHarr2020}.  As an example, we provide the numerical experiment for the N-model network. Although this paper considered only queueing networks with preemptive service, the proposed algorithm can also be applied for queueing networks with non-preemptive service policies. Two modifications  to the MDP formulation in Section \ref{sec:MQN} are required  if  a queueing network operates under non-preemptive service policies. First, a service status of each server should be included into the system state representation along with the  jobcount vector. Second, at each decision time a set of feasible actions should be restricted based on each server status.

Several research questions should be pursued in future  studies. First, future research could examine the necessity  of $\V$-uniform ergodicity  assumption  in Theorem \ref{thm:main}, and whether this assumption can be replaced by $\V$-ergodicity.   Second, one of the key components  of our PPO algorithm for MQNs is the proposed PR expert policy. Further research is needed to design an algorithm that does not require the knowledge of a stable expert policy. Third, our numerical experiments show that the variance reduction techniques are important for a good
performance of the algorithm. Future investigations are desirable to develop more sample efficient simulation methods, potentially, incorporating problem structure knowledge.

Complexity of the queueing control optimization problems highly depends
not only on the network topology, but on the traffic
intensity. As the network traffic intensity
  increases the long-term effects of the actions become less predictable that presents a challenge  for any reinforcement
  learning algorithm.  We believe that multiclass queueing networks should serve as  useful benchmarks for testing reinforcement learning methods.

\section*{Acknowledgment}

This research  was supported in part by  National Science Foundation Grant CMMI-1537795.

\appendix

\section*{Appendix}

\section{Proofs of the theorems in Section \ref{sec:countable}}\label{sec:proofs}

 \begin{proof}[\textbf{Proof of Lemma \ref{lem:Zeq}}]
  We define vector $h:= Z\left(g -(\mu^T g) e\right)$. Matrix $Z$ has a finite $\V-$norm, therefore the inverse matrix of $Z$ is unique and equal to $I-P+\Pi$. Then by definition vector $h$ satisfies
\begin{align}\label{eq:extPois}
(I-P+\Pi) h = g -(\mu^T g) e.
\end{align}

Multiplying both sides of (\ref{eq:extPois}) by $\mu$ we get $\mu^T h = 0$ (and $\Pi h = 0$). Hence, vector $h$ is a solution of the Poisson equation  (\ref{eq:Poisson}) such that $\mu^T h = 0$. It follows from Lemma \ref{lem:poisson_sol} that $h=h^{(f)}.$
 \end{proof}

\begin{proof}[\textbf{Proof of Lemma \ref{lem:st}}]
We denote \begin{align}\label{eq:U}U_{\theta, \eta}: = (P_{\theta} - P_{\eta}) Z_{\eta}\end{align} and define matrix $H_{\theta, \eta}$ as
\begin{align}\label{eq:H}
H_{\theta,\eta} := \sum\limits_{k=0}^\infty U^k_{\theta, \eta}.
\end{align}
The convergence in the $\V$-weighted norm in definition (\ref{eq:H}) follows  from assumption $\|U_{\theta, \eta}\|_\V<1$.

The goal of this proof is to show that the Markov chain  has a unique stationary distribution $\mu_{\theta}$ such that
\begin{align}\label{eq:mu12}
\mu_{\theta}^T = \mu_{\eta}^T H_{\theta,\eta}.
\end{align}

We let $\nu^T :=  \mu_{\eta}^T H_{\theta,\eta}$.  We use $e = (1, 1, ...,1,...)^T$ to denote  the unit vector.  First, we  verify that $\nu^T e = \sum\limits_{x\in \X} \nu(x)=1.$ We note that  $Z_{\eta} e = e.$ Then
 \begin{align*}
 \nu ^T e = \mu_{\eta}^T H_{\theta,\eta} e =  \mu_\eta^T \sum\limits_{k=0}^\infty \left( (P_{\theta} - P_{\eta})Z_{\eta}\right)^k e  = \mu_\eta^T I e = 1.
 \end{align*}

  Second, we verify that $\nu^T P_{\theta} = \nu^T.$ We prove it by first assuming
  that \begin{align}\label{eq:Pfinite}
  P_{\theta}- P_{\eta} + U_{\theta, \eta}  P_{\eta} = U_{\theta, \eta}\end{align}  holds. Indeed,
 \begin{align*}
\nu^T P_{\theta} &=\mu_{\eta}^T \sum\limits_{k=0}^\infty U_{\theta, \eta}^k P_{\theta} \\
&=   \mu_{\eta}^T \sum\limits_{k=0}^\infty U_{\theta, \eta}^k P_{\theta}- \mu_{\eta}^T \sum\limits_{k=0}^\infty U_{\theta, \eta}^k P_{\eta} + \mu_{\eta}^T \sum\limits_{k=0}^\infty U_{\theta, \eta}^{k} P_{\eta}\\
& =  \mu_{\eta}^T +   \mu_{\eta}^T \sum\limits_{k=0}^\infty U_{\theta, \eta}^k P_{\theta}- \mu_{\eta}^T \sum\limits_{k=0}^\infty U_{\theta, \eta}^k P_{\eta} + \mu_{\eta}^T \sum\limits_{k=0}^\infty U_{\theta, \eta}^{k+1} P_{\eta} \\
& =  \mu_{\eta} ^T+ \mu_{\eta}^T \sum\limits_{k=0}^\infty U_{\theta, \eta}^k (P_{\theta}- P_{\eta} + U_{\theta, \eta}  P_{\eta})\\
& =\mu_{\eta}^T + \mu_{\eta}^T \sum\limits_{k=0}^\infty U_{\theta, \eta}^{k+1}\\
& =\mu_{\eta}^T \sum\limits_{k=0}^\infty U_{\theta, \eta}^{k}\\
& = \nu^T.
 \end{align*}

 It remains to prove (\ref{eq:Pfinite}).   Indeed,
      \begin{equation*}
\begin{aligned}[c]
   P_{\theta}- P_{\eta} + U_{\theta, \eta}  P_{\eta}& = P_{\theta}- P_{\eta} +  (P_{\theta} - P_{\eta}) Z_{\eta}   P_{\eta}\\
   & =  (P_{\theta}- P_{\eta} )(I + Z_{\eta}   P_{\eta} )\\
   & = (P_{\theta}- P_{\eta} )(I -\Pi_{\eta} + Z_{\eta}   P_{\eta} )   \\
   & =   (P_{\theta}- P_{\eta} )(I - Z_{\eta}  \Pi_{\eta} + Z_{\eta}   P_{\eta} )    \\
   & =  (P_{\theta}- P_{\eta} ) Z_{\eta} \\
   & = U_{\theta, \eta},
\end{aligned}
\end{equation*}
where the second equality follows from
  \begin{align*}
    \Big((P_{\theta} - P_{\eta}) Z_{\eta} \Big)  P_{\eta} =     (P_{\theta} - P_{\eta}) \Big(Z_{\eta}  P_{\eta} \Big),
  \end{align*}
  which holds   by \cite[Corollary 1.9]{Kemeny1976}, the third equality holds due to $(P_{\theta}- P_{\eta}) \Pi_{\eta}=0$, the fourth equality holds because $ Z_{\eta}  \Pi_{\eta} = \Pi_{\eta}$,  and the fifth equality follows from $ I = Z_{\eta}(I + \Pi_\eta - P_\eta) = Z_\eta + Z_\eta \Pi_\eta - Z_\eta P_{\eta}$.

The uniqueness of the stationary distribution follows from the fact that the Markov chain with transition matrix $P_{\theta}$ is assumed to be irreducible.

\end{proof}

 The following Lemma \ref{lem:norms} will be used in the proofs of Theorem \ref{thm:main} and Lemma  \ref{lem:policies} below. We believe the claim of  Lemma \ref{lem:norms} should be a well-known mathematical fact, but we have not found its proof in any textbook. For completeness, we present it here.

 \begin{lemma}\label{lem:norms}

Let $\M_{\X, \X}$ be a set of all matrices on the countable space $\X\times \X$.
 The operator norm $\|\cdot\|_\V$ on  $\M_{\X\times \X}$ is equivalent to the operator norm  induced from  vector norm $\|\cdot\|_{1, \V}$ and to the operator norm  induced from vector norm $\|\cdot\|_{\infty, \V}$ in the following sense:
\begin{align*}
\|T\|_\V = \sup\limits_{\nu:\|\nu\|_{1, \V}=1} \|\nu T \|_{1, \V} = \sup\limits_{h:\|h\|_{\infty,\V} = 1} \| Th \|_{\infty, \V}\quad\text{ for any }T\in \M_{\X\times \X},
\end{align*}
where $\|\nu\|_{1, \V}=\sum\limits_{x\in \X} |\nu(x)|\V(x)$, $\|\nu\|_{\infty, \V} = \sup\limits_{x\in \X} \frac{|\nu(x)|}{\V(x)}$, $\|T\|_\V=\sup\limits_{x\in \X} \frac{1}{\V(x)} \sum\limits_{y\in \X} |T(x, y)|\V(y)$.

 Furthermore, for any vectors $\nu_1, \nu_2$ on $\X$ and matrices $T_1,T_2\in \M_{\X\times \X}$ the following inequalities hold:
\begin{align}\label{eq:norms1}
\|\nu_1^T T \nu_2\|_\V\leq \|\nu_1\|_{1, \V}\|T\|_\V\| \nu_2 \|_{\infty, \V}
\end{align}
and
\begin{align}\label{eq:norms2}
\|  T_1 T_2\|_\V\leq \|T_1\|_{\V}\|T_2\|_\V.
\end{align}

\begin{proof}

First, we show that $\sup\limits_{h:\|h\|_{1, \V} = 1} \| Th \|_{1, \V} = \|T\|_\V$. On the one hand,
\begin{align*}
 \sup\limits_{\nu:\|\nu\|_{1, \V}=1} \|\nu T \|_{1, \V}  &= \sup\limits_{\nu: \X\rightarrow \R} \frac{1}{\|\nu\|_{1, \V}} \|\nu T \|_{1, \V}  \\
 &\geq\sup\limits_{ \nu\in \{e_x\}} \frac{1}{\|\nu\|_{1, \V}} \|\nu T \|_{1, \V}   \\
 &=  \sup\limits_{x\in \X} \frac{1}{\V(x)} \sum\limits_{y\in X} \V(y) |T(x, y)| = \|T\|_\V,
\end{align*}
where in the second step we choose a set $\{e_x\}$ of unit vectors $ e_x = (0,...,0,1,0,...)$, where the $x$ coordinate is $1$ and the other coordinates are $0$s.

On the other hand,
\begin{align*}
\|T\|_\V &= \sup\limits_{x\in X} \sum\limits_{y\in X}   \frac{1}{\V(x)} |T(x, y)|  \V(y) \sup\limits_{\nu:\|\nu\|_{1, \V}=1}  \sum\limits_{x\in X} |\nu(x)| \V(x)\\
&\geq \sup\limits_{\nu:\|\nu\|_{1, \V}=1}   \sum\limits_{y\in X}   \sum\limits_{x\in X}  \frac{1}{\V(x)}~|T(x, y)| ~ \V(y) ~ |\nu(x)|~ \V(x)\\
& =\sup\limits_{\nu:\|\nu\|_{1, \V}=1}   \sum\limits_{y\in X}   \sum\limits_{x\in X} |T(x, y) \nu(x)| \V(y) \\
&=\sup\limits_{\nu:\|\nu\|_{1, \V}=1} \|\nu T \|_{1, \V}.
\end{align*}

Similarly, we can show the equivalency of  $\sup\limits_{h:\|h\|_{\infty, \V} = 1} \| Th \|_{\infty, \V}$ and $\|T\|_\V$ norms.

Inequalities (\ref{eq:norms1}) and (\ref{eq:norms2}) follow from the properties of a linear operator norm.
\end{proof}
\end{lemma}

\begin{proof}[\textbf{Proof of Theorem \ref{thm:main}}]
 We denote $U_{\theta, \eta}:=(P_\theta - P_\eta)Z_\eta$. Under assumption $\|U_{\theta, \eta}\|_\V = D_{\theta, \eta} <1$  operator $H_{\theta, \eta}:=\sum\limits_{k=0}^{\infty}U^k_{\theta, \eta} $ is well-defined and
 \begin{align}\label{eq:normH}
 \|H_{\theta, \eta}\|_\V\leq \frac{1}{1 - D_{\theta, \eta}}.
 \end{align}

We represent the stationary distribution of the Markov chain with transition matrix $P_\theta$  as  $\mu_\theta^T = \mu_{\eta}^T H_{\theta, \eta}$, see   (\ref{eq:mu12}).  We get
\begin{align}\label{eq:mu_theta}
 \|\mu_\theta\|^{ }_{1, \V} = \mu_\theta^T\V = \mu_\eta H_{\theta, \eta}\V\leq\|H_{\theta, \eta}\|_\V^{ } (\mu_\eta^T \V)<\infty,
\end{align}
 since $H_{\theta,\eta}\V\leq \|H_{\theta, \eta}\|_\V^{ }\V$   by definition of the $\V$-norm.

The long-run average costs difference is equal to
\begin{equation*}
\begin{aligned}[c]
\mu_\theta^Tg - \mu_\eta^Tg &= \mu_\theta^Tg+ \mu_{\theta}^T\left( (P_\theta-I)h_{\eta} \right)- \mu_\eta^Tg \\
&=\mu_{\theta}^T(g + P_\theta h_{\eta}  - h_{\eta} ) -\mu_\eta^Tg\\
&=\mu_{\eta}^T(g  - (\mu_\eta^Tg ) e + P_\theta h_{\eta}  -h_{\eta}) + (\mu_{\theta}^T - \mu_{\eta}^T)(g+P_\theta h_{\eta}  - h_{\eta} )\\
& = \mu_{\eta}^T(g  -(\mu_\eta^Tg)  e + P_\theta h_{\eta}  -h_{\eta}) + (\mu_{\theta}^T - \mu_{\eta}^T)(g - (\mu_\eta^Tg ) e +P_\theta h_{\eta}  - h_{\eta} ).
\end{aligned}
 \end{equation*}

Now we are ready to bound the last term:
\begin{equation*}
\begin{aligned}[c]
 \left|(\mu_{\theta}^T - \mu_{\eta}^T)   (g - (\mu_\eta^Tg)e +P_\theta h_{\eta}  - h_{\eta} )\right|&\leq \|\mu_{\theta} - \mu_{\eta}  \|_{1, \V}~ \|g - (\mu_{\eta}^T g)e +P_\theta h_{\eta}  - h_{\eta} \|_{\infty, \V} \\
 &=\|\mu_{\theta} - \mu_{\eta}  \|_{1, \V}~\|( P_\theta - P_{\eta}) h_{\eta}  \|_{\infty, \V}\\
 &= \|\mu_{\theta} - \mu_{\eta}  \|_{1, \V}~\| (P_\theta - P_{\eta})Z_{\eta} (g - (\mu_{\eta}^T g)e)  \|_{\infty, \V} \\
 &\leq  \|\mu_{\theta} - \mu_{\eta}  \|_{1, \V}~\| (P_\theta - P_{\eta}) Z_{\eta}\|_\V~ \| g - (\mu_\eta^Tg) e  \|_{\infty, \V}\\
 &=  D_{\theta, \eta}  \|\mu_{\theta} - \mu_{\eta}  \|_{1, \V}~\| g -(\mu_{\eta}^T g) e  \|_{\infty, \V}   \\
 &  =  D_{\theta, \eta} \|\mu_\theta^T U_{\theta, \eta} \|_{1, \V}~\| g - (\mu_\eta^Tg) e  \|_{\infty, \V},\\
 &\leq D_{\theta, \eta}  \|\mu_\theta\|_{1,\V}^{  } \|U_{\theta, \eta} \|_{ \V}~\| g - (\mu_\eta^Tg) e  \|_{\infty, \V},\\
  &\leq  D_{\theta, \eta}^2 \|\mu_\theta\|_{1,\V}^{ } \| g - (\mu_\eta^Tg) e  \|_{\infty, \V },\\
  &\leq D^2_{\theta, \eta} \|H_{\theta, \eta}\|_\V^{ }  \| g - (\mu_\eta^Tg) e  \|_{\infty, \V}  (\mu_\eta^T \V),\\
 &\leq  \frac{D^2_{\theta, \eta}}{1-D_{\theta, \eta}} \| g - (\mu_\eta^Tg) e  \|_{\infty, \V} (\mu_\eta^T \V),
\end{aligned}
 \end{equation*}
 where the first, second and third inequalities follow from Lemma  \ref{lem:norms}, the second equality follows from Lemma \ref{lem:Zeq},  the last equality holds due to
$\mu_{\theta}^T - \mu_{\eta}^T  =  \mu_\theta^T U_{\theta, \eta}$
from (\ref{eq:mu12}),   the fourth inequality follows from  (\ref{eq:D}), the fifth inequality follows from (\ref{eq:mu_theta}),  and the last inequality holds due to  (\ref{eq:normH}).

\end{proof}

\begin{proof}[\textbf{Proof of Lemma \ref{lem:policies}}]
    \begin{equation*}
\begin{aligned}[c]
     \|  (P_{\theta} - P_{\eta}) Z_{\eta}\|_\V&\leq  \|  P_{\theta} - P_{\eta} \|_\V\|Z_{\eta}\|_\V\\
     & = \|Z_{\eta}\|_\V\sup\limits_{x\in \X}\frac{1}{\V(x)} \sum\limits_{y\in \X} |P_{\theta} - P_{\eta}|_{x, y} \V(y)\\
     & = \|Z_{\eta}\|_\V\sup\limits_{x\in \X}\frac{1}{\V(x)} \sum\limits_{y\in \X} | \sum\limits_{a\in \A} P(y|x, a)\pi_\theta(a|x) - \sum\limits_{a\in \A} P(y|x, a)\pi_{\eta}(a|x)| \V(y)\\
          & \leq \|Z_{\eta}\|_\V\sup\limits_{x\in \X}\frac{1}{\V(x)} \sum\limits_{y\in \X} \sum\limits_{a\in \A} P(y|x, a) | \pi_\theta(a|x) -  \pi_{\eta}(a|x)| \V(y)\\
          & =\|Z_{\eta}\|_\V\sup\limits_{x\in \X}    \sum\limits_{a\in \A} | \pi_\theta(a|x) -  \pi_{\eta}(a|x)|    \frac{\sum\limits_{y\in \X} P(y|x, a) \V(y) }{\V(x)} \\
          &= \|Z_{\eta}\|_\V\sup\limits_{x\in \X}    \sum\limits_{a\in \A} \Big| \frac{\pi_\theta(a|x)}{ \pi_{\eta}(a|x)} - 1\Big|   G(x, a)
\end{aligned}
 \end{equation*}
\end{proof}

 \section{ Proofs of the theorems in Section \ref{sec:M1} }\label{sec:disc}

 We consider the Poisson equation for a Markov chain with the transition kernel $P$, stationary distribution $\mu$, and cost function $g:\X\rightarrow \R$:
\begin{align*}
g(x) - \mu^Tg +\sum\limits_{y\in \X} P(y|x) h(y) - h(x) = 0, \text{ for each }x\in \X,
\end{align*}
which admits a solution
\begin{align*}
h^{(x^*)}(x):=\E\left[ \sum\limits_{k=0}^{\sigma(x^*)-1} \left(g (x^{(k)} )-\mu^Tg\right)|x^{(0)}=x \right] \text{ for each }x\in \X,
\end{align*}
where $\sigma(x^*) = \min\left\{k>0~|~x^{(k)}=x^*\right\}$ is the first time when state $x^*$ is visited.

Since regenerative cycles can be long in large-size systems, we propose to change the original dynamics and increase the probability of transition to the regenerative state $x^*$ from each state $x\in \X.$

We let $P(y|x)$ be an original transition probability from state $x$ to state $y$, for each $x,y\in \X$. We consider a new Markov reward process with cost function $g$ and a modified transition kernel $\tilde P^{(\gamma)}$:
\begin{align}\label{eq:newP}
\begin{cases}
\tilde P^{(\gamma)}(y|x) := \gamma P(y|x)\quad \text{ for }y\neq x^*,\\
\tilde P^{(\gamma) }(x^*|x) := \gamma P(x^*|x)+(1-\gamma),
\end{cases}
\end{align}
for each $x\in \X.$

We  modified the transition kernel so that the probability of transition to the regenerative state $x^*$ is at least $1-\gamma$ from any state.

The Poisson equation for the modified problem is equal to:
\begin{align}\label{eq:newPoisson}
g(x) - \tilde \mu^T g +\sum\limits_{y\in \X}  \tilde P^{(\gamma)}(y|x) \tilde h(y) - \tilde h(x) = 0, \text{ for each }x\in \X,
\end{align}
 where $\tilde \mu$ is the stationary distribution of the Markov chain $\tilde P^{(\gamma)}$.

 Equation (\ref{eq:newPoisson}) admits a solution
\begin{align}\label{eq:newSol}
\tilde h^{(x^*)}(x):=\E\left[ \sum\limits_{k=0}^{\tilde \sigma(x^*)-1} \left (g (x^{(k)} )- \tilde \mu^Tg \right )~|~x^{(0)}=x \right] \text{ for each }x\in \X,
\end{align}
where   $x^{(k)}$ is the state of the Markov chain with transition matrix $\tilde P^{(\gamma)}$ after $k$ timesteps,  and $\tilde \sigma(x^*) = \min\left\{k>0~|~x^{(k)} = x^*\right\}$.
According to the new dynamics the regeneration occurs more frequently and we can estimate solution (\ref{eq:newSol})   by using  fewer replications of the regenerative simulation.

\begin{lemma}\label{lem:2sol}
Consider the Poisson equation for the Markov chain with the transition kernel $\tilde P^{(\gamma)}$ defined by (\ref{eq:newP}), stationary distribution $\tilde \mu$, and cost function $g:\X\rightarrow \R$:
\begin{align}\label{eq:newPoisson1}
g(x) - \tilde \mu^T g +\sum\limits_{y\in \X}  \tilde P^{(\gamma)}(y|x) \tilde h(y) - \tilde h(x) = 0, \text{ for each }x\in \X.
\end{align}
Equation (\ref{eq:newPoisson1}) admits solutions:
\begin{align*}
J^{(\gamma)}(x):=\E\left[ \sum\limits_{k=0}^{\infty} \gamma^k\left(g(x^{(k)})- \mu^Tg \right)~|~x^{(0)}=x \right] \text{ for each }x\in \X,
\end{align*}
and
\begin{align*}
V^{(\gamma )}(x):=\E\left[ \sum\limits_{k=0}^{\sigma(x^*)-1} \gamma^k\left(g (x^{(k)} )-r(x^*)\right )~|~x^{(0)}=x \right] \text{ for each }x\in \X,
\end{align*}
where   $x^{(k)}$ is the state of the Markov chain with transition matrix $P$   after $k$ timesteps.
\begin{proof}
We substitute the definition of $ \tilde P^{(\gamma)}$ (\ref{eq:newP}) and rewrite equation (\ref{eq:newPoisson1}) as
 \begin{align}\label{eq:newPoisson2}
g(x) - (\tilde \mu^T g - (1-\gamma) \tilde h(x^*)) + \gamma \sum\limits_{y\in \X} P(y|x) \tilde h(y)  - \tilde h(x) = 0, \text{ for each }x\in \X.
\end{align}

Equation (\ref{eq:newPoisson2}) admits infinitely many solutions, but we specify a unique solution fixing $\tilde h(x^*).$
Next, we consider two options.

First, we let $ \tilde h(x^*) = \frac{1}{1-\gamma}(\mu^Tg -\tilde \mu^T g )$. Then the Poisson equation (\ref{eq:newPoisson2}) becomes
 \begin{align*}
g(x) - \mu^T g + \gamma \sum\limits_{y\in \X} P(y|x) \tilde h(y)  - \tilde h(x) = 0, \text{ for each }x\in \X,
\end{align*}
and admits   solution
\begin{align*}
J^{(\gamma)}(x):=\E\left[ \sum\limits_{k=0}^{\infty} \gamma^k\left(g (x^{(k)} )- \mu^Tg \right)~|~x^{(0)}=x \right] \text{ for each }x\in \X.
\end{align*}

Second, we let $\tilde h(x^*) = 0$ in equation (\ref{eq:newPoisson2}). We note that $\tilde \mu^T g = r(x^*)$, where $r(x^*) = (1-\gamma)\E\left[\sum\limits_{k=0}^\infty \gamma^k g\left(x^{(k)}\right)~|~x^{(0)}=x^*\right]$ is a present discounted value at $x^*.$

We get  the Poisson equation (\ref{eq:Poiss_reg})
\begin{align*}
g(x) - r(x^*) +\gamma \sum\limits_{y\in \X} P(y|x) \tilde  h(y) -\tilde  h(x) = 0, \text{ for each }x\in \X,
\end{align*}
which admits   solution  (\ref{eq:Vreg})
\begin{align*}
V^{(\gamma )}(x):=\E\left[ \sum\limits_{k=0}^{\sigma(x^*)-1} \gamma^k\left(g (x^{(k)} )-r(x^*) \right)|x^{(0)}=x \right] \text{ for each }x\in \X.
\end{align*}

Indeed,
\begin{align*}
V^{(\gamma )}(x) &= \E\left[ \sum\limits_{k=0}^{\infty} \gamma^k\left(g (x^{(k)} )-r(x^*)\right )|x^{(0)}=x \right]\\
&=\E\left[ \sum\limits_{k=0}^{\sigma(x^*)-1} \gamma^k\left(g (x^{(k)} )-r(x^*) \right)|x^{(0)}=x \right]+ \E\left[ \gamma^{\sigma(x^*)}\E\left[\sum\limits_{k=0}^{\infty} \gamma^k\left(g(x^{(k)} )-r(x^*)\right)|x^{(0)}=x^* \right]\right]\\
&=\E\left[ \sum\limits_{k=0}^{\sigma(x^*)-1} \gamma^k\left(g (x^{(k)} )-r(x^*) \right)|x^{(0)}=x \right].
\end{align*}

\end{proof}
\end{lemma}

\begin{proof}[\textbf{Proof of Lemma \ref{lem:disc}}]
By Lemma \ref{lem:2sol} function $ V^{(\gamma )}$ is a solution of Poisson equation (\ref{eq:newPoisson}). We consider the discounted value function  $J^{(  \gamma)}  = \E\left[ \sum\limits_{k=0}^{\infty} \gamma^k\left(g\left(x^{(k)}\right)- \mu^Tg \right)~|~x^{(0)}=x \right] $ that is another solution.

Then, for an arbitrary $x\in \X$,
\begin{align}\label{eq:VJ}
\left|V^{(\gamma)}(x) - h^{(x^*)}(x)\right| \leq \left|J^{(\gamma)}(x) - h^{(f)}(x)\right| + \left|V^{(\gamma)}(x) - h^{(x^*)}(x)  - \left(J^{(\gamma)}(x) - h^{(f)}(x)\right)\right|,
\end{align}
where $h^{(f)}$ is the fundamental solution of the Poisson equation (\ref{eq:Poisson}).

First, we  bound $\left|J^{(\gamma)}(x) - h^{(f)}(x)\right|$:
\begin{align*}
|J^{(\gamma)}(x) - h^{(f)}(x) | &\leq \sum\limits_{t=0}^\infty |\gamma^{t} - 1| \Big| \sum\limits_{y\in \X} P^t(y|x) (g(y) - \mu^T g)   \Big| \\
&\leq R \V(x) \sum\limits_{t=0}^\infty (1  - \gamma^{t}) r^t\\
&= R\V(x)r\frac{1-\gamma}{(1-r)(1-\gamma r) },
\end{align*}
where the second inequality follows from (\ref{eq:geo}).

 Since  $V^{(\gamma)}$ and $J^{(\gamma)}$ are both solutions of the Poisson equation (\ref{eq:newPoisson}), therefore,
\begin{align*}
J^{(\gamma)}(x) - V^{(\gamma)}(x) =J^{(\gamma)}(x^*) - V^{(\gamma)}(x^*) = \frac{1}{1-\gamma}(\mu^T g - r(x^*)) \text{ for each }x\in \X.
\end{align*}

Similarly, $ h^{(f)}(x) - h^{(x^*)}(x) = h^{(f)}(x^*)$ for each $x\in \X.$

Second, we bound the last term in inequality (\ref{eq:VJ}):
\begin{align*}
\Big|V^{(\gamma)}(x) - h^{(x^*)}(x) & - \left(J^{(\gamma)}(x) - h^{(f)}(x)\right)\Big|\\
&=\left|h^{(f)}(x^*) - \frac{1}{1-\gamma}(r(x^*)-\mu^T g ) \right|\\
&=\left| \sum\limits_{t=0}^{\infty}\sum\limits_{y\in \X} P^t(y|x^*)(g(y)-\mu^Tg) - \sum\limits_{t=0}^{\infty}\sum\limits_{y\in \X}  \gamma^tP^t(y|x^*)(g(y)-\mu^Tg)  \right|\\
&=\left|  \sum\limits_{t=0}^{\infty}\sum\limits_{y\in \X}  (1-\gamma^t) P^t(y|x^*)(g(y)-\mu^Tg) \right|\\
&\leq    \sum\limits_{t=0}^{\infty} |1-\gamma^t| \left|\sum\limits_{y\in \X} P^t(y|x^*)(g(y)-\mu^Tg) \right|  \\
&\leq R\V(x^*)r\frac{1-\gamma}{(1-r)(1-\gamma r) },
\end{align*}
where the last inequality holds due to (\ref{eq:geo}).
\end{proof}

\begin{proof}[\textbf{Proof of Lemma \ref{lem:var}}]

We denote $\overline g(x^{(k)}): = \left(g (x^{(k)} )- r (x^*) \right )$.
 Following   \cite[Section 17.4.3]{Meyn2009}, we can show that
\begin{align*}
Var[ \hat V^{(\gamma )}(x)] &= \E\left[  \left(\sum\limits_{k=0}^{\sigma(x^*)-1} \gamma^k \overline  g (x^{(k)} ) \right)^2~\Big|~x^{(0)}=x \right] -\left( V^{(\gamma)}(x)\right)^2 \\
&=\E\left[ -\sum\limits_{k=0}^{\sigma(x^*)-1} \gamma^{2k}\overline g^2(x^{(k)}) + 2\sum\limits_{k=0}^{\sigma(x^*)-1} \sum\limits_{j=k}^{\sigma(x^*)-1}\gamma^{k+j}\overline g(x^{(k)})\overline g(x^{(j)}) ~\Big|~x^{(0)}= x \right] -\left( V^{(\gamma)}(x)\right)^2 \\
&=\E\left[ \sum\limits_{k=0}^{\sigma(x^*)-1}  \E\left[ 2  \sum\limits_{j=k}^{\sigma(x^*)-1}\gamma^{k+j}\overline g(x^{(k)})\overline g(x^{(j)})  - \gamma^{2k}\overline g^2(x^{(k)}) ~|~\F_k\right]~\Big|~x^{(0)}= x \right] -\left( V^{(\gamma)}(x)\right)^2 \\
&=\E\left[ \sum\limits_{k=0}^{\sigma(x^*)-1}  2\gamma^{2k}\overline g(x^{(k)})\E\left[   \sum\limits_{j=k}^{\sigma(x^*)-1}\gamma^{j-k}\overline g(x^{(j)})   ~|~\F_k\right]- \gamma^{2k}\overline g^2(x^{(k)})~\Big|~x^{(0)}= x \right] -\left( V^{(\gamma)}(x)\right)^2 \\
&= \E\left[ \sum\limits_{k=0}^{\sigma(x^*)-1}  2\gamma^{2k}\overline  g(x^{(k)} ) V^{(\gamma)}(x^{(k)})-\gamma^{2k} \overline g^2(x^{(k)}) ~\Big|~x^{(0)}= x \right] -\left( V^{(\gamma)}(x)\right)^2,
\end{align*}
where $\F_k$ is a $\sigma-$algebra generated by $x_0, x_1, ..., x_k$.

Hereafter, we denote $\sum\limits_{y\in \X}P(y|x) V^{(\gamma)}(y)$ and $\sum\limits_{y\in \X}P(y|x) \left(V^{(\gamma)}(y)\right)^2$  as $PV^{(\gamma)}(x)$ and $P\left(V^{(\gamma)}\right)^2(x)$, respectively,  to improve readability. We use the Poisson equation (\ref{eq:Poiss_reg}) and replace $ \overline g(x^{(k)} )$ by $V^{(\gamma)}(x^{(k)}) -\gamma P V^{(\gamma)}(x^{(k)})$.
\begin{align*}
  &\E\left[ \sum\limits_{k=0}^{\sigma(x^*)-1}  2\gamma^{2k}\overline  g(x^{(k)} ) V^{(\gamma)}(x^{(k)})-\gamma^{2k} \overline g^2(x^{(k)})  ~\Big|~x^{(0)}= x \right] -\left( V^{(\gamma)}(x)\right)^2\\
&=\E\left[ \sum\limits_{k=0}^{\sigma(x^*)-1}  \gamma^{2k}\left(2\left(V^{(\gamma)}(x^{(k)}) -\gamma P V^{(\gamma)}(x^{(k)})\right)V^{(\gamma)}(x_k)-  \left(V^{(\gamma)}(x^{(k)}) - \gamma PV^{(\gamma)}(x^{(k)})\right)^2\right)~\Big|~x^{(0)}= x \right] -\left( V^{(\gamma)}(x)\right)^2\\
&=\E\left[ \sum\limits_{k=0}^{\sigma(x^*)-1}  \gamma^{2k}\left(\left(V^{(\gamma)}(x^{(k)})\right)^2  -  \left(\gamma PV^{(\gamma)}(x^{(k)})\right)^2 \right)~\Big|~x^{(0)}= x \right] -\left( V^{(\gamma)}(x)\right)^2.
\end{align*}
Next, we subtract the expectation of  a martingale
\begin{align*}
\E\left[  -\left( V^{(\gamma)}(x)\right)^2 +\gamma^{2\sigma(x^*)}\left(V^{(\gamma)}(x^{(\sigma(x^*))})\right)^2 +\sum\limits_{k=0}^{\sigma(x^*)-1}  \gamma^{2k}\left(\left(V^{(\gamma)}(x^{(k)})\right)^2  -  \gamma^2P\left(V^{(\gamma)}\right)^2(x^{(k)}) \right)~\Big|~x^{(0)}= x \right]
\end{align*}
 that is equal to zero by \cite[Proposition 1]{Henderson1997}. Since $V^{(\gamma)}(x^{(\sigma(x^*))})=0$, we get
\begin{align*}
Var[ \hat V^{(\gamma )}(x)] &=\E\left[ \sum\limits_{k=0}^{\sigma(x^*)-1}  \gamma^{2k}\left(\gamma^2 P\left(V^{(\gamma)}\right)^2(x^{(k)}) -    \left( \gamma PV^{(\gamma)}(x^{(k)})\right)^2\right)~\Big|~x^{(0)}= x \right]\\
  &=\gamma^2\E\left[ \sum\limits_{k=0}^{\sigma(x^*)-1}  \gamma^{2k}\left( Var\left[V^{(\gamma)}(x^{(k+1)})~|~x^{(k)}\right]\right)~\Big|~x^{(0)}= x \right],
\end{align*}
where $Var\left[V^{(\gamma)}(x^{(k+1)})~|~x^{(k)}\right] = \sum\limits_{y\in \X}P(y|x^{(k)}) \left(V^{(\gamma)}(y)\right)^2 - \left(\sum\limits_{y\in \X}P(y|x^{(k)}) V^{(\gamma)}(y)\right)^2.$

Next, we want to show that there exists   constant $B_1>0$ such that for any  $\gamma\in [0,1]$
 \begin{align*}\left(V^{(\gamma)}(x)\right)^2\leq B_1 \V(x)\quad \text{for each }x\in \X.\end{align*}
First, we recall that function $V^{(\gamma)}$ is a solution of Poisson equation (\ref{eq:newPoisson}) for the system with modified dynamics (\ref{eq:newP}) and cost function $g$. Given that transition matrix $P$ satisfies the drift condition (\ref{eq:drift}), for the modified dynamics we have the following drift inequality
\begin{align*}
\sum\limits_{y\in \X}\tilde P^{(\gamma)}(y|x)\V(y) \leq b \V(x) +(\V(x^*)+d)\I_{C\cup \{x:b \V(x)\leq \V(x^*)\}}(x)\quad \text{for each }x\in \X,
\end{align*}
for any $\gamma\in [0,1].$
Indeed,
\begin{align*}
\sum\limits_{y\in \X}\tilde P^{(\gamma)}(y|x)\V(y) &=\gamma \sum\limits_{y\in \X} P(y|x)\V(y) +(1-\gamma)\V(x^*)\\
&\leq \gamma b  \V(x) +(1-\gamma)\V(x^*) +d\I_C\\
&\leq b \V(x) +(\V(x^*)+d)\I_{C\cup \{x\in \X:b \V(x)\leq \V(x^*)\}}(x),
\end{align*}
where the first inequality follows from the drift condition (\ref{eq:drift}), and the second inequality follows from the fact that $ \gamma b \V(x) +(1-\gamma)\V(x^*) \leq b \V(x)$ if $b \V(x)\geq \V(x^*)$, and $ \gamma b\V(x) +(1-\gamma)\V(x^*) \leq   \V(x^*)$ otherwise.

Second, we use Jensen's inequality and get that function $\sqrt{V}$ is also a Lyapunov function for the modified system:
\begin{align*}
\sum\limits_{y\in \X}\tilde P^{(\gamma)}(y|x) \sqrt{\V(y)} \leq \sqrt{b \V(x)} +\sqrt{\V(x^*)+d}~\I_{C\cup \{x\in \X:b \V(x)\leq \V(x^*)\}}(x)\quad \text{for each }x\in \X.
\end{align*}
This drift inequality and the assumption that $|g(x)|\leq \sqrt{\V(x)}$ for each $x\in \X$  allow us to apply \cite[Theorem 17.7.1]{Meyn2009}, see also \cite[equation (17.39)]{Meyn2009},   and conclude that, for some $c_0>0$ independent of $\gamma$, Poisson equation (\ref{eq:newPoisson}) admits the fundamental solution $J^{(\gamma)}:\X\rightarrow \R$ such that
\begin{align*}
|J^{(\gamma)}(x)|\leq c_0(\sqrt{\V(x)}+1), \text{ for each } x\in \X.
\end{align*}
 Function $V^{(\gamma)}$ is another solution of Poisson equation (\ref{eq:newPoisson}), such that  $J^{(\gamma)}(x) =V^{(\gamma)}(x)+J^{(\gamma)}(x^*)$, for each $x\in \X$, because $V^{(\gamma)}(x^*)=0$. Since $\V\geq 1$, there exists   constant $B_1>0$ such that
\begin{align*}
|V^{(\gamma)}(x)|\leq |J^{(\gamma)}(x)| +|J^{(\gamma)}(x^*)|\leq c_0(\sqrt{\V(x)}+1) +c_0(\sqrt{\V(x^*)}+1)\leq \sqrt{B_1 \V(x)}, \quad \text{for each }x\in \X.
\end{align*}
We have proved that, for some   constant $B_1>0$,
  \begin{align*}\left(V^{(\gamma)}(x)\right)^2\leq B_1 \V(x)\quad \text{for each }x\in \X,\end{align*} for any $\gamma\in [0,1].$

We let $G^{(\gamma)}(x) := Var\left[V^{(\gamma)}(x^{(1)})~|~x^{(0)}=x\right].$  Then there exists a positive constant $B$ such that
  \begin{align}\label{eq:var_ineq1}
  G^{(\gamma)}(x) \leq  \sum\limits_{y\in \X}P(y|x) \left(V^{(\gamma)}(y)\right)^2\leq B_1P\V(x)\leq B\V(x),
  \end{align}
where the last inequality follows from the drift condition (\ref{eq:drift}).
By  \cite[Theorem 15.0.1]{Meyn2009} we have that there exist constants $R>0$ and $r\in (0,1)$ such that
\begin{align}\label{eq:var_ineq2}
\Big|  P^k G^{(\gamma)}(x) -  \beta^{(\gamma)}\Big|\leq R\V(x) r^k,
\end{align}
where  $\beta^{(\gamma)} := \sum\limits_{x\in \X} \mu(x) G^{(\gamma)}(x)$ is a discounted asymptotic variance.

Then
\begin{align*}
Var[ \hat V^{(\gamma )}(x)] &=\gamma^2\E\left[ \sum\limits_{k=0}^{\sigma(x^*)-1}  \gamma^{2k}\left( Var\left[V^{(\gamma)}(x^{(k+1)})~|~x^{(k)}\right]\right)~\Big|~x^{(0)}= x \right]\\
&\leq \gamma^2 \sum\limits_{k=0}^\infty \gamma^{2k}\E[G^{(\gamma)}(x_k)~|~x_0=x ]\\
&\leq  \gamma^2 \sum\limits_{k=0}^\infty  \gamma^{2k}\left( R\V(x) r^k + \beta^{(\gamma)} \right)\\
&=\gamma^2\left( R\V(x)\frac{1}{1-\gamma^2r} + \beta^{(\gamma)} \frac{1}{1-\gamma^2}  \right)\\
&\leq \gamma^2\left( R\V(x)\frac{1}{1-\gamma^2r} + (\mu^T\V) B\frac{1}{1-\gamma^2}  \right),
\end{align*}
where the second inequality follows from (\ref{eq:var_ineq2}) and the last inequality follows from  (\ref{eq:var_ineq1}).

\end{proof}

\section{Maximal stability of the proportionally randomized policy}\label{sec:PR}

We assert that a discrete-time MDP obtained by
uniformization of the multiclass queueing network SMDP model is stable
under the proportionally randomized (PR) policy if the load conditions
(\ref{eq:load}) are satisfied. We illustrate the proof for the criss-cross
queueing network. We let $x\in \Z_+^3$ be a state for the discrete-time
MDP. The proportionally randomized policy $\pi$ is given by
\begin{align*}
  & \pi(x) =
    \begin{cases}
      \Big(\frac{x_1}{x_1+x_3}, 1, \frac{x_3}{x_1+x_3}\Big) & \text{ if } x_2\ge 1 \text{ and } x_1+x_3\ge 1, \\
      \Big(\frac{x_1}{x_1+x_3}, 0, \frac{x_3}{x_1+x_3}\Big) & \text{ if } x_2= 0 \text{ and } x_1+x_3\ge 1, \\
    \big(0, 1, 0\big) & \text{ if } x_2\ge 1 \text{ and } x_1+x_3=0 , \\
     \big(0, 0, 0\big) & \text{ if } x_2=0 \text{ and } x_1+x_3=0.
      \end{cases}
\end{align*}
Recall the transition probabilities $\tilde P$ defined by (\ref{eq:unif}).  The
discrete-time MDP operating under policy $\pi$ is a DTMC. Now we can  specify its
transition matrix.  For $x\in \Z_+^3$ with $x_1\ge 1$, $x_3\ge 1$, and $x_2\ge 1$,
\begin{align*}
  P(y|x)& = \pi_1(x)\tilde P\big(y|x,(1,2)\big)+\pi_3(x)\tilde P\big(y|x,(3,2)\big) \quad \text{ for each } y\in \Z_+^3.
\end{align*}
For $x\in \Z_+^3$ on the boundary with $x_1\ge 1$, $x_3\ge 1$, and $x_2= 0$,
\begin{align*}
  P(y|x)& = \pi_1(x)\tilde P\big(y|x,(1,0)\big)+\pi_3(x)\tilde P\big(y|x,(3,0)\big) \quad \text{ for each } y\in \Z_+^3.
\end{align*}
For $x\in \Z_+^3$ on the boundary with $x_1= 0$, $x_3\ge 1$, and $x_2\ge 1$,
\begin{align*}
  P(y|x)& = \tilde P\big(y|x,(3,2)\big) \quad \text{ for each } y\in \Z_+^3.
\end{align*}
Similarly, we write the transition probabilities for other boundary cases. One can verify that
\begin{align}
  & P\big((x_1+1,x_2, x_3)| x\big)  = \frac{\lambda_1}{B},
    \quad    P\big((x_1,x_2, x_3+1)| x\big)  = \frac{\lambda_3}{B}, \label{eq:pr1}\\
  &  P\big((x_1-1,x_2+1, x_3)| x\big) =\frac{\mu_1}{B}\frac{x_1}{x_1+x_3}  \quad \text{ if } x_2\ge 1, \\
  & P\big((x_1,x_2, x_3-1)| x\big) =\frac{\mu_3}{B}\frac{x_3}{x_1+x_3} \quad \text{ if } x_3\ge 1, \\
  &  P\big((x_1,x_2-1, x_3)| x\big) =\frac{\mu_2}{B} \quad \text{ if } x_2\ge 1, \\
  & P(x|x) = 1- \sum_{y\neq x} P(y|x),\label{eq:pr5}
\end{align}
where $B=\lambda_1+\lambda_3+\mu_1+\mu_2+\mu_3$. This transition
matrix $P$ is irreducible. Now we consider the continuous-time
criss-cross network operating under the head-of-line
proportional-processor-sharing (HLPPS) policy  defined in
\cite{Bramson1996}. Under the HLPPS policy, the  jobcount  process
$\{Z(t), t\ge 0\}$ is a CTMC. Under the load condition (\ref{eq:load_cc}),
\cite{Bramson1996} proves that the CTMC is positive recurrent. One can
verify that the transition probabilities in
(\ref{eq:pr1})-(\ref{eq:pr5}) are identical to the ones for a
uniformized DTMC of this CTMC. Therefore, the DTMC corresponding to the
transition probabilities (\ref{eq:pr1})-(\ref{eq:pr5}) is positive
recurrent,  and proves  the stability of the discrete-time MDP operating under the
proportionally randomized policy.

\section{Additional experimental results}\label{sec:regVSinf}

In Remark \ref{rem:regVSinf} we discussed two possible biased
estimators of the solution to the Poisson equation. In this section we
compare the performance of the PPO algorithm with these two
estimators.  We consider two versions of line 7 in Algorithm \ref{alg2}: Version 1 uses the regenerative discounted value function
(VF) estimator (\ref{eq:esf}), and Version 2 uses the discounted value
function estimator (\ref{eq:Jesf}). We apply two versions of the PPO
algorithm for the criss-cross network operating under the  balanced medium (B.M.) load
regime.  The queueing network parameter setting is identical to
  the one detailed in Section \ref{sec:cc}, except that the quadratic
  cost function $g(x) = x_1^2+x_2^2+x_3^2$ replaces the linear
  cost function that is used to minimize the long-run average cost, where
$x_i$ is a number of jobs in buffer $i$, $i=1, 2, 3.$

We use  Xavier initialization to initialize
the policy NN parameters $\theta_0$. We take the empty system state $x^* = (0,0,0)$ as a regeneration state. Each episode in each iteration starts at the regenerative state and runs for $6,000$ timesteps. We compute the one-replication estimates of a value function (either regenerative discounted VF or discounted VF)   for the first $N=5,000$ steps at each episode. In this experiment we simulated $Q=20$ episodes in parallel.  The values of the remaining  hyperparameters (not mentioned yet) are  the same as  in Table \ref{tab:par2}.

In Figure \ref{fig:regVSinf} we compare the learning curves of PPO algorithm \ref{alg2}  empirically to demonstrate the benefits of using the regenerative discounted VF estimator over the discounted VF estimator when the system regeneration occurs frequently.

\begin{figure}[H]
\centering%
\includegraphics[width=.5\linewidth]{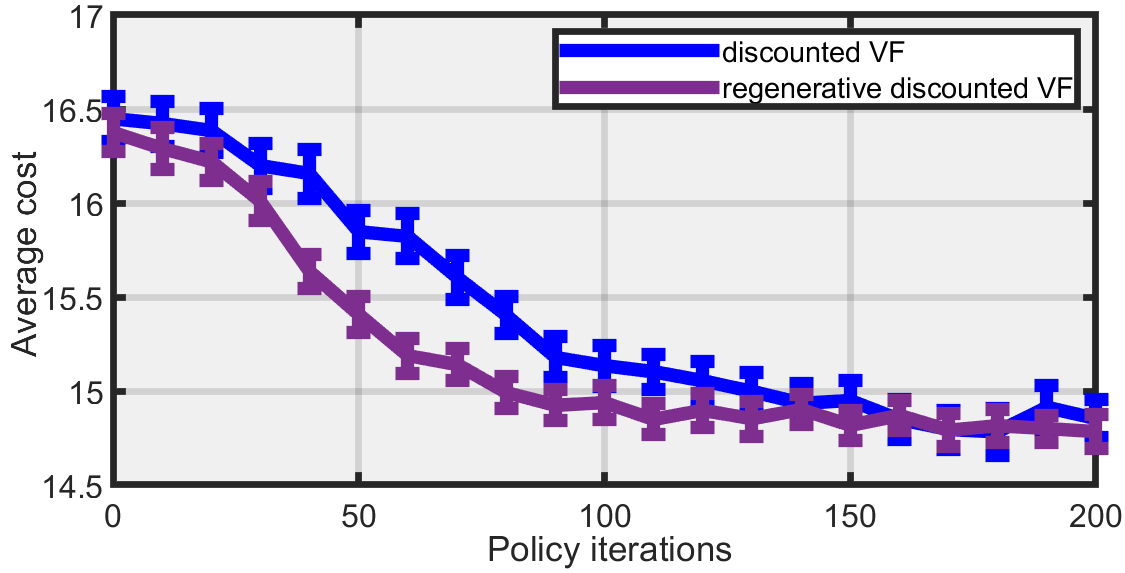}
\caption[]{
Learning curves from Algorithm \ref{alg2} for the criss-cross network with the B.M. load and quadratic cost function.
The solid purple and blue lines show the performance of  the PPO policies obtained from Algorithm \ref{alg2}  in which
the solutions to the Poisson equations are estimated  by the discounted VF estimator and by the regenerative discounted VF estimator, respectively.
}

 \label{fig:regVSinf}
\end{figure}

\section{Neural network structure}\label{sec:nn}

In the experiments we parameterized the RL policy with a neural
network. Here, we use $\theta$ to denote the vector of weights and
  biases of the neural network. For a fixed parameter $\theta$, the
  neural network outputs deterministically distribution
  $\pi_\theta(\cdot|x)$ over the action space for each state
  $x\in \X$. Therefore, the resulting policy $\pi_\theta$ is a randomized policy as explained in Section~\ref{sec:MQN}.

To represent the policy we use a fully connected  feed-forward neural network with one input layer,   three hidden layers with tanh activation functions, and one output layer. The input layer has $J$ units, one for each job class, the first hidden layer has $10\times J$ units, the third hidden layer has $10\times L$, where $L$ is  number of stations in the queueing system. The number of units in the  second hidden layer is a geometric mean of units in the first and third hidden layers (i.e. $10\times \sqrt{LJ}$).

We use $z^{(k)}_j$ to denote the variable in the $j$th unit of hidden
layer $k$, $k=1, 2, 3$. Thus, our feed-forward neural network has the following representations:
\begin{align*}
  & z^{(1)}_j =  h\Big(\sum_{i=1}^J A^{(1)}_{ji}x_i + b^{(1)}_j\Big), \quad j=1, \ldots, 10J, \\
  & z^{(2)}_j =  h\Big(\sum_{i=1}^{10J} A^{(2)}_{ji}z^{(1)}_i + b^{(2)}_j\Big), \quad j=1, \ldots, 10\sqrt{LJ},\\
  & z^{(3)}_j =  h\Big(\sum_{i=1}^{10\sqrt{LJ}} A^{(3)}_{ji}z^{(2)}_i + b^{(3)}_j\Big), \quad j=1, \ldots, 10\sqrt{L},
\end{align*}
where $h:\R\to\R$ is the activation function given by $h(y)=\tanh(y)$ for each $y\in \R$.

 We denote $p=(p_j)$ as the output vector, which is given by
\begin{align*}
  p_j = \sum_{i=1}^{10\sqrt{L}} A^{(4)}_{ji}z^{(3)}_i + b^{(4)}_j, \quad j\in \J,
\end{align*}
and normalize output vector $p$ via a \textit{softmax function}  into $L$ probability
distributions as:
\begin{align}\label{eq:softmax}
  \pi_{\theta}(j|x) = \frac{\exp(p_j)}{\sum_{i\in \mathcal{L}(\ell)} \exp(p_i)} \quad \text{ for each } j\in \mathcal{B}(\ell), \, \ell\in \LL,
\end{align}
where the sets $\J$, $\LL$, and $\mathcal{B}(\ell)$ are defined in Section~\ref{sec:MQN}.

The neural network parameter $\theta$ is the vector of weights $A$'s and biases $b$'s
\begin{align*}
\theta=  \Big(A^{(1)},\, b^{(1)}, A^{(2)},\, b^{(2)}, A^{(3)},\, b^{(3)}, A^{(4)},\, b^{(4)}\Big),
\end{align*}
which has dimension
\begin{align*}
  K=\big(J\times 10J + 10J\big) +   \big(10J \times 10\sqrt{LJ} + 10\sqrt{LJ}\big) +    \big(10\sqrt{LJ}\times 10L + 10L\big) +
    \big(10L\times J + J\big).
\end{align*}
For example, when $J=21$ and $L=7$, this dimension is approximately equal to
$30203$.

To represent the value function we use a neural network whose
architecture is almost identical to the policy neural network except
that the third hidden layer has $10$ units. The
number of units in the second hidden layer is $10\times \sqrt{J}$. The
output layer contains one unit with a linear activation function, which means
that
\begin{align*}
  V(x) = \sum_{i=1}^{10} A^{(4)}_{i} z^{(3)}_i + b^{(4)}.
\end{align*}

 For the N-model processing network  in Section  \ref{sec:nmodel} the  structure of policy and value NNs is the same as for the MQNs described above, except for the meaning of  set $\B(\ell)$ in (\ref{eq:softmax}). We consider   station $\ell\in \LL$ of a processing network. The set   $\B(\ell)$ includes a job class  if and only if  there is an activity  such that station $\ell$ can process jobs from this   class.

\section{Implementation details and experiment parameters}\label{sec:par}

 We use Tensorflow v1.13.1  \cite{Abadi2016} to build a training routine of the neural networks and Ray package v0.6.6 \cite{Moritz2018}
    to maintain parallel simulation of the actors. We run all experiments on a   2.7 GHz  96-core processor with 1510 GB of RAM.

    We optimize  the value and policy functions to minimize the corresponding loss functions (\ref{eq:Vappr}), (\ref{eq:popt}) by the Adaptive Moment Estimation (Adam) method \cite{Kingma2017}.
    The Adam method is an algorithm for mini-batch gradient-based optimization.
We assume that $N$ datapoints have been generated  $D^{(1:N)}= \Big\{  (x^{(j)}, a^{(j)}, \hat A_j)\Big\}_{j=1}^N$ to update the policy NN  or $D^{(1:N)} = \left\{  ( x^{(j)}, \hat V_j)\right\}_{j=1}^N$ to update the value NN.
 The Adam algorithm runs for $E$ epochs. The number of epochs is the number of complete passes through the entire dataset. In the beginning of a new epoch $e$ the entire dataset
is randomly reshuffled and divided into   batches with size $m$. Then each batch (indexed by $n$) is passed to the learning algorithm and the parameters
of the neural networks $\theta = (\theta^1, ..., \theta^i,..., \theta^K)$  are updated at the end of every such step  according to

     \begin{align*}
     \theta_{n+1}^i = \theta_n^i - \varsigma \frac{1}{\sqrt{\hat H^i_n} +\varrho} \hat G^i_n,
     \end{align*}
     where $\varsigma$ is a learning rate, $\hat G_n$ and $\hat H_n$ are moving average estimates of the first and second moments of the gradient, respectively, and $\varrho<<1$ is a constant.

     We use the batches to compute the gradient of a loss function $\hat L(\theta, D)$, such as for (\ref{eq:popt}):
    \begin{align}\label{grad}
    g_n = \frac{1}{m} \nabla_\theta   L\left(\theta,  D_e^{ (nm:nm+m)}\right),
    \end{align}
where $D_e^{ (nm:nm+m)}$ denotes the data segment in the $n$th batch of size $m$ at epoch $e$.

 The moving averages $G_{0}$ and $H_0$ are initialized as vectors of zeros at the first epoch.   Then    the Adam method updates the moving average estimates   and includes bias corrections to   account for their initialization at the origin:

   \begin{align*}
   \begin{cases}
    G_{n} = \beta_1 G_{n-1} +(1-\beta_1) g_{n-1},\\
    \hat G_{n} = \frac{G_{n}}{1 - \beta_1^n},
    \end{cases}
    \end{align*}
where  $\beta_1>0$,  with $\beta_1^n$  denoting $\beta_1$ to the power $n$.

    Similarly, we compute the second moments by
    \begin{align*}
    \begin{cases}
    H_{n} = \beta_2 H_{n-1} +(1-\beta_2)  g_{n-1}^2,\\
    \hat H_{n} = \frac{H_{n}}{1 - \beta_2^n},
    \end{cases}
    \end{align*}
    where $g_n^2$ means the elementwise square, $\beta_2>0$   with $\beta_2^n$  denoting $\beta_2$ to the power $n$.

 Each subsequent epoch continues   the count over $n$  and keeps  updating the moving average estimates $G_n, H_n$  starting from their final values of the latest  epoch.

Table \ref{tab:par1} and Table \ref{tab:par2} list the PPO hyperparameters we choose for the experiments in Section \ref{sec:experiments}. Table 7 reports the estimates of the running time of Algorithm \ref{alg2}.

\begin{table}[H]
\centering%
\begin{tabular}{l|@{\quad}l}
  \hline
  % after \\: \hline or \cline{col1-col2} \cline{col3-col4} ...
  Parameter  & Value\\\hline
   Clipping parameter  $(\epsilon)$ & $0.2\times \max[ \alpha,0.01]$ \\
 No. of regenerative cycles per actor (N) & 5,000 \\
  No. of  actors $(Q)$  & 50 \\
  Adam parameters for policy NN& $\beta_1 = 0.9$, $\beta_2 = 0.999$, $\varrho = 10^{-8},$  $\varsigma=5\cdot 10^{-4}\times \max[ \alpha,0.05]  $   \\
  Adam parameters for policy NN& $\beta_1 = 0.9$, $\beta_2 = 0.999$, $\varrho = 10^{-8},$  $\varsigma=2.5\cdot 10^{-4}  $ \\
  No. of epochs (E) & 3\\
  Minibatch size in Adam method (m) & 2048\\
\end{tabular}
\caption[]{PPO hyperparameters used in Algorithms \ref{alg1} and \ref{alg1amp} for the experiments in Section \ref{sec:cc}.  Parameter $\alpha$ decreases linearly  from $1$ to $0$ over the course of learning: $\alpha =(I - i)/I$ on the $i$th policy iteration, $i=0, 1,...,I-1. $  }\label{tab:par1}
\end{table}

\begin{table}[H]
\centering%
\begin{tabular}{l|@{\quad}l}
  \hline
  % after \\: \hline or \cline{col1-col2} \cline{col3-col4} ...
  Parameter  & Value\\\hline
   Clipping parameter  $(\epsilon)$ & $0.2\times \max[ \alpha,0.01]$ \\
  Horizon $(N)$ & 50,000 \\
  No. of  actors $(Q)$  & 50 \\
  Adam parameters for policy NN& $\beta_1 = 0.9$, $\beta_2 = 0.999$, $\varrho = 10^{-8},$  $\varsigma=5\cdot 10^{-4}\times \max[ \alpha,0.05]  $   \\
  Adam parameters for policy NN& $\beta_1 = 0.9$, $\beta_2 = 0.999$, $\varrho = 10^{-8},$  $\varsigma=2.5\cdot 10^{-4}  $ \\
  Discount factor $(\beta)$  & 0.998  \\
  GAE parameter $(\lambda)$  & 0.99 \\
  No. of epochs (E)& 3\\
  Minibatch size in Adam method (m)  & 2048\\
\end{tabular}
\caption[]{PPO hyperparameters used in  Algorithm \ref{alg2} for the experiments in Sections \ref{sec:ext}  and  \ref{sec:nmodel}. Parameter $\alpha$ decreases linearly  from $1$ to $0$ over the course of learning: $\alpha =(I - i)/I$ on the $i$th policy iteration, $i=0, 1,...,I-1. $}\label{tab:par2}
\end{table}

   \begin{table}[H]
\centering%
\begin{tabular}{|c|c|}
  \hline
  % after \\: \hline or \cline{col1-col2} \cline{col3-col4} ...
  Num. of classes $3L$  & Time (minutes)  \\\hline
  6 & 0.50 \\\hline
  9 & 0.73 \\\hline
  12  & 1.01\\\hline
  15  & 2.12 \\\hline
  18  & 4.31\\\hline
  21  & 7.61 \\
  \hline
\end{tabular}
\caption[]{Running time of \textit{one policy iteration} of Algorithm \ref{alg2}  for the extended six-class network in Figure \ref{fig1}.}\label{tab:rt}
\end{table}

In the Algorithm \ref{alg2} we use  finite length episodes to estimate the expectation of the loss function in line 10. For each episode we need  to specify an initial state.
 We propose sampling the initial states from the set of states generated during previous policy iterations.
  We consider the $i$th policy iteration of the algorithm. We need to choose initial states to simulate policy $\pi_i$.  Since policy $\pi_{i-1}$ has been simulated in the $(i-1)$th iteration of the algorithm,   we can    sample $Q$ states uniformly at random from the episodes generated under policy $\pi_{i-1}$ and  save them in memory. Then we use them as initial states for   $Q$ episodes under $\pi_i$ policy. For policy $\pi_0$ all $Q$ episodes start from state $x = (0,...,0).$

\bibliographystyle{plainurl}
\bibliography{PPO_paper,dai20200529}

\end{document}